\documentclass[article,11pt]{amsart}

\usepackage[percent]{overpic}
\usepackage{enumerate}
\usepackage{amsmath, amsthm, amsfonts,stmaryrd,amssymb}
\usepackage[top=1in, bottom=1.25in, left=1.25in, right=1.25in]{geometry}
\usepackage{mathrsfs}    
\usepackage{ dsfont }                  
\usepackage{mathtools}                   
\usepackage{booktabs}
\usepackage{tikz}
\usepackage{cancel}

\usepackage{soul} 

\usetikzlibrary{spy}
\usepackage{todonotes}
\usepackage[all]{xy}
\usepackage[utf8]{inputenc}
\usepackage{faktor}
\usepackage{ marvosym }
\usepackage{bbm}                         
\usepackage{framed}
\usepackage{graphicx}   
\usepackage{subcaption}
\usepackage{hyperref} 
\usepackage{xcolor}
\hypersetup{
    colorlinks=true,
    linkcolor=blue,  
    urlcolor=cyan,
    citecolor=green!70!black,
    pdftitle={},
    pdfauthor={},
}
\newif\ifcomments
\commentstrue

\definecolor{mycolor}{RGB}{20, 20, 122}

\numberwithin{equation}{section}

\newtheorem{theorem}{Theorem}[section]
\newtheorem{remark}[theorem]{Remark}

\newtheorem{lemma}[theorem]{Lemma}

\newtheorem{proposition}[theorem]{Proposition}

\newtheorem{definition}{Definition}

\newcommand{\bs}[1]{\boldsymbol{#1}}

\newcommand{\mc}[1]{\mathcal{#1}}
\newcommand{\eps}{\varepsilon}
\newcommand{\E}{\mc E}
\newcommand{\meas}{\operatorname{meas}}
\newcommand{\size}{\operatorname{size}}

\def\x{\times}
\def\p{\partial}

\def\R{{\mathbb R}}

\def\N{{\mathbb N}}

\def\Sig{\Sigma}
\def\Tt{\mc{T}}

\newcommand{\brho}{\bs\rho}
\newcommand{\btheta}{\bs\theta}

\def\ed{\mathrm{d}}

\def\exp{\operatorname{exp}}

\def\Div{\operatorname{div}}

\let\mc=\mathcal

\def\dmin{d^{\min}_\Sigma}
\def\dmink{d^{\min}_{\Sigma^k}}

\usepackage{accents}

\begin{document}
\title[Discretizing the Fokker-Planck equation with second-order accuracy]{Discretizing the Fokker-Planck equation with second-order accuracy: a dissipation driven approach}
\author[C. Cancès]{Cl\'ement Canc\`es}
\address{Inria, Univ. Lille, CNRS, UMR 8524 - Laboratoire Paul Painlev\'e, F-59000 Lille, France.}
\email{clement.cances@inria.fr, andrea.natale@inria.fr}
\author[L. Monsaingeon]{Léonard Monsaingeon}
\address{Group of Mathematical Physics, Departamento de Matem\'atica, Instituto Superior T\'ecnico,
Av. Rovisco Pais, 1049-001 Lisboa, Portugal
}
\address{
IECL, Facult\'e des Sciences et technologie, Campus, Boulevard des Aiguillettes, BP 70239, 54506 Vandœuvre-l\`es-Nancy, France.
}
\email{leonard.monsaingeon@ciencias.ulisboa.pt}
\author[A. Natale]{Andrea Natale}

\begin{abstract}
We propose a fully discrete finite volume scheme for the standard Fokker-Planck equation.
The space discretization relies on the well-known square-root approximation, which falls into the framework of two-point flux approximations.
Our time discretization is novel and relies on a tailored nonlinear mid-point rule, designed to accurately capture the dissipative structure of the model.
We establish well-posedness for the scheme, positivity of the solutions, as well as a fully discrete energy-dissipation inequality mimicking the continuous one.
We then prove the rigorous convergence of the scheme under mildly restrictive conditions on the unstructured grids, which can be easily satisfied in practice.
Numerical simulations show that our scheme is second order accurate both in time and space, and that one can solve the discrete nonlinear systems arising at each time step using Newton's method with low computational cost.
\end{abstract}

\maketitle

\section{Introduction}

\subsection{Fokker-Planck equation and Wasserstein gradient flows}
Because of their broad interest in physics~\cite{AS08, CMN19,KMX17,PP10}, biology~\cite{Bla14_KS,casteras2023hidden, FBC23} or social sciences~\cite{galichon2018optimal,MRS10}, Wasserstein gradient flows have been the object of strong interest by the mathematical community in the last decades.
A prototypical example of such Wasserstein gradient flows is the Fokker-Planck equation
\begin{subequations}\label{eq:FP}
\begin{align}\label{eq:FP.cons}
\partial_t \rho + \Div F& \; =0, \\
F  +  \rho \nabla V +\nabla \rho & \; = 0,
\label{eq:FP.F}
\end{align}
set in a space time domain $Q_T = (0,T) \times \Omega$, 
where $\Omega$ is a convex and bounded open subset of $\R^d$ that we further assume to be polyhedral for meshing purposes, and where $T$ is an arbitrary finite time horizon.
The background potential $V \in C^2(\overline \Omega)$ is always assumed to be given and smooth.
The kinematics is complemented by a nontrivial initial condition $\rho^0$ and no-flux boundary conditions
\begin{equation}\label{eq:FP.no-flux}
F \cdot n_{\p\Omega}  = 0  \quad \text{on}\; (0,T) \times \partial \Omega.
\end{equation}
We always assume that
\begin{equation}\label{eq:FP.rho0}
\rho(0,\cdot) = \rho^0 \geq 0
\qquad \text{with} \qquad
\mc H(\rho^0) <+\infty
\quad\text{and}\quad
\int_\Omega \rho^0 >0.
\end{equation}
Here the (negative) entropy $\mc H(\rho)$ and free energy $\E (\rho)$ are defined as
\[
\mc H(\rho)
\coloneqq \int_\Omega (\rho \log \rho - \rho + 1),
\qquad \E (\rho)
\coloneqq \mc H(\rho) + \int_\Omega \rho V,
\]
and the associated stationary Gibbs measure reads
$$
\pi \coloneqq  e^{-V}.
$$
\end{subequations}

A formal multiplication of~\eqref{eq:FP.cons} by $\log \frac{\rho}\pi$ provides that 
\begin{equation}\label{eq:EDE.0}
\frac{\rm d}{{\rm d} t} \E (\rho) + \int_\Omega \rho \left| \nabla \log \frac{\rho}{\pi} \right|^2 = 0,
\end{equation}
meaning that the energy is dissipated along time at a precise rate. 
Since the seminal work of Otto~\cite{JKO98, Otto01} and the book of Ambrosio, Gigli and Savar\'e~\cite{AGS08}, the \emph{minimizing movement scheme} (also referred to as the JKO scheme in this setting) has been playing a central role for the analysis and numerical discretization of gradient flows.
It indeed enjoys strong stability properties as well as a variational structure, in the sense that it amounts to a minimization problem at each time step.
Being the limiting curve obtained by the convergence of a JKO scheme is even one of the possible characterizations of abstract metric gradient flows, see \cite{AGS08}.
However, although very attractive from a theoretical point of view, the JKO scheme suffers from several downsides when it comes to practical implementation.
First, it is merely first order accurate in time.
Second, the optimality condition of the problem to be solved at each time-step amounts to a continuous in time mean field game, that further needs to be discretized, either with inner time stepping~\cite{BCL16, CCWW_FoCM} or by linearizing the Wasserstein distance~\cite{CGT20, Li2020, NT_FVCA9}.
Moreover, full space discretization of the JKO scheme on fixed grids creates difficulties which do not arise in semi-discrete problems~\cite{Maas11, EPSS21, HST_arXiv, MRSS_arXiv}, and as a consequence the JKO time discretization is often replaced by the computationally cheaper Backward Euler scheme~\cite{CMNR_HAL}.
Lagrangian particle schemes \cite{leclerc2020lagrangian,gallouet2022convergence,natale2023gradient} and moving meshes \cite{MO14, CDMM18} on the other hand, are generally only first order accurate in space.

Our new approach here will circumvent these numerical bottlenecks and relies instead on another by-now classical characterization of gradient flows (still formal at this stage):
Any smooth density $\rho=\rho(t,x)$ solving the continuity equation~\eqref{eq:FP.cons} with no-flux boundary conditions satisfies
\begin{equation}\label{eq:EDI.reverse}
\frac{\rm d}{{\rm d} t} \E (\rho)
=  \int_\Omega \log \frac\rho\pi  \,\partial_t \rho
= \int_\Omega F \cdot \nabla \log \frac\rho\pi \geq 
- \frac12 \int_\Omega \frac{|F|^2}{\rho} -
 \frac12\int_\Omega \rho \left| \nabla \log \frac{\rho}{\pi} \right|^2.
\end{equation}
The key observation is that equality holds in Young's inequality above if and only if \eqref{eq:FP.F} is fulfilled.
Therefore if $\rho$ is such that \eqref{eq:FP.cons} holds and satisfies in addition the reverse \emph{Energy-Dissipation Inequality} (EDI):
\begin{equation}
\label{eq:EDI.formal}
\frac{\rm d}{{\rm d} t} \E (\rho)
\leq
- \frac12 \int_\Omega \frac{|F|^2}{\rho} -
\frac12\int_\Omega \rho \left| \nabla \log \frac{\rho}{\pi} \right|^2,
\end{equation}
then $\rho$ must also satisfy \eqref{eq:FP.F}.
The first term in the right-hand side dissipation only depends on the kinematics through the continuity equation \eqref{eq:FP.cons}, while the second part, known as the \emph{Fisher information functional}, is related to the specific choice of an energy through the first variation $\frac{\delta\E }{\delta\rho}=\log\frac{\rho}{\pi}$.
In order to make the EDI formulation rigorous, and following ideas of~\cite{BB00}, one introduces the one-homogeneous, convex, and lower semi-continuous Benamou-Brenier function $B: [0,+\infty) \times\R^d \to [0,+\infty]$ defined by
\begin{equation}\label{eq:BB}
B(a,b) \coloneqq \left\{
\begin{array}{ll}
\displaystyle \frac{|b|^2}{2a} & \text{ if } a>0\,,\\
0 & \text{ if } b = 0 \text{ and } a =0\,,\\
+\infty & \text{ otherwise}\,, 
\end{array}
\right.
\end{equation}
and observes that the Fisher information rewrites as a convex function of $\rho$ under the form
\[
\mc R(\rho)
\coloneqq
\frac12\int_\Omega \rho \left| \nabla \log \frac{\rho}{\pi} \right|^2
=
2 \int_\Omega \pi \left| \nabla \sqrt{\frac{\rho}{\pi}} \right|^2
.
\]
Then, integrating \eqref{eq:EDI.formal} over time leads to the following notion of EDI solution to~\eqref{eq:FP}, thoroughly developed in~\cite{AGS08}.
\begin{definition}\label{def:EDI}
A curve $\rho \in C([0,T]; L^1_w (\Omega))$ is an EDI solution to~\eqref{eq:FP} corresponding to the initial solution $\rho^0$ if, denoting by $\rho^T = \rho(T,\cdot)$, there holds
\begin{equation}\label{eq:EDI}
\E (\rho^T) +\int_0^T \mc R(\rho) + \inf_F \left\{\int_{Q_T} B(\rho, F)\right\} \leq \E (\rho^0),
\end{equation}
where the infimum is taken among vector fields $F\in L^1(Q_T;\mathbb{R}^d)$ satisfying the continuity equation $\partial_t\rho+\Div F=0$ with initial/terminal data $\rho^0,\rho^T$ and no-flux boundary conditions:
\begin{equation}
\label{eq:FP.cons.weak}
\int_{Q_T} \rho\, \partial_t \varphi + \int_{Q_T} F\cdot\nabla \varphi - \int_\Omega \rho^T \varphi(T,\cdot)  + \int_\Omega  \rho^0 \varphi(0,\cdot)
=0\,, \quad \forall \,\varphi \in C^1(\overline Q_T).
\end{equation}
\end{definition}
\noindent
In the previous definition $L^1_w (\Omega)$ denotes the space of integrable functions equipped with its weak topology.
It is not difficult to check \cite{AGS08} that densities $\rho \in L^1(Q_T)$ satisfying \eqref{eq:FP.cons.weak} with finite kinetic energy $\iint _{Q_T}B(\rho,F)$ are $L^1_w (\Omega)$-continuous in time, and satisfy $\rho(0) = \rho^0,\rho(T)=\rho^T$.

For the sake of self-completeness we collect basic properties of EDI solutions in Appendix~\ref{sec:appendix_EDI}.
Let us only mention at this stage that a) EDI solutions are unique, and b) they are automatically distributional solutions of the Fokker-Planck equation~\eqref{eq:FP}.

\subsection{Our contribution and organization of the paper}
Our goal here is to propose a fully discrete finite volume scheme based on a \emph{two-point flux approximation} (TPFA) which
satisfies a discrete counterpart of the EDI formulation \eqref{eq:EDI.formal} while being second order accurate in both space and time.

The space discretization we adopt relies on the well-established \emph{square-root approximation} (SQRA) scheme~\cite{LFW13, Heida18}.
As exploited in~\cite{CV23}, this space discretization enjoys a dissipative structure involving hyperbolic cosine dissipation potentials, cf. Section~\ref{ssec:dissip}, very much related to \cite{MPR14, PRST22, PS23}.
Our main contribution here concerns the time discretization, and our strategy consists in capturing the energy dissipation
\begin{equation}\label{eq:dissip.step}
\int_{n\tau}^{(n+1)\tau}  \mc R(\rho) + \int_{n\tau}^{(n+1)\tau} \int_\Omega B(\rho, F)
\end{equation}
over each time step in a sufficiently accurate way to be discussed shortly.
In~\cite{JST19,stefanelli2022new} this is achieved by recursively minimizing a discrete version of the full dissipation functional appearing in the left-hand side of the EDI \eqref{eq:EDI}, resulting in a variational -- but first order scheme.
In this work our approximation of the above quantity will rely instead on a mid-point rule in order to recover second order accuracy.
More precisely, given a density $\rho^n$ at time $t^n=n\tau$, we first introduce a density $\theta^{n+1/2}$ at the intermediate time $t^{n+1/2}=(n+1/2) \tau$.
The next time step will be obtained as a particular extrapolation
$$
\rho^{n+1}=\Xi(\rho^n,\theta^{n+1/2})
$$
to be detailed in Section~\ref{ssec:scheme} below.
Defining the Fokker-Planck flux at time $t^{n+1/2}$
$$
F^{n+1/2}
=
\theta^{n+1/2}\nabla\log \frac{\theta^{n+1/2}}{\pi}
=
\pi\nabla\frac{\theta^{n+1/2}}{\pi}
=\nabla\theta^{n+1/2}+\theta^{n+1/2}\nabla V
$$
(or rather, its discrete SQRA counterpart \eqref{eq:fluxes.linear} later on), we further impose the discrete continuity equation
\[
\frac{\rho^{n+1} - \rho^n}\tau + \Div\, F^{n+1/2} = 0,
\]
and we finally approximate the dissipation \eqref{eq:dissip.step} by
\[
\tau \left\{\int_\Omega B(\theta^{n+1/2}, F^{n+1/2}) + \mc R(\theta^{n+1/2}) \right\}.
\]
In addition to the desirable preservation of mass and positivity, our specific choice of extrapolation combined with the variational structure of the SQRA flux  will entail a discrete (upper) chain rule for the energy density (see Lemma~\ref{lem:upper_chain_rule} below), which in turn will crucially result in a fully discrete EDI inequality.
Passing next to the limit in an appropriate sense, we will establish full convergence of the scheme towards a dissipative EDI solution.
In the time-continuous setting, this idea of proving convergence by passing to the limit in a semi-discrete EDI was already implemented in \cite{forkert2022evolutionary} for a similar finite volume discretization of the Fokker-Planck.
This essentially boils down to proving asymptotic lower bounds on the two dissipation functionals involved in a continuous-in-time discrete-in-space EDI, which in our setting will be given by Propositions \ref{prop:fisher} and \ref{prop:bb} below.
Nonetheless, the possibility of vacuum (vanishing of the densities $\rho,\theta$) makes the analysis more delicate in our setting compared to \cite{forkert2022evolutionary}, and the methods of proof are different.

\begin{remark}
Exploiting the estimates and the resulting compactness properties established later on, one could directly prove that the solutions produced in the limit by our numerical scheme are distributional solutions of the Fokker-Planck equation, following for instance the methodology proposed in~\cite{CG_VAGNL, CV23} and even for non-convex domains $\Omega$.
At the price of the convexity assumption on $\Omega$, the recovery of~\eqref{eq:FP} in the distributional sense is for free (cf. Proposition~\ref{prop:props_EDI} in the Appendix), together with uniqueness of the EDI solutions~\cite{gigli2010heat}.
\end{remark}

Finally, let us stress that our time extrapolation is local and only requires pointwise evaluation ($\rho^{n+1}_K=\Xi(\rho^n_K,\theta^{n+1/2}_K)$ for each cell $K$ in the finite volume discretization), in contrast to the nonlocal one introduced in~\cite{GNT24}.
Similarly, the mid-point rule proposed in~\cite{LT17} relies on the computationally expensive evaluation of intermediate Wasserstein geodesics, which we completely dispense from.
Our approach also shares some features with \cite{LWWYZ_arXiv}, but with different relations between $\theta^{n+1/2}$, $\rho^n$ and $\rho^{n+1}$.
The specific choices we make for $\Theta$ and $\Xi$ in this paper allow us to rigorously prove the convergence or our scheme, beyond the partial consistency and stability results provided in \cite{LWWYZ_arXiv}. 

The paper is organized as follows:
The scheme is introduced in Section~\ref{sec:scheme}.
After introducing usual concepts related to TPFA finite volumes in  Section~\ref{ssec:mesh}, the space and time discretization are presented in Section~\ref{ssec:scheme}, where the extrapolation is constructed.
Then elements of numerical analysis at fixed grid are presented in Section~\ref{sec:fixedgrid}.
This encompasses the well-posed character of the scheme in Section~\ref{ssec:wellposedness} as well as the fully discrete EDI in Section~\ref{ssec:dissip}.
The latter plays a crucial role in Section~\ref{sec:convergence}, where the convergence of the scheme towards an EDI solution is established under some restriction on the mesh detailed in Section~\ref{ssec:mesh2}.
Compactness properties on the approximate reconstructions are then derived in Section~\ref{ssec:compact} and refined in \ref{ssec:compact.strong} thanks to some discrete Aubin-Lions-Simon argument.
We pass to the limit and establish two separate Gamma-liminf's for the dissipation functionals in Sections~\ref{ssec:conv.Fisher} and \ref{ssec:conv.BB}, and the full convergence is then detailed in Section~\ref{ssec:conv.final}.
Numerical results are then presented in Section~\ref{sec:num}, showing that our scheme is second order accurate in time and space.
Finally, we defer two technical parts to the appendix: Appendix~\ref{sec:appendix_EDI} recalls basic properties of EDI solutions, while Appendix~\ref{sec:appendix_Moussa} contains an extension adapted to our needs of an Aubin-Lions-Simon lemma by Moussa~\cite{moussa2016some}.

\section{A space-time discretization for the Fokker-Planck equation}\label{sec:scheme}

\subsection{Finite volume discretization}\label{ssec:mesh}
The space discretization of our scheme falls into the framework of TPFA finite volumes. It requires the definition of an admissible mesh of the domain $\Omega \subset\mathbb{R}^d$, which is assumed to be polyhedral with Lebesgue measure $m_\Omega>0$.

An admissible mesh of $\Omega$ is a triplet $(\Tt,\overline{\Sigma},(x_K)_{K\in\Tt})$, consisting in a set of cells $K\in\Tt$, facets $\sigma \in \overline{\Sigma}$, and cell centers $x_K\in \Omega$, satisfying in addition the conditions in \cite[Definition 9.1]{eymard2000finite}.
Specifically, we require the following:
 \begin{enumerate}[1)]
 \item
 The cells $K\in\Tt$ are open disjoint polyhedra with positive $d$-dimensional Lebesgue measure $m_K>0$.
 They form a tessellation of $\Omega$, i.e.\
 \[\bigcup_{K\in\Tt} \overline{K} = \overline{\Omega}\quad \text{and} \quad \sum_{K\in\Tt} m_K = m_\Omega.\]
 \item
 The facets $\sigma \in \overline{\Sigma}$ are closed subsets of $\overline{\Omega}$ contained in an hyperplane of $\mathbb{R}^d$, and with strictly positive $(d-1)$-dimensional Hausdorff (or Lebesgue)
measure denoted by $m_\sigma>0$.
Every facet $\sigma \in \overline{\Sigma}$ satisfies either $\sigma = K |L \coloneqq \partial K \cup \partial L$ or $\sigma= \partial K \cup \partial \Omega$, for some $K,L\in \Tt$ with $K \neq L$.
The subset of interior facets  $\Sigma\subset \overline{\Sigma}$ is the set of facets $\sigma$ for which there exists $K,L \in \Tt$ such that  $\sigma = K |L $.
\item
For any cell $K\in \Tt$, there exists a subset $\overline{\Sigma}_K \subset \overline{\Sigma}$ such that
\[
\partial K =  \bigcup_{\sigma \in \overline{\Sigma}_K} \sigma\quad \text{and} \quad \overline{\Sigma} =  \bigcup_{K \in {\Tt}} \overline{\Sigma}_K\,.
\]
We denote the interior facets associated with a cell $K$ by  $\Sigma_K = \overline{\Sigma}_K \cap \Sigma$.
\item
Two cell centers $x_K$ and $x_L$ coincide if and only if $K=L$. Moreover, if $\sigma = K|L$ then $x_K-x_L$ is orthogonal to $\sigma$, and  denoting $
d_\sigma \coloneqq |x_K - x_L|$, the outward normal to the cell $K$ on the facet $\sigma \in \Sigma_K$ is given by
\begin{equation}\label{eq:ortho}
\quad n_{K\sigma} = \frac{x_L - x_K}{d_\sigma}\,.
\end{equation}
\end{enumerate}
Discrete densities are represented by collections of degrees of freedom $\brho = (\rho_K)_{K\in\Tt} \in \mathbb{R}^{\Tt}_+$, where $\rho_K$ is the degree of freedom associated with the cell $K\in\Tt$.
Similarly, fluxes are represented by the collection of outward fluxes through the inner facets, and denoted as follows: $\bs{F}= ((F_{K\sigma}, F_{L\sigma}))_{\sigma = K|L \in\Sigma}\subset\mathbb{R}^{2\Sigma}$.
We also define the space of conservative fluxes as follows
\[
\mathbb{F}_{\Sigma} \coloneqq \{ \bs{F}\in \mathbb{R}^{2\Sigma}~;~ F_{K\sigma}+F_{L\sigma} =0 \quad \forall\, \sigma = K|L \in \Sigma\}\,.
\]
For any $\bs{F} \in \mathbb{F}_{\Sigma}$ we denote $F_\sigma \coloneqq |F_{K\sigma}| = |F_{L\sigma}|$.

We discretize a fixed time interval $[0,T]$ in $N \in \mathbb{N}^*$ time steps of size $\tau = T/N$.
A discrete time-dependent density is described by a collection $(\brho^n)_{n=0}^N$, where $\brho^n$ is the discrete density associated with the time $t^n = n \tau$.
Discrete time-dependent fluxes are staggered in time with respect to the densities and they are therefore described by $(\bs{F}^{n+1/2})_{n=0}^{N-1}$, with $\bs{F}^{n+1/2}$ representing now the  discrete fluxes at time $t^{n+1/2} = (n+1/2) \tau$.
\subsection{Numerical scheme}\label{ssec:scheme}
As already mentioned, the dissipation properties of our scheme will only be guaranteed by the correct choice of an extrapolation $\rho^{n+1}=\Xi(\rho^n,\theta^{n+1/2})$.
In order to construct $\Xi$ we first define a specific nonlinear mean $\Theta:\R_+ \times \R_+ \rightarrow \R_+$
\begin{equation}\label{eq:theta1}
\Theta(a,b) \coloneqq {H^*}'\left( \frac{H(b)-H(a)}{b-a} \right) = \exp\left( \frac{b\log b-a \log a}{b-a}- 1\right)\,, \quad \forall\, a,b >0 \,, ~ a\neq b,
\end{equation}
where the entropy function
$$
H(a) \coloneqq \left\{ 
\begin{array}{ll}
a\log a -a +1 & \text{if }a > 0\\
1           & \text{if } a=0 \\
+ \infty & \text{if}\; a <0
\end{array}
\right.,
$$
has explicit Legendre-Fenchel transform $H^*(p) = \exp(p)-1$. We naturally extend by continuity
\begin{equation}\label{eq:theta2}
\Theta(a,a) \coloneqq a\,, \quad \Theta(a,0) = \Theta(0,a) \coloneqq e^{-1} a\,, \quad \forall\,a\geq 0.
\end{equation}
Note that by usual properties of convex duality there holds $(H^*)'(p)=(H')^{-1}(p)$ for all $p \in \R$. This fact together with \eqref{eq:theta1}--\eqref{eq:theta2} implies that
\begin{equation}
\label{eq:chain_rule_H}
(b-a)H'(\Theta(a,b))=H(b)-H(a) \,, \quad \forall\, a,b\geq 0\,, ~(a,b)\neq (0,0)\,,
\end{equation}
which will precisely entail the discrete chain rule.

At least formally, our extrapolation is simply given by inverting the mean, i.e.\ $\Xi(a,\cdot)=\Theta(a,\cdot)^{-1}$.
However, as is clear from Figure~\ref{fig:Fx}, $\Theta(a,0)>0$ prevents any global invertibility and some extra care is needed in order to obtain a well-posed scheme.
To this end, one can check that $\Theta\in C(\R_+ \times \R_+; \R_+)$ defined by \eqref{eq:theta1} and \eqref{eq:theta2} is jointly concave in its arguments and 1-homogeneous.
In particular, defining the concave, increasing function $f \in C(\R_+;[e^{-1},\infty))$ as
\[
f(r) \coloneqq \Theta(1,r)=\left\{
\begin{array}{ll}\displaystyle
 \exp \left(\frac{r\log(r)-r+1}{r-1} \right) & \text{if } r>0\\
 e^{-1} & \text{if } r=0  
\end{array}\right.\,,
\quad 
\]
we have that for any $a,b>0$
\[
\Theta(a,b)
=
a f\left(\frac{b}{a}\right)
= b f\left(\frac{a}{b}\right) = \Theta(b,a)\,.
\]
Since $f$ is concave and increasing (see Figure \ref{fig:Fx}), its inverse $r=f^{-1}(s)$ is unambiguously defined at least on $[e^{-1},+\infty)$.
Extending this inverse to the whole real line $s\in\R$, our extrapolation $\Xi$ is finally defined as
\begin{equation}\label{eq:GH}
g(s) \coloneqq \left\{
\begin{array}{ll}
 f^{-1}(s) & \text{if } s> e^{-1}\\
 0 & \text{otherwise}
\end{array}\right.
\quad\text{and}
\quad 
\Xi(a,c) \coloneqq \left\{
\begin{array}{ll}
\displaystyle a g\left(\frac{c}{a}\right)& \text{if } a>0 \\
{e} c  & \text{if }a=0
\end{array},\quad c\in\R
\right.\,.
\end{equation}
Note that $f$ has vertical tangent at $r=0^+$, which implies that $g(\cdot)$ and $\Xi(a,\cdot)$ are $C^1$, convex functions for any fixed $a\geq 0$ as depicted in Figure~\ref{fig:Fx}.
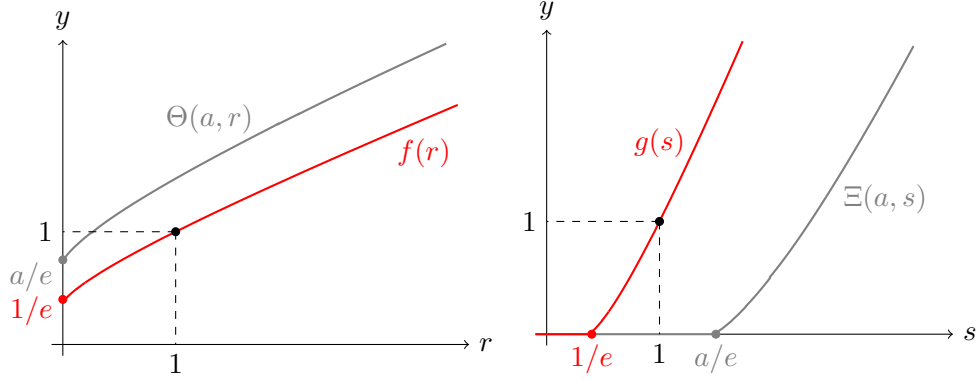
\begin{figure}
\centering
\begin{tikzpicture}[scale=1.5]
  \draw[->] (-.1,0) -- (3.6,0) node[right] {$r$};
  \draw[->] (0,-0.1) -- (0,2.7) node[above] {$y$};
  \draw[ thick, red, domain=.2:3.5,smooth,variable =\x] plot ({\x},{exp((ln{\x}*\x -\x+1)/(\x-1))});
  \draw[ thick, gray, domain=.001:.96,smooth,variable =\x,samples = 201] plot ({2*\x},{exp((ln{\x}*\x -\x+1)/(\x-1))*2});
  \draw[ thick, gray, domain= .95:1.05,smooth,variable =\x,samples = 201] plot ({2*\x},{(2*\x+2)/2});
  \draw[ thick, gray, domain= 1.04:1.7,smooth,variable =\x,samples = 201] plot ({2*\x},{exp((ln{\x}*\x -\x+1)/(\x-1))*2});
  \node [gray] at (1.3,2.) {$\Theta(a,r)$};
  \filldraw [gray] (0,.75) circle (1pt) ;
   \node[gray,left] at (0,.6) {$a/e$};
  \node[red] at (3.2,1.7) {$f(r)$};
  \draw[ thick, red, domain=.01:.2,smooth,variable =\x, samples = 101] plot ({\x},{exp((ln{\x}*\x -\x+1)/(\x-1))});
  \draw[ thick, red, domain=.001:.01,smooth,variable =\x, samples = 101] plot ({\x},{exp((ln{\x}*\x -\x+1)/(\x-1))});
  \filldraw (1,1) circle (1pt) ;
  \filldraw[red] (0,0.4) circle (1pt);
  \node[red,left] at (0,0.3) {$1/e$};
  \draw[dashed] (1,1) -- (1,0) node[below] {$1$};
  \draw[dashed] (1,1) -- (0,1) node[left] {$1$};
\end{tikzpicture}
\begin{tikzpicture}[scale=1.5]
  \draw[->] (1.47,0) -- (3.6,0) node[right] {$s$};
  \draw[->] (0,-0.1) -- (0,2.7) node[above] {$y$};
  \draw[ thick, red, domain=.2:2.6,smooth,variable =\x, samples = 101] plot ({exp((ln{\x}*\x -\x+1)/(\x-1))},{\x});
  \draw[ thick, red, domain=.01:.2,smooth,variable =\x, samples = 101] plot ({exp((ln{\x}*\x -\x+1)/(\x-1))},{\x});
  \draw[ thick, red, domain=.001:.01,smooth,variable =\x, samples = 101] plot ({exp((ln{\x}*\x -\x+1)/(\x-1))},{\x});
  \draw[ gray,thick, domain=-.1: 1.47,smooth,variable =\x, samples = 101] plot ({\x},0);
  \draw[ thick, red, domain=-.1:0.367,smooth,variable =\x, samples = 101] plot ({\x},0);
  \draw[ gray, thick, domain=.001:.64,smooth,variable =\x, samples = 401] plot ({4*exp((ln{\x}*\x -\x+1)/(\x-1))},{4*\x});
  \node[red] at (1,1.7) {$g(s)$};
  \node[gray] at (3,1.2) {$\Xi(a,s)$};
  \filldraw (1,1) circle (1pt) ;
  \filldraw [red] (.4,0) circle (1pt) ;
   \node[red,below] at (.4,0) {$1/e$};
   \filldraw [gray] (1.5,0) circle (1pt) ;
   \node[gray,below] at (1.5,0) {$a/e$};
  \draw[dashed] (1,1) -- (1,0) node[below] {$1$};
  \draw[dashed] (1,1) -- (0,1) node[left] {$1$};
\end{tikzpicture}
\\
\caption{ A graphical representation of the functions $f$, $g$, $\Theta(a,\cdot)$ and $\Xi(a,\cdot)$.
}\label{fig:Fx}
\end{figure}
For any fixed $a \geq 0$, $\Theta(a,\cdot)$ is an invertible map from $\R^+$ to  $[a e^{-1},\infty)$ and $\Xi(a, \cdot)$ coincides with its inverse when restricted on $[a e^{-1},\infty)$.
In other words, $c = \Theta(a,b)$ is a mean between $a$ and $b$, whereas $b= \Xi(a,c)$ is the corresponding extrapolation.
We stress again that, for any $a\geq 0$, $\Xi(a,\cdot) $ is a well-defined $C^1$, convex, non-decreasing function on the whole $\mathbb{R}$.
We have moreover
\[
\Xi(a, \Theta(a,b)) = b \,, \quad \Theta(a, \Xi(a,c)) = c\,
\qquad\forall\,a,b \geq 0, \,c\geq e^{-1} a,
\]
but this invertibility relation may fail if $c<e^{-1}a$.
This turns out to be quite delicate because our numerical scheme primarily solves for $\theta^{n+1/2}$, and then extrapolates $\rho^{n+1}=\Xi(\rho^n,\theta^{n+1/2})$:
If for some reason $0\leq \theta^{n+1/2}<e^{-1}\rho^n$, which does happen at least from our numerical experiments, then the invertibility relation fails and the chain rule \eqref{eq:chain_rule_H} does not hold as such.
Fortunately, and this is the whole cornerstone of our subsequent analysis, one still has an \emph{upper} chain-rule \eqref{eq:upper_chain_rule}.
For convenience we collect here useful properties of $\Theta,\Xi$.
\begin{lemma}
\label{lem:upper_chain_rule}
There holds
\begin{equation}
 \label{eq:upper_chain_rule}
 \forall\,a\geq 0, c>0:
 \qquad
 b=\Xi(a,c) \quad \Rightarrow \quad (b-a)H'(c)\geq H(b)-H(a)
\end{equation}
with equality if $c\geq e^{-1}a$, and moreover
\begin{equation}
\label{eq:bounds_Theta_Xi}
\forall \,a,b\geq 0,\,c\in\R:
\qquad
\Theta(a,b)\geq \min(a,b)
\quad\text{and}\quad
 \Xi (a,c)\geq 2c-a.
\end{equation}
\end{lemma}
\noindent
The possible failure of equality in \eqref{eq:upper_chain_rule} is precisely what makes our scheme non-variational 
in the sense that $\brho^{n+1}$ cannot be characterized as the minimizer of some functional (cf. Remark~\ref{rem:variational}). 
Whenever the scheme produces a value $\theta^{n+1/2}_K<e^{-1}\rho^n_K$ (which eventually happens at least in our simulations) an entropy release
\[
r_K^{n+1/2} = H(\rho^{n+1}_K) - H(\rho_K^n) - H'(\theta_K^{n+1/2}) (\rho_K^{n+1} - \rho_K^n) < 0
\] 
occurs in \eqref{eq:upper_chain_rule} for $\rho^{n+1}_K=\Xi(\rho^n_K,\theta^{n+1/2}_K)$, compared to the expected variational equality.
Note however that our scheme keeps some variational character as it amounts to a minimization problem in $\btheta^{n+1/2}/ \boldsymbol{\pi}$, cf. the proof of Proposition~\ref{prop:existence} below.

\begin{proof}
Let us begin with \eqref{eq:upper_chain_rule} and fix $a\geq 0$.
Since $c> 0$ one always has $b=\Xi(a,c)\geq 0$, including if $a=0$ (in which case $b=\Xi(0,c)=ec$).
If $b>0$ then by definition of $\Xi$ we have $c> e^{-1}a$, thus one can legitimately write $\Theta(a,b)=\Xi(a,\cdot)^{-1}(b)\Rightarrow c=\Theta(a,b)$ and from \eqref{eq:chain_rule_H} we see that equality holds in \eqref{eq:upper_chain_rule}. If $b=0$ and $a=0$, \eqref{eq:upper_chain_rule} is trivially safisfied.
If now $b=0$ and $a>0$ then, again by definition of $\Xi$, we see that necessarily $c\leq \Theta(a,0)$ and therefore by convexity $H'(c)\leq H'(\Theta(a,0))$.
Whence
$$
(b-a)H'(c)=(0-a)H'(c)
\geq (0-a)H'(\Theta(a,0))
=H(0)-H(a)
$$
as desired, where the last equality follows again from \eqref{eq:chain_rule_H}.

As for \eqref{eq:bounds_Theta_Xi}, consider first the case $a\leq b$.
Then $\Theta(a,b)=a f(b/a)\geq af(1)=a$, and thus by symmetry $\Theta(a,b)\geq \min (a,b)$.
The second inequality in \eqref{eq:bounds_Theta_Xi} follows by convexity: for $a\geq 0$ simply write $\Xi (a,c)=ag(c/a)\geq a[g(1)+g'(1)(c/a-1)]=a[1+2(c/a-1)]=2c-a$, and the proof is complete.
\end{proof}

We are now in position of defining the scheme.
Let $V\in C^2(\overline \Omega)$ be a given potential and $\rho^0 \in L^1(\Omega;\R_+)$ a nonnegative density with finite entropy and positive total mass
\begin{equation*}
\mc{H}(\rho^0) = \int_\Omega H(\rho^0) < \infty\,,
\qquad
\rho^0[\Omega] \coloneqq \int_\Omega \rho^0 >0\,,
\end{equation*}
where as before $H(a) = a \log a -a +1$.
Denote by $\bs{V} \in \mathbb{R}^\Tt$, $\bs{\pi} \in \mathbb{R}^\Tt$ and $\brho^0 \in \mathbb{R}^\Tt$ the discrete functions defined by
\begin{equation}
\label{eq:Vpirho0}
V_K \coloneqq V(x_K) \,,
\qquad
\pi_K \coloneqq \exp(-V_K)\,,
\qquad
\rho^0_K \coloneqq\frac{1}{m_K} \int_K \rho^0,
\qquad \forall\,K \in \Tt.
\end{equation}
A discrete solution is a pair of discrete curves $(\brho^n)_{n=0}^N$ and $(\btheta^{n+1/2})_{n=0}^{N-1}$ satisfying for all $n=0,\ldots, N-1$,
\begin{equation}\label{eq:discretePDE}
m_K \frac{\rho^{n+1}_K - \rho^n_K}{\tau} + \sum_{\sigma \in\Sigma_K} m_\sigma F_{K\sigma}^{n+1/2} = 0, \qquad K \in \Tt,
\end{equation}
where $\bs{F}^{n+1/2} \in \mathbb{F}_\Sigma$ is the square-root approximation (SQRA) finite volume flux \cite{LFW13, Heida18}
\begin{equation}
\label{eq:fluxes.linear}
F_{K\sigma}^{n+1/2}  = \frac{1}{d_\sigma}  \pi_\sigma \left(\frac{\theta^{n+1/2}_K}{\pi_K} - \frac{\theta^{n+1/2}_L}{\pi_L}\right)
\quad\text{with}\quad
\pi_\sigma = \sqrt{\pi_K\pi_L},
\qquad K \in \Tt, \, \sigma\in \Sigma_K,
\end{equation}
constructed on the intermediate densities $\btheta^{n+1/2} = ( \theta_K^{n+1/2})_{K\in\Tt}$ at time $t^{n+1/2} = t^n + \tau/2$.
To complete the scheme, the discrete density  $\brho^{n+1}$ at time $t^{n+1}$ is defined from $\brho^n$ and $\btheta^{n+1/2}$ by extrapolation:
\begin{equation}\label{eq:rhoG}
\rho^{n+1}_K = \Xi(\rho^n_K,\theta^{n+1/2}_K), \qquad K \in \Tt.
\end{equation}
We stress again that this can be considered as a problem in the single primary variable $\btheta^{n+1/2}$, from which $\bs F^{n+1/2},\bs\rho^{n+1}$ can be explicitly obtained whenever needed.
\begin{remark}\label{rem:CN}
Our scheme can be thought of as an extension of the usual Crank-Nicolson scheme, which corresponds to the linear time-extrapolation
\[
\rho_K^{n+1} = \Xi_\text{CN}(\rho_K^n, \theta_{K}^{n+1/2}) = 2  \theta_{K}^{n+1/2} - \rho_K^n.
\]
This scheme is known to be second-order accurate in time and energy stable for quadratic energies.
However, it is neither positivity preserving nor entropy-stable for Boltzmann type energies, and its extension to our entropic framework thus requires the introduction of the nonlinear extrapolation \eqref{eq:GH}. 

Let us also mention that our approach shares similarities with the so-called discrete variational derivative method~\cite{FM11}, at least when the relation  $\btheta^{n+1/2} = \Theta(\brho^{n}, \brho^{n+1})$ holds true (i.e.\ when equality holds in~\eqref{eq:upper_chain_rule}). However our choice to use $\btheta^{n+1/2}$ as an unknown and then to extrapolate to reconstruct $\brho^{n+1}$ allows us to deal with the degenerate geometry stemming from optimal transportation and to incorporate the positivity constraint in the scheme, while allowing the entropy release leading to the inequality in \eqref{eq:upper_chain_rule}. This is a cornerstone to carry out the rigorous convergence analysis presented in this paper. The choice of keeping $\btheta^{n+1/2}$ as the main unknown is also key in the implementation strategy, which shows great robustness despite the singularly nonlinear character of the scheme. 
\end{remark}


\section{Discrete well-posedness and dissipative structure}\label{sec:fixedgrid}
In this section we prove the main properties of the scheme \eqref{eq:discretePDE}--\eqref{eq:rhoG}.
We establish existence and uniqueness of solutions, as well as a discrete version of the energy dissipation inequality which will be crucial for the convergence analysis carried out in Section~\ref{sec:convergence}.


\subsection{Existence and uniqueness of solutions}\label{ssec:wellposedness}
We first establish one-step well-posedness of the scheme, and therefore global existence and uniqueness of the whole discrete curve by immediate recursion.

\begin{proposition}
\label{prop:existence}
For any $\brho^n\geq 0$ with $\sum_K  m_K\rho^n_K>0$ there exists a unique $\btheta^{n+1/2}$ and $\brho^{n+1}$ verifying \eqref{eq:discretePDE}--\eqref{eq:rhoG}.
Moreover there holds
\[
\theta^{n+1/2}_K>0,\qquad \rho^{n+1}_K \geq 0 ,
\qquad
\sum_{K\in\Tt} m_K \rho^{n+1}_K = \sum_{K\in\Tt}  m_K\rho^{n}_K,
\]
and
\begin{equation}
\label{eq:lower_upper_theta}
 \frac{\rho^{n+1}_K +\rho^n_K}{2}
 \geq
 \theta^{n+1/2}_K
 \geq
 \min(\rho^{n+1}_K,\rho^n_K)\,
\end{equation}
for all $K\in \Tt$.
\end{proposition}
Note in particular that our scheme is positivity and mass preserving.
\begin{proof}
Recall that on can view \eqref{eq:discretePDE}--\eqref{eq:rhoG} as a single equation for $\btheta^{n+1/2}$.
Changing variables $s_K = \theta^{n+1/2}_K/\pi_K$ for all $K\in\Tt$, it is easy to see that the former problem is equivalent to finding a critical point of 
\begin{equation}\label{eq:mins}
\mc J(\bs s)
\coloneqq
\frac 1\tau \sum_{K\in\Tt} m_K J_K(s_K)
+\,\sum_{\sigma \in \Sigma_\text{int}} \frac{m_\sigma \pi_\sigma }{2 d_\sigma} | s_K - s_L|^2,
 \qquad \bs s\in\R^\Tt,
\end{equation}
where $J_K(\cdot)$ is any primitive of the function $s \mapsto \Xi(\rho^{n}_K, s \pi_K) - \rho^n_K$.
Note that $\mc J$ is $C^1$ and convex.
Hence critical points are necessarily global minima, and by compactness at least one minimum exists.
(By definition of $\Xi$ it is not difficult to check that all the $J_K$'s are coercive, regardless of the particular value of $\rho^n_K$)

Let $\bs s$ be any minimizer and write $\btheta^{n+1/2},\bs F^{n+1/2},\bs \rho^{n+1}$ for the corresponding auxiliary variables.
Summing~\eqref{eq:discretePDE} over $K\in\Tt$ immediately guarantees mass conservation
\begin{equation}\label{eq:mass}
\sum_{K\in\Tt} m_K \rho_K^{n+1} = \sum_{K\in\Tt} m_K \rho_K^{n} >0\,.
\end{equation}
This implies that, for any minimizer $\bs{s}$, there exists a $K\in \Tt$ such that ${\theta}_K^{n+1/2} =  s_K \pi_K>e^{-1}\rho^n_K\geq 0$.
For if not, then $\rho^{n+1}_K\leq 0$ for all $K\in \Tt$ by definition \eqref{eq:GH} of $\Xi$, which in turn would contradict \eqref{eq:mass}.
As a consequence for any minimizer $\bs{s}$ at least one of the $J_K$'s is strictly convex in a neighborhood of $s_K$.
This improved convexity in at least one direction suffices to compensate for the lack of strict convexity of the discrete Dirichlet energy in \eqref{eq:mins}, and $\mc J$ is thus strictly convex in the neighborhood of $\bs s$.
Since $\mc J$ is also globally convex, this proves existence and uniqueness of $\bs \theta^{n+1/2}$ as claimed.

In order to show that $\btheta^{n+1/2}\geq 0$, set $K^* = \operatorname{argmin}\limits_K \theta_K^{n+1/2}/\pi_K$.
Then by \eqref{eq:discretePDE}--\eqref{eq:fluxes.linear}
\begin{equation}
\label{eq:min_K*}
\Xi(\rho^n_{K^*}, \theta_{K^*}^{n+1/2})
=\rho^{n+1}_{K^*}
=\rho^n_{K^*} + \frac{\tau}{m_{K^*}}\sum_{\sigma \in\Sigma_K}\frac{ m_\sigma}{d_\sigma}  \pi_\sigma \left(\frac{\theta^{n+1/2}_L}{\pi_L} - \frac{\theta^{n+1/2}_{K^*}}{\pi_{K^*}}\right)
\geq \rho^n_{K^*}
\end{equation}
From this we see that if $\rho^n_{K^*}=0$ then by \eqref{eq:GH} $e\theta_{K^*}^{n+1/2}=\Xi(0,\theta_{K^*}^{n+1/2})\geq 0$.
If now $\rho^n_{K^*}>0$ and $\theta_{K^*}^{n+1/2}\leq 0$ then, again by definition of $\Xi$, we would have that $\Xi(\rho^n_{K^*},\theta_{K^*}^{n+1/2})= 0$ \eqref{eq:min_K*}, and this would contradict \eqref{eq:min_K*} since $\rho^n_{K^*}>0$.
Whence $\theta_{K^*}^{n+1/2} \geq 0$, and therefore $\theta_K^{n+1/2} \geq 0$ for all $K\in \Tt$.

Let us now implement a strong maximum principle-type argument in order to improve this nonnegativity to strict positivity.
Assuming by contradiction that $\theta^{n+1/2}_{K^*}=0$, we see that $\rho^{n+1}_{K^*}=\Xi(\rho^n_{K^*},0) =0$, and evaluating \eqref{eq:discretePDE} for $K=K^*$ yields
\[
0\geq
m_K\frac{0-\rho^{n}_{K^*}}{\tau}
=
\sum_{\sigma \in \Sigma_{K^*}} \frac{m_\sigma \pi_\sigma }{d_\sigma } \left(\frac{\theta_L^{n+1/2}}{\pi_L}-0\right) \geq 0.
\]
This would imply $\theta^{n+1/2}_L =0$ for all $L\in\Tt$ sharing a facet with $K$, thus also $\rho^{n+1}_L=0$ since $\Xi(\rho^n_L,0)=0$ always.
Propagating from neighboring cell to neighboring cell we would conclude that $\bs\rho^{n+1}\equiv 0$, which would in turn contradict the mass conservation \eqref{eq:mass}.
Hence, we have shown that $\theta^{n+1/2}_K>0$ for all $K\in\Tt$, and as a consequence $\rho^{n+1}_K = H(\rho^n_K,\theta^{n+1}_K)\geq 0$, yet again by definition \eqref{eq:GH} of $\Xi$.

Let us finally establish the bounds \eqref{eq:lower_upper_theta} for $\theta^{n+1/2}$.
Since $\theta^{n+1/2}_K>0$, clearly the lower bound $\theta^{n+1/2}_K\geq \min(\rho^{n+1}_K,\rho^n_K)$ only needs to be checked when both $\rho^{n+1}_K,\rho^n_K>0$.
However, in this case we are necessarily in the ``invertibility regime'' $\theta^{n+1/2}_K = \Theta(\rho^{n+1}_K,\rho^n_K)$, and the claim immediately follows from the first bound in \eqref{eq:bounds_Theta_Xi}.
The second bound in \eqref{eq:bounds_Theta_Xi} also gives
$\frac 12[\rho^n_K + \rho^{n+1}_K)]
=\frac 12[\rho^n_K + \Xi(\rho^{n}_K,\theta^{n+1/2}_K)]
\geq \frac 12[\rho^n_K + (2\theta^{n+1/2}_K-\rho^n_K)]=\theta^{n+1/2}_K
$
and the proof is complete.
\end{proof}

\subsection{Discrete energy dissipation equality}\label{ssec:dissip}
For any nonnegative discrete density $\brho\geq 0$ we define the discrete total energy of the system
\[
\E _{\Tt}(\brho)
\coloneqq \mc{H}_{\Tt}(\brho) + \sum_{K\in \Tt} m_K V_K \rho_K\,, \quad \text{where} \quad \mc{H}_{\Tt}(\brho) \coloneqq \sum_K m_K H(\rho_K)\,.
\]
In this section we show that the solutions of our scheme satisfy a fully discrete energy dissipation inequality with respect to the discrete energy $\mc{H}$.
We will strongly rely on the following convex real-valued conjugate functions
\[
\psi(z) = 2 z \, \mathrm{arcsinh}(z/2) - 2\sqrt{4+z^2} +4\,,\quad \psi^*(\xi) = 4(\cosh(\xi/2)-1)\,,
\]
which emerge naturally in (electro-)chemistry \cite{FC67, MPP17, CCMRV23},  large deviations of
jump processes~\cite{MPR14}, multi-scale limits of diffusion processes~\cite{LMPR17, FL21}, and more \cite{PS23}.
Note in particular that for any $a,b>0$ we have identity
\[
\sqrt{ab} (\psi^*)'(\log a - \log b) = a-b\,.
\]
This allows to recast the SQRA fluxes $\bs{F}^{n+1/2}$ in \eqref{eq:fluxes.linear} as
\begin{equation}\label{eq:fluxes}
F_{K\sigma}^{n+1/2}
=
\frac{\theta_\sigma^{n+1/2}}{d_\sigma} (\psi^*)'\left( \log\left(\frac{\theta^{n+1/2}_K}{\pi_K}\right)-\log\left(\frac{\theta^{n+1/2}_L}{\pi_L}\right) \right)
\end{equation}
with \[ \theta_\sigma^{n+1/2} \coloneqq \sqrt{\theta_K^{n+1/2} \theta_L^{n+1/2}}.\]
Next, observe from the critical upper-chain rule \eqref{eq:upper_chain_rule} and $\rho^{n+1}_K=\Xi(\rho^n_K,\theta^{n+1/2}_K)$ that
\begin{equation}
\label{eq:thetaineq}
H(\rho^{n+1}_K) - H(\rho^n_K)\leq \log(\theta^{n+1/2}_K) (\rho^{n+1}_K - \rho^n_K) ,
\qquad \forall K\in \Tt,n\geq 0.
\end{equation}
Adding $V_K(\rho^{n+1}_K - \rho^n_K)=-\log \pi_K (\rho^{n+1}_K - \rho^n_K)$ on both sides, multiplying by $m_K$, and denoting for convenience
\[
\phi^{n+1/2}_K \coloneqq \log\left(\frac{\theta^{n+1/2}_K}{\pi_K}\right),
\]
we find
\begin{equation}
\label{eq:E-E<D}
\begin{aligned}
\E _{\Tt}(\brho^{n+1}) -  \E _{\Tt}(\brho^n)
&
=
\sum_K m_K \left[H(\rho^{n+1}_K) - H(\rho^n_K)\right] - \sum_K m_K \log \pi_K\left[\rho^{n+1}_K - \rho^n_K\right]
\\
& \overset{\eqref{eq:thetaineq}}{\leq }\sum_K m_K \log\left(\frac{\theta^{n+1/2}_K}{\pi_K}\right) (\rho^{n+1}_K - \rho^n_K)
\\
& \overset{\eqref{eq:fluxes}}{=}- \tau \sum_K \sum_{\sigma\in \Sigma_K} \frac{m_\sigma}{d_\sigma} \phi^{n+1/2}_K \theta_\sigma^{n+1/2} (\psi^*)'\left(\phi^{n+1/2}_K - \phi^{n+1/2}_L\right)
\\
& = - \tau \sum_{\sigma\in \Sigma} \frac{m_\sigma\theta_\sigma^{n+1/2}}{d_\sigma} \left(\phi^{n+1/2}_K - \phi^{n+1/2}_L\right)  (\psi^*)'\left(\phi^{n+1/2}_K - \phi^{n+1/2}_L\right),
\end{aligned}
\end{equation}
where we used the fact that the function $(\psi^*)'$ is odd in the last equality.
Let us define, for all $\brho\in \mathbb{R}^\Tt_+$ and $\bs{F},\bs{\xi}  \in\mathbb{F}_{\Tt}$,
\[
\mc{D}_\psi(\brho,\bs{F}) \coloneqq \sum_{\sigma \in \Sigma} \frac{m_\sigma \rho_\sigma}{d_\sigma} \psi \left( \frac{d_\sigma F_\sigma}{\rho_\sigma}  \right)\,,
\qquad
\mc{D}_{\psi}^*(\brho,\bs{\xi}) \coloneqq \sum_{\sigma \in \Sigma} \frac{m_\sigma \rho_\sigma}{d_\sigma} \psi^* (d_\sigma \xi_\sigma) \,,
\]
where as before $\rho_\sigma = \sqrt{\rho_K \rho_L}$.
By definition $\mc{D}_{\psi}^*(\brho,\cdot)$ is nothing but the Legendre transform of $\mc{D}_\psi(\brho,\cdot)$ with respect to the pairing
\[
\langle \bs{\xi},\bs{F}\rangle_\Sigma = \sum_{\sigma \in \Sigma} m_\sigma d_\sigma \xi_{K\sigma}F_{K\sigma} \,.
\]
We also define
\begin{equation}
\label{eq:def_Rpsi}
\mathcal{R}_\psi(\brho) \coloneqq \mc{D}_{\psi}^*\left(\brho,-\nabla_\Sigma \bs{\phi} \right) \,,
\end{equation}
with $\bs{\phi}\in \mathbb{R}^\Tt$ and $\nabla_\Sigma \bs{\phi} \in \mathbb{F}_\Sigma$ given by
\[
\phi_K\coloneqq \log\left(\frac{\rho_K}{\pi_K}\right)\,,
\qquad
(\nabla_\Sigma \bs{\phi} )_{K\sigma} \coloneqq \frac{\phi_L-\phi_K}{d_\sigma}\,.
\]
With these definitions, the calculations above imply altogether:
\begin{proposition}\label{prop:ede}
Any discrete solution satisfies the one-step discrete EDI
\begin{equation}\label{eq:discreteedi}
\E _{\Tt}(\brho^{n+1})  +\tau  \mc{D}_\psi(\btheta^{n+1/2},\bs{F}^{n+1/2}) + \tau \mathcal{R}_\psi(\btheta^{n+1/2})
\leq
\E _{\Tt}(\brho^n),
\end{equation}
and equality holds if $\theta^{n+1/2}_K \geq e^{-1} \rho^n_K$ for all $K\in\Tt$.
\end{proposition}
\begin{proof}
Leveraging the expression \eqref{eq:fluxes} for the fluxes $F_{K\sigma}$, we obtain from \eqref{eq:E-E<D}
\begin{equation*}
\begin{aligned}
\E _{\Tt}(\brho^{n+1}) -  \E _{\Tt}(\brho^n)
& \leq- \tau \sum_{\sigma\in \Sigma} m_\sigma d_\sigma \frac{\phi^{n+1/2}_K - \phi^{n+1/2}_L}{d_\sigma} F_{K\sigma}^{n+1/2}
\\
& =
-\tau\langle -\nabla_\Sigma \bs{\phi}^{n+1/2},\bs{F}^{n+1/2}\rangle
\\
& =-   \tau \mc{D}_{\psi}\left(\btheta^{n+1/2},\bs{F}^{n+1/2} \right) - \tau \mc{D}_{\psi}^*\left(\btheta^{n+1/2},-\nabla_\Sigma \bs{\phi}^{n+1/2} \right)\,.
\end{aligned}
\end{equation*}
 In the last equality we simply used the equality case in the $\mc{D}_\psi,\mc{D}_\psi^*$ Fenchel duality, which stands owing to $F_{K\sigma}^{n+1/2}=\frac{\theta_\sigma^{n+1/2}}{d_\sigma} (\psi^*)'\left( -(\nabla_\Sigma \bs{\phi} )_{K\sigma} \right)$ in \eqref{eq:fluxes}.
\end{proof}
\begin{remark}\label{rem:variational}
By analogy with the continuous setting, and similarly to \cite{JST19,stefanelli2022new},
an alternative scheme could consist in defining recursively $\tilde{\bs{\rho}}^{n+1}$ as a solution to the following variational problem:
\begin{equation}\label{eq:variational}
\tilde{\brho}^{n+1} \in
\underset{{\brho\geq \bs{0}}}{\operatorname{argmin}} \inf_{\bs{F}}\{ \mc{E}_{\Tt}(\brho)  +\tau  \mc{D}_\psi(\btheta,\bs{F}) + \tau \mathcal{R}_\psi(\btheta) \},
\end{equation}
in which the continuity equation $
m_K \frac{\rho_K - \tilde{\rho}^n_K}{\tau} + \sum_{\sigma \in\Sigma_K} m_\sigma F_{K\sigma} = 0$ is imposed as a constraint and $\btheta = \Theta(\tilde{\brho}^n,\brho)$.
Note that this problem admits indeed minimizers since, by the same calculations as above, we always have that the function minimized in \eqref{eq:variational} is bounded from below by $\mc{E}_{\Tt}(\tilde{\brho}^n)$, and the set of admissible discrete densities is compact.
Note also that, discarding $\mathcal{R}_\psi(\btheta)$, one is left with a discretized version of the classical JKO scheme.

In general, the solution obtained via \eqref{eq:variational} is different from the solution $(\bs{\rho}^{n})_n$ obtained using our scheme.
In fact, due to \eqref{eq:rhoG}, we may have $\btheta^{n+1/2} \neq \Theta(\brho^n,\brho^{n+1})$ if $\bs{\rho}^{n+1}$ is not strictly positive.
On the other hand, if and whenever our scheme outputs $\btheta ^{n+1/2} \geq e^{-1} \brho^n$, then the invertibility $\brho^{n+1}=\Xi(\brho^n,\btheta^{n+1/2})\Leftrightarrow\btheta^{n+1/2}=\Theta(\brho^n,\brho^{n+1})$ holds and the equality holds in \eqref{eq:discreteedi}.
As a consequence $\bs\rho^{n+1}$ solves \eqref{eq:variational} with $\tilde{\brho}^n = \brho^n$, since it realizes the lower bound $\mc{E}_{\Tt}(\bs{\rho}^{n})$.
Our scheme is somehow ``almost variational'', in the sense that it is locally variational except in those situations when entropy releases occur due to equality failure in \eqref{eq:upper_chain_rule}.
The significant advantage of using our scheme is that the optimality conditions for \eqref{eq:variational} are much harder to manage than the system \eqref{eq:discretePDE}--\eqref{eq:rhoG}, and both the theoretical analysis and numerical implementation for \eqref{eq:variational}  become more intricate.
\end{remark}

Starting from the expression of $\psi^*$, easy algebra allows to recast the discrete Fisher functional \eqref{eq:def_Rpsi} as
\begin{equation}\label{eq:Dpsistar}
\mc{R}_{\psi}(\brho) = 2 \sum_{\sigma \in \Sigma} \frac{m_\sigma\pi_\sigma}{d_\sigma} \left|\sqrt{\frac{\rho_K}{\pi_K}} - \sqrt{\frac{\rho_L}{\pi_L}}\right|^2 \,.
\end{equation}
Clearly this is a consistent approximation of the dissipation rate
\begin{equation}\label{eq:dissipationrate}
2 \int_\Omega  \pi\left| \nabla \sqrt{\frac{\rho}{\pi}} \right|^2 \,
=
\frac 12\int_\Omega \rho \left| \nabla \log \frac{\rho}{\pi} \right|^2
\end{equation}
appearing in \eqref{eq:EDE.0}.

In order to gain compactness in the next section we exploit Proposition~\ref{prop:ede} to retrieve uniform bounds for the discrete curves $(\brho^n)_n$, $(\bs F^{n+1/2})_n$, $(\btheta^{n+1/2})_n$.

\begin{lemma}\label{lem:entropybounds}
There exists a constant $C>0$ only depending on $m_\Omega$, $\mc{H}(\rho^0)$, $(\max(V) - \min(V))$, and the total mass $\rho^0[\Omega]$, such that
\[
\sum_{n=0}^{N-1} \tau \left[\mc{D}_\psi(\btheta^{n+1/2},\bs{F}^{n+1/2}) + \mc{R}_{\psi}(\btheta^{n+1/2})\right] \leq C
\]
and
\[
\sup\limits_{0\leq n\leq  N-1}\left(
\mc{H}_{\Tt}(\brho^{n+1}) + \mc{H}_{\Tt}(\btheta^{n+1/2})\right) \leq C.
\]
\end{lemma}
\begin{proof}
Summing Proposition~\ref{prop:ede} over $n$ we get
\begin{multline*}
\sum_{n=0}^{N-1}\tau\left[\mc{D}_\psi(\btheta^{n+1/2},\bs{F}^{n+1/2}) + \mc{R}_{\psi}(\btheta^{n+1/2})\right]
 \overset{\eqref{eq:discreteedi}}{\leq} \E _{\Tt}(\brho^0) - \E _{\Tt}(\brho^N)
\\
 =  \left[\mc{H}_{\Tt}(\brho^0) +\sum_K m_KV_K\rho^0_K\right]- \left[\mc{H}_{\Tt}(\brho^N) +\sum_K m_KV_K\rho^N_K\right]
 \\
 \leq \mc{H}(\rho^0)   + (\max(V) - \min(V)) \int_\Omega \rho^0,
\end{multline*}
where in the last inequality we used successively Jensen's inequality to bound $\mc{H}_{\Tt}(\brho^0) \leq \mc{H}(\rho^0)$, $\mc H_{\mc T}(\bs\rho^N)\geq 0$, and the mass conservation $\sum_Km_k\rho^N_K=\sum_Km_k\rho^0_K=\int_\Omega\rho^0$.
\\
For the bound on $\mc{H}_{\Tt}(\brho^{n+1})$, note first that $\mc D_\psi,\mc R_\psi\geq 0$ in \eqref{eq:discreteedi} and therefore
$$
\E _{\mc T}(\bs \rho^{n+1})
\leq
\E _{\mc T}(\bs \rho^{n})
\leq\dots\leq
\E _{\mc T}(\bs \rho^{0}).
$$
This gives similarly
$$
\begin{aligned}
\mc{H}_{\Tt}(\brho^{n+1})
& =
\E _{\Tt}(\brho^{n+1}) - \sum m_K\rho^{n+1}_K V_K
\\
& \leq
\E _{\Tt}(\brho^{0}) - \sum m_K\rho^{n+1}_K V_K
\\
& =\mc{H}_{\Tt}(\brho^{0})+  \sum m_K\rho^{0}_K V_K - \sum m_K\rho^{n+1}_K V_K
\\
& \leq
\mc{H}(\rho^{0})+  (\max V-\min V)\rho^0[\Omega].
\end{aligned}
$$
As for the bound on $\mc{H}_{\Tt}(\btheta^{n+1/2})$, let us first recall the elementary but useful property of the entropy function
\begin{equation}
 \label{eq:H(c)leq}
 H(c)\leq 1+\frac 12 [H(a)+H(b)]
 \qquad \text{for }a,b,c\geq 0,\quad c\leq \frac{a+b}{2}
\end{equation}
(For $c\leq 1$ one has trivially $H(c)\leq H(0)=1$, while for $c\geq 1$ one can simply use the monotonicity $H(c)\leq H(a+b/2)$ and conclude by convexity.)
Owing to $\theta^{n+1/2}_K\leq \frac 12(\rho^n_K +\rho^{n+1}_K)$ from Proposition~\ref{prop:existence}, we get
\[
\begin{aligned}
\mc{H}_{\Tt}(\btheta^{n+1/2})
&=
\sum_{K} m_K H(\theta^{n+1/2}_K)
\\
&\leq
\sum_{K} m_K \left[1+\frac 12 \left(H(\rho^{n}_K)+H(\rho^{n+1}_K)\right)\right]
\\
& = m_\Omega +  \frac{1}{2} \left[ \mc{H}_{\Tt}(\brho^n) + \mc{H}_{\Tt}(\brho^{n+1})\right] \,,
\end{aligned}
\]
and the previous uniform bound on $\mc{H}_{\Tt}(\brho^n)$ concludes the proof.
\end{proof}

\section{Convergence via the energy dissipation equality}\label{sec:convergence}
In this section we establish the convergence of the discrete solutions associated with a sequence of meshes $\Tt$ and time steps $\tau$, in the limit $\tau, \size(\Tt)\rightarrow 0$, where
\[
\size(\Tt) \coloneqq \max\{\operatorname{diam}(K) \,; \, K \in\Tt\}\,.
\]
For technical reasons that will appear later in the proofs, we require the sequence of meshes to satisfy some asymptotic isotropy condition inspired from~\cite{gladbach2020scaling}, up to some subset of vanishing $d$-dimensional Lebesgue measure as in~\cite{DN18}. 
We further have to assume some CFL-type condition, cf. \eqref{eq:CFL} in what follows, which for quasi-uniform meshes would simply write $\tau = o(\size(\Tt))$.

Throughout the section we will denote by $Q_T\coloneqq [0,T]\times \Omega$ the space-time domain.

\subsection{Assumptions on the sequence of meshes}\label{ssec:mesh2}
Our convergence result relies on the following assumptions on the sequence of meshes: 
\begin{enumerate}[1)]
\item mesh regularity: there exists a constant $\zeta>0$ uniform w.r.t. $\mc T$ such that
  \begin{equation}\label{eq:volumebound}
\sum_{\sigma\in\Sigma_K}\frac{m_\sigma d_\sigma}{2d} \leq \zeta m_K
\quad\text{and}\quad
\zeta^{-1} \operatorname{dist}(x_K,K) \leq \operatorname{diam}(K) \leq \zeta \min_{\sigma \in \Sigma_K} d_\sigma\, ,
 \qquad \forall \, K \in \Tt\,
  \end{equation}
  whereas 
  \begin{equation}\label{eq:dsigmah}
  d_\sigma \leq \zeta \, \size(\Tt), \qquad\forall \sigma \in \overline \Sigma.
  \end{equation} 
\item
asymptotic isotropy: there exists a subset $\Tt_{\mathrm{iso}} \subset \Tt$ and a nonnegative $\eps_{\Tt}\rightarrow 0$ as $\size(\Tt) \rightarrow 0$, such that
\begin{equation}\label{eq:isocond}
 (1-\eps_{\Tt})|v|^2 \leq ~ \frac{1}{2m_K} \sum_{\sigma \in \overline{\Sigma}_K} m_\sigma d_\sigma (v\cdot n_{K\sigma})^2 \leq (1+\eps_{\Tt})|v|^2,\qquad \forall\, K \in \Tt_{\mathrm{iso}},\,
 \forall\,v\in \mathbb{R}^d\,;
\end{equation}
moreover, denoting $\Omega_{\mathrm{iso}} \coloneqq \cup \{ K\, ;\, K \in \Tt_{\mathrm{iso}}\}$ we have
\begin{equation}\label{eq:isovanish}
\lim_{\size(\Tt) \rightarrow 0} \meas(\Omega \setminus \Omega_{\mathrm{iso}}) =0\,;
\end{equation}
  \item CFL-type condition: denoting by $\dmin= \min_{\sigma \in \Sigma} d_\sigma$, then we assume that 
  \begin{equation}\label{eq:CFL}
  \frac{\tau}{\dmin} \to 0  \quad \text{as}\; \size(\Tt) \to 0. 
  \end{equation}  
\end{enumerate}
Conditions~\eqref{eq:volumebound} and \eqref{eq:dsigmah} are satisfied by usual discretizations based on Delaunay triangulations (or dual Voronoi diagrams)
under mild regularity assumptions.
Condition~\eqref{eq:isocond} is much more restrictive.
A weighted version of condition \eqref{eq:isocond} was introduced in \cite{gladbach2020scaling} under the name of \emph{asymptotic isotropy} to study convergence of discrete optimal transport models to their continuous counterparts.
In order to ensure convergence, such weights need to be chosen consistently with the reconstruction operator mapping densities from cells to edges.
In our case, the reconstruction is defined by the map $(\theta_K,\theta_L) \mapsto \theta_\sigma = \sqrt{\theta_K\theta_L}$, and for this specific choice the isotropy assumption in \cite{gladbach2020scaling} takes precisely the form \eqref{eq:isocond}.
This condition also imposes a strong regularity requirement on the meshes.
In particular, taking $v= e_i$ with $\{e_i\}_{i=1}^d$ an orthonormal basis and summing over all $i=1,\ldots,d$, this implies
\begin{equation*}
\frac{1}{2m_K} \sum_{\sigma \in \overline{\Sigma}_K} m_\sigma d_\sigma  \leq d (1+\eps_{\Tt}),\,
\qquad\forall\,K \in \Tt_{\mathrm{iso}}.
\end{equation*}
This is verified if, at least in the limit $\size(\Tt)\rightarrow 0$, each edge $\sigma$ divides the corresponding diamond subcell in two parts of equal area $m_\sigma d_\sigma/2d$.
However, in contrast with \cite{gladbach2020scaling}, we allow the isotropy condition to fail in an asymptotically negligible volume $\Omega_{\mathrm{iso}}$, which is precisely the meaning of \eqref{eq:isovanish}.
This improved flexibility allows us to consider a practical refinement strategy and generate a sequence of meshes for which the assumption is verified; see Remark \ref{rem:subdivision} and \cite{DN18}.
As noted already in \cite{gladbach2020scaling}, condition \eqref{eq:isocond} can be obtained by requiring a stronger condition, which is usually referred to as \emph{superadmissibility} \cite{eymard2010discretization} or \emph{center of mass condition}.
Specifically, denoting by $x_\sigma$ the barycenter of the facet $\sigma$, suppose that
\begin{equation}\label{eq:super}
x_\sigma = \frac{x_K + x_L }{2}, \qquad \forall \sigma = (K|L) \in \Sigma_K
\end{equation}
Then applying Gauss's theorem to the vector fields $\langle x-x_K,e_i\rangle e_j$ for $i,j=1,\ldots,d$, we recover
\[
\frac{2}{m_K} \sum_{\sigma \in \overline{\Sigma}_K} {m_\sigma}{d_\sigma}  (x_\sigma-x_K)\otimes (x_\sigma-x_K) = \mathrm{Id}\,,
\]
which directly implies \eqref{eq:isocond} on the cell $K$. This suggests that the classical refinement procedure by subsequent subdivisions, described below in Remark \ref{rem:subdivision}, generates a sequence of meshes for which the assumption holds. 

Finally, condition~\eqref{eq:CFL} is introduced for purely technical reasons in order to guarantee that the reconstructions based on $\left(\brho^{n+1}\right)_{n\geq 0}$ and $\left(\btheta^{n+1/2}\right)_{n\geq 0}$, defined in~\eqref{eq:densreco} below, share their cluster points as $\size(\Tt) \to 0$.

\begin{remark}[Refinement by subdivision] \label{rem:subdivision}
Given a bounded polygonal set $\Omega\subset \mathbb{R}^2$, consider an admissible mesh $\Tt^0$ made of acute triangles, and subdivide each control volume by partitioning its edges using a fixed number of points and joining the corresponding points on all edges.
Choosing as cell centers the triangles' circumcenters, the superadmissibility condition \eqref{eq:super} holds for all triangles not sharing an edge with the initial partition $\Tt^0$.
Consequently, increasing the number of subdivisions yields a sequence of admissible meshes verifying the asymptotic isotropy assumption above.
\end{remark}


\subsection{Compactness and limit densities}\label{ssec:compact}

Let us define a reconstruction for the discrete densities and fluxes.
For the densities we define, for $0\leq n\leq N-1$,
\begin{equation}\label{eq:densreco}
\begin{aligned}
\rho_{\Tt,\tau}(t,x)
&= \rho_K^{n+1}  \quad \text{for a.e. } x\in K,~ t\in (t^{n},t^{n+1}]\,,
\\
\theta_{\Tt,\tau}(t,x)
&= \theta_K^{n+1/2}  \quad \text{for a.e. } x\in K,~ t\in (t^{n},t^{n+1}]\,,
\\
\theta_{\Sigma,\tau}(t,x)
&= \theta_\sigma^{n+1/2}  \quad \text{for a.e. } x\in \Delta_\sigma,~ t\in (t^{n},t^{n+1}]\,,
\end{aligned}
\end{equation}
where
\[
\theta^{n+1/2}_\sigma = \left\{ \begin{array}{ll}
\sqrt{\theta^{n+1/2}_K\theta^{n+1/2}_L} & \text{if } \sigma =K|L
\\
\theta_K^{n+1/2} & \text{if } \sigma = K\cap \partial \Omega
\end{array}
\right.
\,.
\]
The diamond cell $\Delta_\sigma$ corresponding to the edge $\sigma$ is a polytope included in $\Omega$, the vertices of which being $x_K$ and those of $\sigma$ if $\sigma \subset \partial \Omega$, and additionally $x_L$ if $\sigma = K|L \in \Sigma$.
Note that we do note require $\Delta_\sigma$ to be convex as $x_K$ can lie outside of $K$.

For the initial density profile $\rho^0$, which has been discretized into $\brho^0$ by~\eqref{eq:Vpirho0}, we build the approximation $\rho_\Tt^0$ defined by 
\[
\rho_\Tt^0(x) = \rho^0_K \qquad \text{for a.e. } x\in K. 
\]
Then one readily checks that $\rho_\Tt^0$ converges (strongly) in $L^1(\Omega)$ towards $\rho^0$.
We will also need a reconstruction for the terminal discrete density at $t=T$, which will be given by
\begin{equation}\label{eq:rhoTreco}
\rho^T_{\Tt,\tau}(x) = \rho^{N}_K
\qquad \text{for a.e. } x\in K\,.
\end{equation}
Finally, for the fluxes we use the following reconstruction:
\begin{equation}\label{eq:Freco}
F_{\Sigma,\tau}(t,x) = d F^{n+1/2}_{K\sigma} n_{K\sigma} \qquad \text{for a.e. } \, x \in \Delta_\sigma \,, \, t\in (t^{n},t^{n+1}],
\end{equation}
where $F^{n+1/2}_{K\sigma}$ is defined in \eqref{eq:fluxes} for $\sigma=K|L$, and $F^{n+1/2}_{K\sigma}=0$ on the boundary $\sigma \subset \partial \Omega$.
Note that this is well-defined since $F^{n+1/2}_{K\sigma} = -F^{n+1/2}_{K\sigma}$ and $ n_{K\sigma} = - n_{L\sigma}$ for $\sigma=K|L$.

By Lemma \ref{lem:entropybounds}, the total space-time entropies of $\rho_{\Tt,\tau}$ and $\theta_{\Tt,\tau}$ are uniformly bounded, i.e.\ there exists a constant $C>0$ independent of $\Tt,\tau$ such that
\begin{equation}\label{eq:bound_entropy}
\int_{Q_T} H(\rho_{\Tt,\tau}) \leq CT
\qquad\text{and}\qquad
\int_{Q_T} H(\theta_{\Tt,\tau}) \leq CT\,.
\end{equation}
Therefore, given any family of admissible meshes $\Tt^k$ and time steps $\tau^k$ with $\size(\Tt^k)\rightarrow 0$ and $\tau^k\rightarrow 0$ as $k\rightarrow \infty$, there exists $\rho,\theta \in L^1(Q_T)$ such that, up to extraction of a subsequence if needed and as $k\rightarrow \infty$,
\begin{equation}\label{eq:rhothetalimit}
\theta_{\Tt^k,\tau^k} \rightharpoonup \theta \,,\quad \rho_{\Tt^k,\tau^k} \rightharpoonup \rho \,, \quad \text{weakly in } L^1(Q_T)\,.
\end{equation}
Similarly, since the entropy of $\rho^T_{\Tt^k,\tau^k}$ is uniformly bounded, we have that there exists $\rho^T\in L^1(\Omega)$ such that, up to a further extraction as $k\rightarrow \infty$,
\begin{equation}
\label{eq:rhoTlimit}
\rho^T_{\Tt^k,\tau^k} \rightharpoonup \rho^T\,,\quad \text{ weakly in } L^1(\Omega)\,.
\end{equation}
We claim now that the entropy of $\theta_{\Sigma,\tau}$ is also uniformly bounded.
Indeed, as
\[
\theta_\sigma^{n+1/2}=\sqrt{\theta_K^{n+1/2}\theta_L^{n+1/2}}\leq \frac 12\left(\theta_K^{n+1/2}+\theta_L^{n+1/2}\right),
\]
it follows from our previous entropy bound \eqref{eq:bound_entropy} that 
\begin{equation}
\label{eq:entropy_theta_diamond}
\begin{aligned}
\int_{Q_T} H(\theta_{\Sigma,\tau})
& = \sum_{n=0}^{N-1} \tau \sum_{\sigma\in \Sigma} \meas(\Delta_\sigma)  H(\theta_\sigma^{n+1/2})
\\
& = \sum_{n=0}^{N-1} \tau \sum_{\sigma\in \Sigma} \frac{m_\sigma d_\sigma}{d}  H(\theta_\sigma^{n+1/2})
\\
& \overset{\eqref{eq:H(c)leq}}{ \leq} \sum_{n=0}^{N-1} \tau \sum_{\sigma\in \Sigma} \frac{m_\sigma d_\sigma}{d}\left[1+\frac 12\left(H(\theta_K^{n+1/2})+H(\theta_L^{n+1/2})\right)\right]
\\
&=
\sum_{n=0}^{N-1} \tau \sum_K\sum_{\sigma\in \Sigma_K} \frac{m_\sigma d_\sigma}{2d}\left[1+H(\theta_K^{n+1/2})\right]
\\
& \overset{\eqref{eq:volumebound}}{\leq} \zeta
\sum_{n=0}^{N-1} \tau \sum_K  m_K\left[1+H(\theta_K^{n+1/2})\right]
\\
& = \zeta \left(m_\Omega T + \int_{Q_T} H(\theta_{\Tt,\tau})\right)
 \leq CT.
\end{aligned}
\end{equation}
This gives equiintegrability of $\{\theta_{\Sigma,\tau}\}_{\Tt,\tau}$ for any family of admissible meshes and time steps.
We use this to show that the fluxes $\{F_{\Sigma,\tau}\}_{\Tt,\tau}$ are also equiintegrable.
To this end, denote by $d_{\Sigma} \in L^\infty(\Omega)$ the piecewise constant function equal to $d_\sigma$ in each diamond subcell, and define
\begin{equation}\label{eq:Dpsi}
D_\psi \coloneqq \sum_{n=0}^{N-1} \tau \mc{D}_\psi(\btheta^{n+1/2},\bs{F}^{n+1/2}) = \int_{Q_T} \frac{\theta_{\Sigma,\tau}}{d_{\Sigma}^2}\psi\left( \frac{|F_{\Sigma,\tau}| d_{\Sigma}}{d\theta_{\Sigma,\tau}}\right) \,.
\end{equation}
Observe that $D_\psi$ is uniformly bounded due to Lemma \ref{lem:entropybounds}.
Let now $A\subset Q_T$ be an arbitrary measurable subset, and for any $\eps>0$ write
\[
 \|F_{\Sigma,\tau} \|_{L^1(A)} = \int_A |F_{\Sigma,\tau}| = \frac{d \eps }{2D_\psi}  \int_A  \frac{\theta_{\Sigma,\tau}}{d_{\Sigma}^2} \left(\frac{|F_{\Sigma,\tau}| d_{\Sigma}}{d\theta_{\Sigma,\tau}}\right) \left( \frac{2 d_{\Sigma} D_\psi}{\eps}\right).
\]
By $\psi,\psi^*$-Young's inequality and the expression for $D_\psi$ in \eqref{eq:Dpsi}, we obtain
\[
\begin{aligned}
\|F_{\Sigma,\tau} \|_{L^1(A)}
&\leq  \frac{d \eps }{2D_\psi}  \int_A  \frac{\theta_{\Sigma,\tau}}{d_{\Sigma}^2}\left[\psi\left( \frac{|F_{\Sigma,\tau}| d_{\Sigma}}{d\theta_{\Sigma,\tau}}\right) + \psi^*\left(\frac{2 d_{\Sigma} D_\psi}{\eps}\right) \right]
\\
& \leq  \frac{\eps}{2} + \frac{d \eps }{2D_\psi} \int_A \frac{\theta_{\Sigma,\tau}}{d_{\Sigma}^2} \psi^*\left(\frac{2 d_{\Sigma} D_\psi}{ \eps}\right)
\\
& \leq \frac{\eps}{2}  +\frac{d \eps }{2D_\psi} \max_{\sigma}\left( \frac{1}{d_\sigma^2} \psi^*\left(\frac{2 D_\psi d_\sigma}{ \eps}\right)\right) \int_A \theta_{\Sigma,\tau}\,
\end{aligned}
\]
by definition of $d_{\Sigma}$.
For all $\eps\geq 0$, there exists a constant $C_\eps$ such that $ \psi^*\left( {\xi}/{\eps}\right) = 4(\cosh(\xi/2\eps)-1)\leq C_\eps \xi^2/2$ if $\xi\leq 1$.
Since $D_\psi$ is bounded we have that eventually $\xi_\sigma=2D_\psi d_\sigma\leq 1$ in the $\max\limits_\sigma(\dots)$ term if $\size(\Tt)$ is sufficiently small (using \eqref{eq:dsigmah}), hence
\begin{equation}\label{eq:F.equiint}
\|F_{\Sigma,\tau} \|_{L^1(A)}
\leq
\frac{\eps}{2}
+ d{\eps  C_\eps } D_\psi  \int_A \theta_{\Sigma,\tau}.
\end{equation}
Since $\{\theta_{\Sigma,\tau}\}_{\Tt,\tau}$ is equiintegrable one can pick $\delta=\delta(\eps)>0$ such that $\meas(A)\leq \delta\implies\|\theta_{\Sigma,\tau}\|_{L^1(A)}\leq \frac{1}{2 d C_\eps D_\psi}$, and therefore
$$
\meas(A)\leq \delta \quad \implies\quad \|F_{\Sigma,\tau} \|_{L^1(A)}
\leq
\frac{\eps}{2}
+ d{\eps  C_\eps } D_\psi \|\theta_{\Sigma,\tau}\|_{L^1(A)}
\leq
\eps.
$$
This means precisely that $\{F_{\Sigma,\tau}\}_{\Tt,\tau}$ is equiintegrable, and as a consequence we can assume up to extraction of a further subsequence that
\begin{equation}
\label{eq:Flimit}
F_{\Sigma^k,\tau^k} \rightharpoonup F
\qquad \text{weakly in } L^1(Q_T;\mathbb{R}^d)\,
\end{equation}
for some vector field $F\in L^1(Q_T;\mathbb{R}^d)$.

The next lemma shows that the previous weak limits $\rho,\theta$ from \eqref{eq:rhothetalimit} coincide, and as of now one should keep in mind $\theta=\rho$.
Note carefully that this requires a condition $\tau=o(\dmin)$ on the mesh.

\begin{lemma}\label{lem:limitid}
Consider a sequence of solutions associated with $\left(\Tt^k,\tau^k\right)_k$ satisfying $\size(\Tt^k)\to 0$ and \eqref{eq:CFL}, i.e. $\tau^k = o(\dmink)$ as $k\to \infty$.
Then
\[
\lim_{k\rightarrow \infty} \|\theta_{\Tt^k,\tau^k} - \rho_{\Tt^k,\tau^k}\|_{L^1(Q_T)}  = 0\,.
\]
\end{lemma}
\noindent
\begin{proof}
Let us consider the solution $(\brho^n)_{n=0}^N,(\btheta^{n+1/2})_{n=0}^{N-1}$ obtained for fixed $\Tt,\tau$.
By Proposition \ref{prop:existence} and \eqref{eq:discretePDE}--\eqref{eq:fluxes.linear} we control first by $\psi,\psi^*$-Young inequality
\[
\begin{aligned}
\sum_{K} m_K |\theta^{n+1/2}_K - \rho_K^{n+1}| &
\overset{\eqref{eq:lower_upper_theta}}{\leq} \sum_{K\in\Tt} m_K |\rho^{n+1}_K - \rho_K^n|
\\
 & \leq \, 2 \tau \sum_{\sigma \in \Sigma_\text{int}} \frac{m_\sigma}{d_\sigma} \pi_\sigma\left|\frac{\theta^{n+1/2}_K}{\pi_K} - \frac{\theta^{n+1/2}_L}{\pi_L}\right|
\\
& \leq
2 \tau  \sum_{\sigma \in \Sigma_\text{int}} \frac{m_\sigma}{d_\sigma^2} \theta^{n+1/2}_\sigma
\left[\psi\left( \frac{\pi_\sigma}{\theta^{n+1/2}_\sigma} \left|\frac{\theta^{n+1/2}_K}{\pi_K} - \frac{\theta^{n+1/2}_L}{\pi_L}\right| \right)
+\psi^*(d_\sigma ) \right]
\\
& \leq
 \frac{2\tau}{\dmin}  \sum_{\sigma \in \Sigma} \frac{m_\sigma}{d_\sigma} \theta^{n+1/2}_\sigma
\left[\psi\left( \frac{\pi_\sigma}{\theta^{n+1/2}_\sigma} \left|\frac{\theta^{n+1/2}_K}{\pi_K} - \frac{\theta^{n+1/2}_L}{\pi_L}\right| \right)
+\psi^*(d_\sigma ) \right].
\end{aligned}
\]
Pick now $C>0$ such that $\psi^*(\xi)=4(\cosh(\xi/2)-1)\leq C\xi^2$ for $\xi\leq 1$.
Summing the above estimate over $n$, and recalling the definition \eqref{eq:Dpsi} of $D_\psi$, we obtain
\[
\begin{aligned}
\|\theta_{\Tt,\tau} - \rho_{\Tt,\tau}\|_{L^1(Q_T)} &
= \sum_{n=0}^{N-1} \tau  \sum_{K\in\Tt}  m_K |\theta^{n+1/2}_K - \rho_K^{n+1}|
\\
 & \leq
 \frac{2\tau}{\dmin} \left( D_\psi + C \sum_n \tau \sum_{\sigma\in \Sigma} m_\sigma d_\sigma \theta_\sigma^{n+1/2} \right)
 =
\frac{2\tau}{\dmin} \left( D_\psi + C d\|\theta_{\Sigma,\tau}\|_{L^1(Q_T)} \right).
\end{aligned} 
\]
Recall now that $D_\psi$ is bounded (Lemma~\ref{lem:entropybounds}), and observe that since $\{\theta_{\Sigma,\tau}\}$ is equiintegrable (owing to the entropy bound \eqref{eq:entropy_theta_diamond}) it has bounded $L^1(Q_T)$ norm.
Due to our standing assumption \eqref{eq:CFL} the last term is $o(1)$ as $k\to\infty$ and the proof is complete.
\end{proof}
%
\subsection{Strong convergence of the approximate densities}\label{ssec:compact.strong}

The main goal of this section is to establish improved compactness and therefore strong convergence of the reconstructructions
$\theta_{\Tt,\tau}$ and $\rho_{\Tt,\tau}$.
Since the Fokker-Planck equation is linear this is not strictly required in order to prove the convergence of the scheme in Section \ref{ssec:conv.final} (see also Remark \ref{rem:weakenough}), and this should rather be read as a separate result of independent interest.

As a preliminary, we derive a uniform estimate on the discrete $L^2_t\dot H^1_x$ semi-norm of $\sqrt{\theta_{\Tt,\tau}}$, where the discrete TPFA $\dot H^1$ semi-norm is classically defined as
\[
{|\bs u|}_{1,\Tt}^2 = \sum_{\sigma \in \Sigma} \frac{m_\sigma}{d_\sigma} (u_K - u_L)^2, \qquad \bs u \in \R^\Tt. 
\]
Starting from~\eqref{eq:Dpsistar}, we first rearrange the Fisher information as the sum of a linear part plus the $\dot H^1$ semi-norm
\begin{equation}\label{eq:Fisher.1}
\frac{1}{2} \mc{R}_{\psi}(\brho) = \sum_{\sigma \in \Sigma} \frac{m_\sigma\pi_\sigma}{d_\sigma} \left|\sqrt{\frac{\rho_K}{\pi_K}} - \sqrt{\frac{\rho_L}{\pi_L}}\right|^2 =  \mc{I}(\brho) + |\sqrt{\brho}|_{1,\Tt}^2,
\end{equation}
where
\begin{equation}\label{eq:Fisher.I}
\mc{I} (\brho) \coloneqq \sum_{K \in \Tt} \sum_{\sigma \in \Sigma_K}  \frac{m_\sigma}{d_\sigma} \rho_K \left( \sqrt{\frac{\pi_L}{\pi_K}} -1\right).
\end{equation}

\begin{lemma}\label{lem:dissip2H1}
There exists $C$ depending on $V$ and $\zeta$ (but neither on $\tau$ nor on $\text{size}(\Tt)$) such that
\begin{equation}\label{eq:dissip2H1}
\left| \mc{I} (\brho)\right|
\leq
C \brho[\Omega]
 + \frac12  \left|\sqrt{\brho}\right|_{1,\Tt}^2, \qquad \forall\,\brho\in \R_+^\Tt.
\end{equation}
As a consequence, there exists $C_T$ uniform with respect to $\Tt,\tau$ such that the solution of our scheme satisfies
\begin{equation}\label{eq:L2H1_sqrt_theta}
\sum_{n=0}^{N-1} \tau \left|\sqrt{\btheta^{n+1/2}}\right|_{1,\Tt}^2 \leq C_T.
\end{equation}
\end{lemma}

\begin{proof}
Let us first focus on~\eqref{eq:dissip2H1}. 
Bearing in mind that $\pi_K=e^{-V_K}$, we write first $\sqrt{\frac{\pi_L}{\pi_K}} -1=e^{\frac{V_K-V_L}{2}}-1$ in \eqref{eq:Fisher.I}.
Applying the mean value theorem $e^x - 1 = x + \frac{x^2}2 e^y$ for some $y$ between $0$ and $x$, we further split, for any $\brho\in \R^\Tt_+$
\[
\mc{I} (\brho) = \mc{I}_1 (\brho) + \mc{I}_2 (\brho)
\]
with
\[
\begin{aligned}
 \mc{I}_1 (\brho)
 &
 \coloneqq \frac12 \sum_{\sigma \in \Sigma} \frac{m_\sigma}{d_\sigma}
(\rho_K - \rho_L) (V_K - V_L)
\\
 \mc{I}_2 (\brho)
 &\coloneqq \frac 18\sum_{K \in \Tt} \sum_{\sigma \in \Sigma_K}  \frac{m_\sigma}{d_\sigma}
 \rho_K (V_K - V_L)^2 e^{y_{K,\sigma}^{n+1/2}}
 \end{aligned}
\]
for some $y_{K,\sigma}^{n+1/2}$ between $0$ and $(V_K - V_L)/2$.
As $\rho_K \geq 0$, and using the regularity of $V$, we get that
\[
\mc{I}_2 (\brho)
\leq
e^{\frac{\max V - \min V}2}\|\nabla V\|_\infty^2\frac 18   \sum_{K \in \Tt}
 \rho_K \sum_{\sigma \in \Sigma_K} m_\sigma d_\sigma
 \overset{\eqref{eq:volumebound}}{\leq}
 C  \sum_{K \in \Tt}  \rho_K m_K
 =C\brho[\Omega]
\]
with $C$ depending only on $V$ and $\zeta$.
For the $\mc{I}_1$ term we use next the
identities $a-b = (\sqrt a - \sqrt b) (\sqrt a + \sqrt b)$ and $ab \leq (a^2 + b^2 )/ 4$ to get
\[
\mc{I}_1(\brho) \leq \frac12  \left|\sqrt{\brho}\right|_{1,\Tt}^2 +
\frac18 \sum_{\sigma \in \Sigma} \frac{m_\sigma}{d_\sigma} \left(\sqrt{\rho_K} + \sqrt{\rho_L} \right)^2
\left(V_K - V_L \right)^2.
\]
The Lipschitz continuity of $V$ and the elementary inequality $(\sqrt a + \sqrt b)^2 \leq 2(a + b)$ then yield
\[
\begin{aligned}
\mc{I}_1(\brho)
& \leq
\frac12  \left|\sqrt{\brho}\right|_{1,\Tt}^2 + \frac{\| \nabla V \|^2_\infty}4
\sum_{K \in \Tt}  \rho_K \sum_{\sigma \in \Sigma_K} m_\sigma d_\sigma
\\
&
\overset{\eqref{eq:volumebound}}{\leq}
\frac12  \left|\sqrt{\brho}\right|_{1,\Tt}^2 + \frac{\| \nabla V \|^2_\infty}4
\sum_{K \in \Tt}  \rho_K  2d\zeta m_K
\\
&
=
\frac12  \left|\sqrt{\brho}\right|_{1,\Tt}^2 + C\brho[\Omega],
\end{aligned}
\]
for $C$ depending again only on $V,\zeta$.
Combining the above elements provides~\eqref{eq:dissip2H1}.

Turning now to \eqref{eq:L2H1_sqrt_theta}, observe from \eqref{eq:Fisher.1}--\eqref{eq:dissip2H1} that
$$
\frac12  \left|\sqrt{\brho}\right|_{1,\Tt}^2
\leq
\frac{1}{2} \mc{R}_{\psi}(\brho)
+
C \brho[\Omega],\qquad
\forall\,\brho\in \R_+^\Tt.
$$
Summing over $n$ gives, for any discrete curve $(\btheta^{n+1/2})_{n=0}^{N-1}$,
\begin{equation}
\label{eq:control_sum_H1}
\sum_{n=0}^{N-1} \tau \left|\sqrt{\btheta^{n+1/2}}\right|_{1,\Tt}^2
\leq
\sum_{n=0}^{N-1} \tau \mc{R}_{\psi}(\btheta^{n+1/2})
+
2C\|\theta_{\Tt,\tau}\|_{L^1(Q_T)}.
\end{equation}
When evaluated for our particular solution of the discrete scheme, the first term in the right-hand side is bounded by Lemma~\ref{lem:entropybounds}.
Recalling $\theta^{n+1/2}_K\leq \frac{\rho^{n+1}_K +\rho^n_K}{2}$ from \eqref{eq:lower_upper_theta} and the mass conservation $\brho^{n+1}[\Omega]=\brho^n[\Omega]=\rho^0[\Omega]$ from Proposition~\ref{prop:existence}, we see that $\|\theta_{\Tt,\tau}\|_{L^1(Q_T)}\leq T\rho^0[\Omega]$ and the proof is complete.
\end{proof}

We can now upgrade the previous weak $L^1(Q_T)$ convergence of the approximate densities $\rho_{\Tt,\tau},\theta_{\Tt,\tau}$ into strong convergence to a common limit.
\begin{proposition}\label{prop:convL1}
Let $\rho,\theta$ be as in~\eqref{eq:rhothetalimit} and assume~\eqref{eq:CFL} as in Lemma~\ref{lem:limitid}.
Then
$
\rho=\theta
$
and, up to extraction of a subsequence,
\begin{align}\label{eq:convL1.theta}
&\theta_{\Tt^k, \tau^k} \xrightarrow[k\to+\infty]{}\rho \quad \text{ strongly in $L^1(Q_T)$}, \\
\label{eq:convL1.rho}
&\rho_{\Tt^k, \tau^k} \xrightarrow[k\to+\infty]{} \rho \quad \text{ strongly in $L^1(Q_T)$}.
\end{align}
\end{proposition}
\begin{proof}
Observe first from Lemma~\ref{lem:limitid} that the weak limits must coincide $\rho=\theta$, so it suffices to prove that $\theta_{\Tt^k, \tau^k}\to \rho$ strongly in $L^1(Q_T)$.
Our proof relies on a combination of an Aubin-Lions-Simon concentration-compactness argument and a monotone Minty's trick, already proposed in~\cite{ACM17}.
This will however need some adaptation of results from \cite{moussa2016some} to our specific setup, which we defer to Proposition~\ref{prop:aubin_lions} in the appendix.

Let $f$ be an increasing and bounded function from $\R_+$ to $\R$ such that $f(0) = 0$ and such that $z\mapsto f(z^2)$ is $1$-Lipschitz continuous (typically $f(z) = \tanh\sqrt z$).
Define next the piecewise constant and discrete functions
$$f_{\Tt,\tau} \coloneqq  f(\theta_{\Tt,\tau})
\qquad\text{and}\qquad
\bs{f}^{n+1/2} \coloneqq f(\btheta^{n+1/2}).
$$
Recall from \eqref{eq:rhothetalimit} that
$$
\rho_{\Tt^k,\tau^k}\xrightharpoonup[k\to +\infty]{} \rho
\quad\text{weakly in }L^1(Q_T),
$$
and observe that, since $f$ is bounded, $\{f_{\Tt,\tau}\}$ is bounded in $L^\infty(Q_T)$ and therefore
\begin{equation*}\label{eq:f.Linfwstar}
f_{\Tt^k,\tau^k} \xrightarrow[k\to +\infty]{}\mathfrak f \quad \text{weakly-$\ast$ in $L^\infty(Q_T)$}
\end{equation*}
for some $\mathfrak f \in L^\infty(Q_T)$ and possibly up to extraction of a subsequence.
We aim to use Proposition~\ref{prop:aubin_lions} from the Appendix to guarantee that, with suitable time-compactness on $\{\rho_{\Tt,\tau}\}$ and space compactness on $\{f_{\Tt,\tau}\}$, we can pass to the limit in the product $\rho_{\Tt,\tau}f_{\Tt,\tau}\rightharpoonup\rho \mathfrak f$ in the sense of measures.

We first focus on the space compactness.
Since $z\mapsto f(z^2)$ is $1$-Lipschitz, we have that
\[
\left|\bs{f}^{n+1/2}\right|_{1,\Tt^k} \leq  \left|\sqrt{\btheta^{n+1/2}}\right|_{1,\Tt^k}, \qquad n \geq 0,
\]
hence we deduce from Lemma~\ref{lem:dissip2H1} that
\begin{equation}\label{eq:f.L2H1}
\sum_{n=0}^{N-1} \tau {|\bs{f}^{n+1/2}|}_{1,\Tt^k}^2 \leq C_T.
\end{equation}
A slight adaptation of \cite[Lemma 9.3]{eymard2000finite} shows first of all that the limit $\mathfrak f\in L^2(0,T;H^1(\Omega))$, and moreover controls the $L^2$ space difference quotients by the discrete $H^1$ norm in the following quantitative sense:
there exists a constant $C>0$ only depending on $\Omega$ such that, for any compact subset $\omega \subset\subset \Omega$ and $h \in \mathbb{R}^d$ such that $|h| < \mathrm{dist}(\omega, \partial \Omega) $,

\begin{equation*}
\int_0^T \int_{\omega}
\left|f_{\Tt,\tau} (t,x +h )-f_{\Tt,\tau} (t,x )\right|^2
\leq
|h |\big[|h |+C\,\text{size}(\Tt)\big] \sum_n \tau | \sqrt{\bs{f}^{n+1/2}}|^2_{1,\Tt}
\leq
C_T |h|.
\end{equation*}
This gives in turn the (suboptimal) $L^1$ difference quotient estimate
\begin{equation}
\label{eq:compact.nabla_f}
\int_0^T \int_{\omega}
\left|f_{\Tt^k,\tau^k} (t,x +h )-f_{\Tt^k,\tau^k} (t,x )\right|
\leq C\sqrt{|h|},
\qquad\text{uniformly in }k.
\end{equation}

Turning now to the compactness in time for $\rho$, take any arbitrary test function $\varphi \in C^\infty_c(Q_T)$, define $\left(\bs{\varphi}^{n}\right)_n\in \R^{\Tt}$ by
\[
\varphi_K^{n}
\coloneqq
\frac1{m_K} \int_K \varphi(t^n, x) \ed x, \qquad K \in \Tt, \; n =0,\dots, N-1,
\]
and compute
\begin{align*}
\left|\sum_{n=0}^{N-1} \sum_{K\in\Tt} m_K (\rho_K^{n+1} - \rho_K^n) \varphi_K^{n}\right|
& \overset{\eqref{eq:discretePDE}}{=}
\left|\sum_{n=0}^{N-1} \tau \sum_{\sigma \in \Sig} m_\sigma F_{K\sigma}^{n+1/2}  \left(\varphi_L^{n} -  \varphi_K^{n}\right)\right|
\\
& \leq  \max_{n,\sigma} \left| \frac{\varphi_L^{n} -  \varphi_K^{n}}{d_\sigma} \right|
\sum_{n=0}^{N-1} \tau \sum_{\sigma \in \Sig} m_\sigma d_\sigma \left| F_{K\sigma}^{n+1/2} \right|
\\
& =  \max_{n,\sigma}\left| \frac{\varphi_L^{n} -  \varphi_K^{n}}{d_\sigma} \right|\,\|F_{\Sigma,\tau} \|_{L^1(Q_T)}.
\end{align*}
Taking $\eps=1$ and $A=Q_T$ in \eqref{eq:F.equiint} gives that $\|F_{\Sigma,\tau} \|_{L^1(Q_T)}\leq C$ is bounded uniformly in $\Tt,\tau$.
Moreover, it is shown in~\cite[\S4.4]{ACM17} that
\[
\left| \frac{\varphi_L^{n} -  \varphi_K^{n}}{d_\sigma} \right| \leq (1+2\zeta) \|\nabla \varphi \|_\infty,
\]
so that altogether we get the $\mc M(0,T;(W^{1,\infty}(\Omega)'))$ estimate
\begin{equation}\label{eq:compact.dtrho}
\left|\langle \partial_t\rho_{\Tt,\tau},\varphi\rangle\right|
=
\left|\sum_{n=0}^{N-1} \tau\sum_{K\in\Tt} m_K \frac{\rho_K^{n+1} - \rho_K^n}\tau \varphi_K^{n}\right|
\leq
C  \|\nabla \varphi \|_\infty.
\end{equation}
We are now in position of rigorously applying our Aubin-Lions compactness from Proposition~\ref{prop:aubin_lions}: \eqref{item:rho_f_L1_Linfty} $\rho_{\Tt^k,\tau^k}$ is bounded in $L^\infty(0,T;L^1(\Omega))$ (conservation of mass), hence in $L^1(Q_T)$, and $f_{\Tt^k,\tau^k}$ is bounded in $L^\infty(Q_T)$, the equiintegrability \eqref{item:rho_equi} follows from the entropy bound \eqref{eq:bound_entropy}, the space compactness \eqref{item:sp_compact} is exactly given by \eqref{eq:compact.nabla_f}, and the time compactness \eqref{item:time_compact} is just \eqref{eq:compact.dtrho}.
We conclude that, up to extraction of a subsequence if need be,
\begin{equation}\label{eq:rho*f}
\rho_{\Tt^k,\tau^k} f_{\Tt^k,\tau^k} \xrightharpoonup[k\to+\infty]{} \rho\, \mathfrak f \quad \text{in $\mc M(Q_T)$},
\end{equation}
and from Lemma~\ref{lem:limitid} also
\begin{equation}\label{eq:theta*f}
\theta_{\Tt^k,\tau^k} f_{\Tt^k,\tau^k} \xrightharpoonup[k\to+\infty]{} \rho\, \mathfrak f \quad \text{in $\mc M(Q_T)$}.
\end{equation}
Fix now any $z \in \R_+$ and $\varphi \in C(\bar Q_T)$ with $\varphi\geq 0$.
Recalling that $f_{\Tt,\tau}=f(\theta_{\Tt,\tau})$, the monotonicity of $f$ gives
\[
\int_{Q_T} (\theta_{\Tt^k, \tau^k} - z) (f_{\Tt^k, \tau^k} - f(z)) \, \varphi \geq 0
\]
and therefore owing to \eqref{eq:theta*f}
\[
\int_{Q_T} (\rho - z) (\mathfrak{f} - f(z)) \, \varphi
=
\lim_{k\to +\infty} \int_{Q_T} (\theta_{\Tt^k, \tau^k} - z) (f_{\Tt^k, \tau^k} - f(z)) \, \varphi
\geq 0.
\]
Since $\varphi\geq 0$ is arbitrary we see that
\[
 \forall \, z \geq 0\,,\hspace{1cm}
 (\rho - z) (\mathfrak{f} - f(z)) \geq 0 \quad \text{a.e. in } Q_T,
\]
which implies
\begin{equation}\label{eq:Minty}
\mathfrak{f} = f(\rho) \quad \text{a.e. in }Q_T.
\end{equation}
Finally, we claim that the non-negative sequence $\left(R_k\right)_{k} \subset L^1(Q_T)$ defined by
\[
R_k \coloneqq (\theta_{\Tt^k, \tau^k} - \rho) (f(\theta_{\Tt^k, \tau^k}) - f(\rho))\geq 0
\]
converges to $0$ strongly in $L^1(Q_T)$.
Indeed, thanks to~\eqref{eq:theta*f} and \eqref{eq:Minty}:
\begin{multline*}
0\leq \|R_k\|_{L^1(Q_T)}=\int_{Q_T}R_k
=
\int_{Q_T} (\theta_{\Tt^k, \tau^k} - \rho) (f(\theta_{\Tt^k, \tau^k}) - f(\rho))
\\
=
\int_{Q_T}\theta_{\Tt^k, \tau^k} f_{\Tt^k, \tau^k}
-\int_{Q_T}\rho f_{\Tt^k, \tau^k}
-\int_{Q_T}\theta_{\Tt^k, \tau^k} f(\rho)
+ \int_{Q_T}\rho f(\rho)
\\
\xrightarrow[k\to\infty]{}
\int_{Q_T}\rho\mathfrak{f}
-\int_{Q_T}\rho \mathfrak{f}
-\int_{Q_T}\theta f(\rho)
+\int_{Q_T}\rho f(\rho)=0
\end{multline*}
since we already proved that $\rho=\theta$ and $\mathfrak f=f(\rho)$.
This strong convergence implies almost everywhere convergence $R_k(t,x)\to 0$ in $Q_T$, up to a subsequence if needed.
As $f$ is increasing, this also implies almost everywhere convergence of $\theta_{\Tt^k, \tau^k}$ towards $\rho$.
Vitali's convergence theorem then provides \eqref{eq:convL1.theta}, and \eqref{eq:convL1.rho} finally follows from Lemma~\ref{lem:limitid}.
\end{proof}
\subsection{Asymptotic lower bound for the discrete Fisher information}\label{ssec:conv.Fisher}
 In this section we show that the $\Gamma$-$\liminf$ of the total dissipation functional
\[
\sum_{n=0}^{N-1}\tau \mc{R}_{\psi}(\btheta^{n+1/2})
\]
with respect to the weak $L^1$ convergence is bounded from below by the total dissipation of the continuous system, i.e. the dissipation rate \eqref{eq:dissipationrate} integrated in time.
As this is a statement on the functional itself, throughout this section we will consider arbitrary discrete curves $(\btheta^{n+1/2})_n$ that possibly do not solve \eqref{eq:discretePDE}--\eqref{eq:fluxes.linear}--\eqref{eq:rhoG}.

Let us first consider the easy case of a trivial background potential $V\equiv 0$.
In that case, according to \eqref{eq:Dpsistar}, the dissipation is exactly given by the semi-norm
\[
\frac{1}{2} \mc{D}_{\psi^*}(\brho)
= |\sqrt{\brho}|_{1,\Tt}^2
=\sum_{\sigma \in \Sigma} \frac{m_\sigma}{d_\sigma} |\sqrt{\rho_K} - \sqrt{\rho_L}|^2\,.
\]
Using the previous adaptation of \cite[Lemma 9.3]{eymard2000finite} to control $L^2$ space difference quotients by the discrete seminorms, there exists a constant $C>0$ only depending on $\Omega$ such that, for any compact subset $\omega \subset\subset \Omega$ and $h \in \mathbb{R}^d$ such that $|h| < \mathrm{dist}(\omega, \partial \Omega) $,

\begin{equation}
\label{eq:H1_diff-quotient_xi}
\int_0^T \int_{\omega}  \left|\sqrt{\theta_{\Tt,\tau} (t,x +h )}-\sqrt{\theta_{\Tt,\tau} (t,x )}\right|^2
\leq
\sum_n \tau | \sqrt{\btheta^{n+1/2}}|^2_{1,\Tt}  |h |\big[|h |+C\,\text{size}(\Tt)\big].
\end{equation}
In particular, consider any sequence of meshes and time steps such that $\size(\Tt^k)\rightarrow 0$ and $\tau^k \rightarrow 0$, and let $(\btheta^{n+1/2}_k)_{n=0}^{N-1}$ be any associated discrete curve (again, not necessarily solution to our discrete scheme).
Assume that the reconstruction $\theta_{\Tt^k,\tau^k} \in L^1(Q_T)$ from \eqref{eq:densreco} converges as
\begin{equation*}\label{eq:weakL1}
\theta_{\Tt^k,\tau^k}\rightharpoonup \theta \quad \text{weakly in } L^1(Q_T)\,.
\end{equation*}
Because $f(a,b)=|\sqrt a-\sqrt b|^2$ is convex and continuous, the left-hand side of \eqref{eq:H1_diff-quotient_xi} is lower-semicontinuous for the weak $L^1$ convergence, hence
\[
\int_0^T \left \|\sqrt{\theta(t,\cdot +h )} -\sqrt{ \theta(t,\cdot)}\right \|_{L^2(\omega)}^2 \ed t \,
\leq
|h |^2 \liminf_{k\rightarrow \infty}  \sum_n \tau^k | \sqrt{\btheta^{n+1/2}_k}|^2_{1,\Tt^k}
\]
and therefore, by classical characterization of $H^1(\Omega)$ by difference quotients,
\begin{equation}\label{eq:h1sqrt}
\int_0^T \|\nabla \sqrt{\theta(t,\cdot)} \|^2_{L^2(\Omega)} \ed t
\leq
\liminf_{k\rightarrow \infty}  \sum_n \tau^k | \sqrt{\btheta^{n+1/2}_k}|^2_{1,\Tt^k}
\,.
\end{equation}
This settles the case $V\equiv 0$.

In order to prove the analogue result in the presence of a non-zero potential $V$, we will rewrite below the dissipation functional as the sum of the above $H^1$ seminorm plus a linear term, which can be dealt with easily at least when $\nabla V\cdot n_{\partial \Omega}=0$ on the boundary.
More precisely, at the continuous level, if $\nabla V\cdot n_{\partial \Omega}=0$ on the boundary we have the identity
\begin{equation}
\label{eq:identity_Vn=0}
\int \pi \left|\nabla \sqrt{\frac{\rho}{\pi}} \right|^2 = \frac{1}{2} \int \rho \left(\frac{|\nabla V|^2}2 -\Delta V\right) + \int  \left|\nabla \sqrt{\rho} \right|^2\,.
\end{equation}
This formula can be directly related to the expression for the discrete Fisher information \eqref{eq:Dpsistar}, decomposed into $\frac{1}{2} \mc{R}_{\psi}(\brho) =  \mc{I}(\brho) + |\sqrt{\brho}|_{1,\Tt}^2$ as in~\eqref{eq:Fisher.1}--\eqref{eq:Fisher.I}.
The case $\nabla V\cdot n_{\partial \Omega}\not\equiv 0$ on the boundary will be handled via an approximation argument from~\cite{droniou2002density}.

\begin{proposition}\label{prop:fisher}
Let $(\btheta^{n+1/2}_k)_{n=0}^{N-1} \in (\mathbb{R}^\Tt_{+})^N$ be a given discrete curve associated with $\Tt^k,\tau^k$ with reconstruction $\theta_{\Tt^k,\tau^k}\in L^1(Q_T)$ as in \eqref{eq:densreco}, and suppose that
\[
\theta_{\Tt^k,\tau^k} \rightharpoonup \theta \,, \quad \text{weakly in } L^1(Q_T)\,
\]
for some $\theta\in L^1(Q_T)$.
Then
\begin{equation*}
\liminf_{k\rightarrow \infty}  \sum_n \tau^k \mc{R}_{\psi}(\btheta^{n+1/2}_k)
\geq
2\int_{Q_T} \pi \left|\nabla \sqrt{\frac{\theta}{\pi}} \right|^2\,.
\end{equation*}
\end{proposition}
\begin{proof}
Consider first $\nabla V \cdot n_{\partial \Omega} = 0$ on the boundary, and assume that the $\liminf$ in the statement is finite (otherwise the statement is vacuous).
By \eqref{eq:control_sum_H1} we readily obtain that
$$
\liminf \limits_{k\to \infty}
\sum_n \tau^k  \left|\sqrt{\btheta^{n+1/2}_k}\right|^2_{1,\Tt}
\leq
2 C\limsup_{k\rightarrow \infty} \|\theta_{\Tt^k,\tau^k}\|_{L^1(Q_T)}
+
\liminf_{k\rightarrow \infty}  \sum_n \tau^k \mc{R}_{\psi}(\btheta^{n+1/2}_k)
<+\infty
$$
is finite.
By the previous considerations for $V\equiv 0$ we see that \eqref{eq:h1sqrt} holds, hence comparing \eqref{eq:Fisher.1}--\eqref{eq:Fisher.I} on the one hand and \eqref{eq:identity_Vn=0} on the other hand, clearly it suffices to show:
\begin{equation}
\label{eq:linearpart}
\liminf_{k\rightarrow \infty} \sum_n \tau^k \mc{I}(\btheta^{n+1/2}_k)  \geq  \frac{1}{2} \int \theta \left(\frac{|\nabla V|^2}2 -\Delta V\right)\,.
\end{equation}
To this end we first observe that
\[
\sqrt{\frac{\pi_L}{\pi_K}}-1 = \exp\left(\frac{V_K-V_L}{2}\right) -1 \geq \frac{V_K-V_L}{2} + \frac{(V_K-V_L)^2}{8} .
\]
Therefore we can write $\mc{I}(\brho) \geq \mc{I}_1(\brho) + \mc{I}_2(\brho)$ where
\[
\mc{I}_1(\brho)\coloneqq \sum_{K \in \Tt} \sum_{\sigma \in \Sigma_K}  \frac{m_\sigma}{d_\sigma} \rho_K \frac{(V_K-V_L)^2}{8} \,,
\]
\[ \mc{I}_2(\brho)\coloneqq \sum_{K \in \Tt} \sum_{\sigma \in \Sigma_K}  \frac{m_\sigma}{d_\sigma} \rho_K \frac{V_K-V_L}{2}\,.
\]
For fixed $K\in\Tt$ we have
\[
\sum_{\sigma \in \Sigma_K}  \frac{m_\sigma}{d_\sigma} \frac{V_K-V_L}{2}
\geq
-\frac{1}{2} \int_K \Delta V - C \size(\Tt^k),
\]
where $C$ only depends on $V$.

Since $V$ belongs to $C^2(\overline \Omega)$, and because of the orthogonality condition~\eqref{eq:ortho}, 
there holds 
\[
\frac{V_K-V_L}{d_\sigma} = - \nabla V(x_K) \cdot n_{K\sigma} + \mc{O}(d_\sigma) = - \nabla V(x) \cdot n_{K\sigma} + \mc{O}(d_\sigma + |x-x_K|)
\]
for any $x \in K$. Therefore, we deduce from~\eqref{eq:volumebound} and \eqref{eq:dsigmah} that 
\[
\frac{V_K-V_L}{d_\sigma} = - \nabla V(x) \cdot n_{K\sigma} + \mc{O}(\operatorname{size}({\mc T}^k)).
\]
Using the isotropy condition \eqref{eq:isocond} and then integrating w.r.t. $x\in K$ yields
\begin{equation}\label{eq:trucmuche}
\sum_{\sigma \in \Sigma_K}  \frac{m_\sigma}{d_\sigma} \frac{|V_K-V_L|^2}{8} \geq(1-\varepsilon_{\Tt^k}) \left(\int_K \frac{|\nabla V|^2}{4} - C \size(\Tt^k) \right)
= \int_K \frac{|\nabla V|^2}{4} + o(1),
\end{equation}
where the $o(1)$ remainder is uniform both in $K\in \Tt^k$ and $n=1,\dots N-1$ as $k\to\infty$.
Hence we obtain that
\[
\sum_n \tau^k \mc{I}_1(\btheta^{n+1/2}_k) \geq  \int_{Q_T^k} \theta_{\Tt^k,\tau^k} \frac{|\nabla V|^2}{4}\,+\, o(1)
\]
and
\[
\sum_n \tau^k \mc{I}_2(\btheta^{n+1/2}_k) \geq - \int_{Q_T} \theta_{\Tt^k,\tau^k} \Delta V \,+\, o(1),
\]
where we have set
\[
Q_T^k \coloneqq (0,T)\times(\Omega\setminus \Omega^k_{\mathrm{iso}}).
\]
Since we assume $ \theta_{\Tt^k,\tau^k}\rightharpoonup\theta$ weakly in $L^1(Q_T)$, the  inequality involving $\mc I_2$ immediately passes to the limit.
For the $\mc I_1$ lower bound, the weak convergence implies that $\{\theta_{\Tt^k,\tau^k}\}_k$ is equi-integrable.
Owing to our standing assumption \eqref{eq:isovanish} we see that $\meas\left(Q\setminus Q_T^k\right)\to 0$ hence the $\mc I_1$ inequality also passes to the limit and our claim \eqref{eq:linearpart} follows.

Let us finally settle the case $\nabla V\cdot n_{\partial \Omega}\neq 0$ on the boundary.
Fix $\eps>0$ and take an approximation $V^\eps\in C^\infty_c(\mathbb{R}^d)$ of $V$ satisfying
\[
\nabla V^\eps \cdot n_{\partial \Omega} = 0 ~~ \text{ on } \partial \Omega, \quad \|V-V^\eps\|_{W^{1,p}(\Omega)} \leq \eps,
\]
for some $p\in[1,\infty)$ to be chosen later.
The existence of such a function for arbitrarily large but finite $p\geq 1$ is due to Droniou \cite{droniou2002density}.
Define $\mc{I}_2^\eps$ in the obvious way, simply substituting $V^\eps$ for $V$ in the previous definition of $\mc I_2$, and for any discrete function $\bs\rho\in\R_+^{\Tt}$ decompose now $\frac 12\mc R_\psi(\bs\rho)=\mc{I}(\brho) +|\sqrt{\brho}|_{1,\Tt}^2\geq \mc{I}_1(\brho)+\mc{I}_2(\brho) +|\sqrt{\brho}|_{1,\Tt}^2$ as
\begin{equation}
\label{eq:R_I_Ieps}
\frac{1}{2} \mc{R}_\psi(\brho)
\geq
\left[\mc{I}_1(\brho) + \mc{I}_2^\eps(\brho)  +  |\sqrt{\brho}|_{1,\Tt}^2\right]
+
\left[\mc{I}_2(\brho)-\mc{I}_2^\eps(\brho)\right]\,.
\end{equation}
Let us first estimate the difference $\mc{I}_2-\mc{I}^\eps_2$.
To this end, pick an arbitrary $\eta>0$ small, and for any $\btheta\in\R_+^\Tt$ with spatial reconstruction $\theta_\Tt\in L^1(\Omega)$ write
\[
\begin{aligned}
\left| \mc{I}_2(\btheta^{ })-  \mc{I}_2^\eps(\btheta^{ })\right|
&
\leq    \sum_{\sigma\in\Sigma} \frac{m_\sigma}{2d_\sigma}|\theta_K^{ } -\theta_L^{ }| | V_K- V_K^\eps -  V_L + V_L^\eps|
\\
& =
\sum_{\sigma\in\Sigma} \frac{m_\sigma}{2d_\sigma}|\sqrt{\theta_K} -\sqrt{\theta_L}|\cdot \left(\sqrt{\theta_K} +\sqrt{\theta_L}\right) | V_K- V_K^\eps -  V_L + V_L^\eps|
\\
&
\leq   \left(  \eta  \left|\sqrt{\btheta^{ }}\right|_{1,\Tt}^2 + \frac{1}{4\eta}  \sum_{\sigma\in\Sigma} \frac{m_\sigma}{d_\sigma} \left|\sqrt{\theta_K^{ }} +\sqrt{\theta_L^{ }}\right|^2 \left| V_K- V_K^\eps -  V_L + V_L^\eps\right|^2 \right)
\\
&
\leq   \left(  \eta  \left|\sqrt{\btheta^{ }}\right|_{1,\Tt}^2 + \frac{1}{2\eta}  \sum_{\sigma\in\Sigma} \frac{m_\sigma}{d_\sigma} \left(\theta_K^{ } +\theta_L^{ }\right)\, \left| V_K- V_K^\eps -  V_L + V_L^\eps\right|^2 \right).
\end{aligned} 
\]
Similar arguments as those employed to establish~\eqref{eq:trucmuche} show that
\[
 \sum_{\sigma \in \Sigma_K} \frac{m_\sigma}{2d_\sigma} {| V_K- V_K^\eps -  V_L + V_L^\eps|^2} \leq  \frac{1}{m_K}\int_K |\nabla (V-V^\eps)|^2 + C_\eps (\varepsilon_\Tt + \size(\Tt)),
 \qquad \forall\,K\in \Tt_\mathrm{iso},
\]
where the constant $C_\eps$ depends on $V-V^{\eps}$ but not on the mesh.
On the other hand for $K\not\in \Tt_\mathrm{iso}^k$ we simply write
\[
\begin{aligned}
 \forall\,K\not\in \Tt_\mathrm{iso}^k,\qquad \sum_{\sigma \in \Sigma_K} \frac{m_\sigma}{2d_\sigma} {| V_K- V_K^\eps -  V_L + V_L^\eps|^2}
 & \leq
 \sum_{\sigma \in \Sigma_K} \frac{m_\sigma d_\sigma}{2} \left(\|\nabla V\|_\infty +\|\nabla V^\eps\|_\infty\right)^2
 \\
 & \overset{\eqref{eq:volumebound}}{\leq} d\zeta m_k \left(\|\nabla V\|_\infty +\|\nabla V^\eps\|_\infty\right)^2
 \\
  & \leq C_{\eps}m_K,
 \end{aligned}
\]
where $C_\eps$ again depends on $\eps$ but not on $\Tt^k$.
Hence we find
\begin{multline*}
\left|\mc{I}_2(\btheta )-\mc{I}_2^\eps(\btheta )\right|
 \leq \eta   |\sqrt{\btheta }|_{1,\Tt}^2 + \frac{1}{\eta} \int_{\Omega} \theta_{\Tt} |\nabla (V- V^\eps)|^2
 \\
 + C_\eps  \left((\varepsilon_\Tt + \size(\Tt)) \|\theta_{\Tt}\|_{L^1(\Omega)}
 +\|\theta_\Tt\| _{L^1(\Omega\setminus\Omega_{\mathrm{iso}})}\right),
\end{multline*}
which inserted into \eqref{eq:R_I_Ieps} yields
\begin{multline*}
\frac{1}{2} \mc{R}_\psi(\btheta)
\geq
\left[\mc{I}_1(\btheta) + \mc{I}_2^\eps(\btheta)  +  (1-\eta)|\sqrt{\btheta}|_{1,\Tt}^2\right]
\\
-\frac{1}{\eta} \int_{\Omega} \theta_{\Tt} |\nabla (V- V^\eps)|^2
-C_\eps  \left((\varepsilon_\Tt + \size(\Tt)) \|\theta_{\Tt}\|_{L^1(\Omega)}  +\|\theta_\Tt\| _{L^1(\Omega\setminus\Omega_{\mathrm{iso}})}\right).
\end{multline*}
Evaluating for a discrete curve $(\bs \theta^{n+1/2}_k)$ and summing over $n$, the first three terms in the right-hand side pass to the $\liminf$ as soon as $\theta_{\Tt^k,\tau^k}\rightharpoonup\theta$ weakly in $L^1(Q_T)$, exactly as in the previous case where $\nabla V\cdot n_{\partial \Omega}=0$ on the boundary.
Moreover the non-isotropic term vanishes as $k\to\infty$ -- as before, $\theta_{\Tt^k,\tau^k}$ is equiintegrable and $Q^k_T=(0,T)\times (\Omega\setminus\Omega_{\mathrm{iso}})$ has vanishing measure -- hence for fixed $\eta>0$ and $\eps>0$ we obtain
\begin{multline*}
 \liminf_{k\rightarrow\infty} \sum_n \tau^k \mc{R}_\psi(\btheta^{n+1/2}_k)
 \\
\geq
\int_{Q_T} \theta \left(\frac{|\nabla V|^2}2 -\Delta V^\eps\right)
+ 2(1-\eta) \int_{Q_T} |\nabla\sqrt{ \theta}|^2
-\frac{2}{\eta} \int_{Q_T} \theta |\nabla (V-V^\eps)|^2
\end{multline*}
This shows as a byproduct that, whenever the left-hand side is finite, $\| \nabla \sqrt{\theta}\|_{L^2(Q_T)}$ is finite too and $\nabla\theta =2\sqrt\theta\nabla\sqrt\theta\in L^2_tL^1_x$ due to $\sqrt\theta\in L^\infty_tL^2_x$.
At this stage one would wish to substitute $-\int \theta \Delta V^\eps$ by the desired $+\int\nabla\theta\cdot\nabla V$.
To this end we use the exact same strategy as before, but this time at the continuous level:
Since $\nabla V^\eps\cdot n_{\partial \Omega}$ on the boundary and $\theta \in L^2_tW^{1,1}_x$ we can legitimately integrate by parts
$$
\begin{aligned}
 \left| \int_{Q_T} \theta \Delta V^\eps+ \int_{Q_T}   \nabla \theta \cdot  \nabla V  \right|
  &
  = \left| \int_{Q_T} \nabla\theta\cdot (\nabla V^\eps-\nabla V)\right|
  \\
  & = 2 \left| \int_{Q_T}\nabla\sqrt\theta\cdot  \sqrt{\theta}(\nabla V^\eps-\nabla V)\right|
  \\
  & \leq 2\eta \int_{Q_T}\left|\nabla\sqrt\theta\right|^2 + \frac{2}{\eta}\int_{Q_T}\theta\left|\nabla V^\eps-\nabla V\right|^2
\end{aligned}
$$
and thus
\begin{multline*}
\liminf_{k\rightarrow\infty} \sum_n \tau^k \mc{R}(\btheta^{n+1/2}_k)
\\
\geq
\int_{Q_T} \left( \theta \frac{|\nabla V|^2}2+ \nabla \theta \cdot  \nabla V \right)
+ (2-4\eta) \int_{Q_T} |\nabla\sqrt{ \theta}|^2
-\frac{4}{\eta} \int_{Q_T} \theta |\nabla (V-V^\eps)|^2\,.
\end{multline*}
Fix now $q>1$ such that, for any $f\in H^1(\Omega)$,
\[
\|f\|_{L^{2q}(\Omega)} \leq C_q \|f\|_{H^1(\Omega)}
\]
for some constant $C_q>0$ depending only on $\Omega$ and $q$.
Choosing $p$ as $1/p+1/q =1$ we estimate in the last term
\[
\begin{aligned}
\int_{Q_T} \theta |\nabla (V-V^\eps)|^2
& \leq
\|\nabla (V-V^\eps)\|_{L^{2p}(\Omega)}^2 \int_0^T \| \theta(t,\cdot) \|_{L^q(\Omega)} \ed t
\\
& =
\|\nabla (V-V^\eps)\|_{L^{2p}(\Omega)}^2 \int_0^T \| \sqrt{\theta}(t,\cdot) \|_{L^{2q}(\Omega)}^2 \ed t
\\
& \leq
{C_q^2}  \|\nabla (V-V^\eps)\|_{L^{2p}(\Omega)}^2 \|\sqrt{\theta}\|_{L^{2}(0,T;H^1(\Omega))}^2
\end{aligned}
\]
Choosing $\eps>0$ sufficiently small so that
\[
C_q^2 \|\nabla (V-V^\eps)\|^{2}_{L^{2p}(\Omega)} \leq \eta^2
\]
gives
\begin{multline*}
\liminf_{k\rightarrow\infty} \sum_n \tau^k \mc{R}(\btheta^{n+1/2}_k)
\\
\geq
\int_{Q_T} \underbrace{\left( \theta \frac{|\nabla V|^2}2+ \nabla \theta \cdot  \nabla V  + 2 \left|\nabla\sqrt\theta\right|^2\right)}_{=2\pi \Big|\nabla \sqrt{\frac{\theta}{\pi}} \Big|^2}
-4\eta \int_{Q_T} |\nabla\sqrt{ \theta}|^2
-4\eta,
\end{multline*}
and since $\eta>0$ was arbitrary the proof is complete.
\end{proof}
%
\subsection{Asymptotic lower bound for the discrete Benamou-Brenier functional}\label{ssec:conv.BB}
In this section we establish a lower bound for the kinetic part of the dissipation, analogous to the previous section but now for
\[
\sum_{n=0}^{N-1}\tau \mc{D}_\psi(\btheta^{n+1/2},\bs{F}^{n+1/2}) = \sum_{n=0}^{N-1}\tau \sum_{\sigma\in \Sigma} \frac{m_\sigma}{d_\sigma} \theta^{n+1/2}_\sigma \psi\left( \frac{d_\sigma F_{K\sigma}^{n+1/2}}{\theta_\sigma^{n+1/2}} \right)\,.
\]
The lower bound will be given by the Benamou-Brenier functional (see \cite[Section 5.3.1]{santambrogio2015optimal}), which is the map given by $(\rho,F) \in L^1(Q_T;\mathbb{R}^{d+1}) \mapsto \int_{Q_T} B(\rho,F)$ with the function $B$ defined as in \eqref{eq:BB}.

Just as in section~\ref{ssec:conv.Fisher} for the Fisher information, this statements is about how the geometry discretization allows a consistent interplay between the discrete $\mc D_\psi$ dissipation and the continuous Benamou-Brenier functional functional, and does not deal with particular solutions of our scheme properly speaking.
Accordingly, consider $\theta \in L^1(Q_T)$ and $F\in L^1(Q_T;\mathbb{R}^d)$ that are obtained as weak $L^1$ limits of arbitrary sequences of reconstructed densities $\{\theta_{\Tt^k,\tau^k}\}_{k}$ and fluxes $\{F_{\Sigma^k,\tau^k}\}_{k}$.
Our strategy below is similar to \cite{lavenant2021unconditional}, where unconditional convergence of discrete to continuous optimal transport models is proved.
Recall that the Benamou-Brenier functional can be classically written (e.g. \cite[Proposition 5.18]{santambrogio2015optimal}) as
\begin{equation}
\label{eq:dualB}
\int_{Q_T} B(\theta,F)
=
\sup_{b\in C(\overline{Q}_T;\mathbb{R}^d)} \left\{ \langle b,F \rangle - \int_{Q_T} \theta \frac{|b|^2}{2}\right\},
\end{equation}
where the duality pairing is $\langle b,F\rangle=\int_{Q_T}b\cdot F $.
Consider first the case when \eqref{eq:dualB} is finite.
Then, for any arbitrary small $\eta>0$ we can find $b \in C(\overline{Q}_T;\mathbb{R}^d)$ such that
\begin{equation}\label{eq:etabound}
\begin{aligned}
\int_{Q_T}B(\theta,F) &\leq \langle b,F\rangle - \int_{Q_T} \theta \frac{|b|^2}{2} +\eta
& \leq \lim_{k\rightarrow \infty}  \langle b,F_{\Sigma^k,\tau^k} \rangle - \int_{Q_T} \theta_{\Tt^k,\tau^k} \frac{|b|^2}{2} +\eta ,
\end{aligned}
\end{equation}
and by density we can actually assume that $b \in C^1(\overline{Q}_T;\mathbb{R}^d)$.
At the discrete level, the analogue of \eqref{eq:dualB} is (by definition of the $\psi,\psi^*$ convex duality)
\begin{multline*}
\sum_n \tau \sum_\sigma \frac{m_\sigma \theta_\sigma^{n+1/2}}{d_\sigma} \psi\left( \frac{F_\sigma^{n+1/2} d_\sigma}{\theta_\sigma^{n+1/2}} \right)
\\
=\sup_{\mathbf{b}\in\mathbb{F}_\Tt}\left\{ \sum_{n,\sigma} \tau m_\sigma d_\sigma  b^{n+1/2}_{K\sigma} F^{n+1/2}_{K\sigma}-  \sum_n \tau \sum_\sigma \frac{m_\sigma \theta_\sigma^{n+1/2}}{d_\sigma}\psi^*\left( b_{K\sigma}^{n+1/2} d_\sigma\right)\right\}\,.
\end{multline*}
Let us take $b$ as in \eqref{eq:etabound}, and define
\[
b_{K\sigma}^{n+1/2} \coloneqq \frac{1}{\tau\meas(\Delta_\sigma)}\int_{t^n}^{t^{n+1}}\int_{\Delta_\sigma} b\cdot n_{K\sigma}
\]
so that
\[
\langle b,F_{\Sigma,\tau} \rangle
=
\sum_{n,\sigma}  \tau m_\sigma d_\sigma  b^{n+1/2}_{K\sigma}F^{n+1/2}_{K\sigma} \,.
\]
Moreover owing to $\theta^{n+1/2}_\sigma=\sqrt{\theta^{n+1/2}_K\theta^{n+1/2}_L}\leq \frac 12(\theta^{n+1/2}_K+\theta^{n+1/2}_L)$ there holds
\[
 \sum_n \tau \sum_\sigma \frac{m_\sigma \theta_\sigma^{n+1/2}}{d_\sigma}\psi^*\left( b_{K\sigma}^{n+1/2} d_\sigma\right)
\leq
\sum_n \tau \sum_K\theta_K^{n+1/2} \sum_{\sigma \in \Sigma_K} \frac{m_\sigma }{2 d_\sigma}\psi^*\left( b_{K\sigma}^{n+1/2} d_\sigma\right).
\]
Note that $|b_{K\sigma}^{n+1/2}|\leq \|b\|_{\infty}$ and, owing to our assumption \eqref{eq:dsigmah} on the meshes, $d_\sigma\leq \zeta \,\size(\Tt)$ is small.
Since $ \psi^*(\xi)=4(\cosh(\xi/2)-1)= |\xi|^2/2+\mc O(|\xi|^4)\leq \frac 12|\xi|^2(1+|\xi|)$ for small $\xi$ we can bound
$$
\psi^*\left( b_{K\sigma}^{n+1/2} d_\sigma\right)
\leq
\frac{d_\sigma^2}{2} \big|b_{K\sigma}^{n+1/2}\big|^2\big[1 + \|b\|_\infty d_\sigma\big]\,.
$$
Since $b\in C^1(\bar Q_T)$ is Lipschitz and $d_\sigma\leq \zeta \,\size(\Tt)$ we obtain
\[
\psi^*\left( b_{K\sigma}^{n+1/2} d_\sigma\right)
\leq
\frac{d_\sigma^2}{2} \left[(b(t,x)\cdot n_{K\sigma})^2 + C (\tau+\size(\Tt))\right],
\qquad (t,x) \in [t^{n}, t^{n+1}]\times K
\]
for some constant $C>0$ depending on $b,\zeta$, but not on $\Tt,\tau$.
Leveraging one last time the Lipschitz regularity of $b$ we see that
\[
\psi^*\left( b_{K\sigma}^{n+1/2} d_\sigma\right) \leq \frac{d_\sigma^2}{2} \left[ \frac{1}{\tau m_K} \int_{t^n}^{t^{n+1}}\int_{K}(b(t,x)\cdot n_{K\sigma})^2 + C (\tau+\size(\Tt))\right]\,.
\]

Hence, using yet again the mesh isotropy, we get
\[
\begin{aligned}
 \sum_{n,K}\tau &\theta_K^{n+1/2} \sum_{\sigma \in \Sigma_K} \frac{m_\sigma }{2 d_\sigma}\psi^*\left( b_{K\sigma}^n d_\sigma\right)\\
 &\leq \sum_{n,K} {\theta_K^{n+1/2}}\sum_{\sigma \in \Sigma_K}\left[\int_{t^n}^{t^{n+1}} \int_K  \frac{m_\sigma d_\sigma }{4 m_K }  (b(t,x)\cdot n_{K\sigma})^2 + \tau \frac{m_\sigma d_\sigma}{2}C (\tau+\size(\Tt))\right]
 \\&
 \overset{\eqref{eq:isocond}--\eqref{eq:volumebound}}{\leq}
 (1+\eps_{\Tt}) \int_{Q_T} \theta_{\Tt,\tau} \frac{|b|^2}{2} + \frac{\zeta\|b\|_\infty^2}{2}\int_0^T \int_{\Omega\setminus \Omega_{\mathrm{iso}}} \theta_{\Tt,\tau} + \frac{Cd \zeta}{2} (\tau+\size(\Tt))\|\theta_{\Tt,\tau}\|_{L^1(Q_T)}
 \\&
\leq
\int_{Q_T}  \theta_{\Tt,\tau} \frac{|b|^2}{2}+ r_{\Tt,\tau}
\end{aligned}
\]
where the remainder
\[
r_{\Tt,\tau} \coloneqq
\frac{\zeta\|b\|_\infty^2}{2}\int_0^T \int_{\Omega\setminus \Omega_{\mathrm{iso}}} \theta_{\Tt,\tau} + \left[\frac{Cd \zeta}{2} (\tau+\size(\Tt)) + \frac{\|b\|_\infty^2}{2} \eps_{\Tt} \right]
\|\theta_{\Tt,\tau}\|_{L^1(Q_T)} \,.
\]
Therefore, combining the previous estimates with \eqref{eq:etabound} we deduce that
\[
\sum_{n=0}^{N-1}\tau \mc{D}_\psi(\btheta^{n+1/2}_k)
\geq
\langle b, F_{\Sigma^k,\tau^k}\rangle - \int_{Q_T} \theta_{\Tt^k,\tau^k} \frac{|b|^2}{2} - r_{\Tt^k,\tau^k}\,.
\]
Since $(\theta_{\Tt^k,\tau^k})_k$ is converging weakly in $L^1$ it is also equiintegrable and $L^1$-bounded.
Due to our assumption \eqref{eq:isovanish} that $ \meas((0,T)\times(\Omega \setminus \Omega_{\mathrm{iso}}))\to 0$ we see that $r_{\Tt^k,\tau^k}\rightarrow 0$ as $k\rightarrow \infty$, and as a consequence
\[
\liminf_{k\rightarrow \infty} \sum_{n=0}^{N-1}\tau \mc{D}_\psi(\btheta^{n+1/2}_k)
\geq
\langle b, F\rangle - \int_{Q_T} \theta \frac{|b|^2}{2}
\geq
\int_{Q_T}  B(\theta,F) - \eta\,.
\]
Since $\eta>0$ was arbitrary the claim follows.

If now \eqref{eq:dualB} is infinite we can proceed in a similar fashion.
For any fixed $M>0$ large there is $b \in C^1(\overline{Q}_T;\mathbb{R}^d)$ such that
\[
\begin{aligned}
 \langle b,F\rangle - \int_{Q_T} \theta \frac{|b|^2}{2} \geq M.
\end{aligned}
\]
Following the same line of thought as above we find
\[
\liminf_{k\rightarrow \infty} \sum_{n=0}^{N-1}\tau \mc{D}_\psi(\btheta^{n+1/2}_k) \geq M.
\]
Since $M$ is arbitrary, this $\liminf$ is also infinite.
We have just proven the following:
\begin{proposition}\label{prop:bb}
Let $(\btheta^{n+1/2}_k)_{n=0}^{N-1} \in (\mathbb{R}^\Tt_{+})^N$ and $(\bs{F}^{n+1/2}_k)_{n=0}^{N-1} \in (\mathbb{F}_\Tt)^N$  be given discrete curves associated with the mesh $\Tt^k$ and time step $\tau^k$.
Suppose that their reconstruction $\theta_{\Tt^k,\tau^k}\in L^1(Q_T)$ and $F_{\Sigma^k,\tau^k}\in L^1(Q_T;\mathbb{R}^d)$ from \eqref{eq:densreco} and \eqref{eq:Freco} converge as
\[
\theta_{\Tt^k,\tau^k} \rightharpoonup \theta  \quad \text{weakly in } L^1(Q_T)\,, \quad F_{\Sigma^k,\tau^k} \rightharpoonup F  \quad \text{weakly in } L^1(Q_T;\mathbb{R}^d)\,.
\]
Then,
\begin{equation*}\label{eq:lscbb}
\liminf_{k\rightarrow \infty} \sum_{n=0}^{N-1}\tau \mc{D}_\psi(\btheta^{n+1/2}_k,\bs{F}^{n+1/2}_k) \geq \int_{Q_T}  B(\theta,F)\,.
\end{equation*}
\end{proposition}
%
%
\subsection{Convergence of the scheme}\label{ssec:conv.final}
It remains now to prove that the curve $\rho=\theta$, constructed in the previous sections as the limit of solutions to our discrete scheme, is actually an EDI solution.
\begin{theorem}\label{thm:main}
Let $(\brho^{n}_k)_{n=0}^{N} \in (\mathbb{R}^{\Tt^k}_{+})^N$ and  $(\btheta^{n+1/2}_k)_{n=0}^{N-1} \in (\mathbb{R}^{\Tt^k}_{+})^N$ be the densities obtained as the unique solution of the scheme \eqref{eq:discretePDE}--\eqref{eq:fluxes.linear}--\eqref{eq:rhoG}, associated with a mesh $\Tt^k$ and time step $\tau^k$ satisfying {$\tau^k = o(\dmink)$}.
Let $\rho_{\Tt^k,\tau^k}\in L^1(Q_T)$, $\theta_{\Tt^k,\tau^k}\in L^1(Q_T)$ and $\rho^T_{\Tt^k,\tau^k}\in L^1(\Omega)$ be the reconstructions defined in \eqref{eq:densreco} and \eqref{eq:rhoTreco}.
Then, there exists $\rho\in L^1(Q_T)\cap C([0,T];L^1_{\text{w}}(\Omega))$ such that
\[
\rho_{\Tt^k,\tau^k}, \theta_{\Tt^k,\tau^k} \rightarrow \rho \text{ in } L^1(Q_T)
\qquad\text{and}\qquad
\rho^T_{\Tt^k,\tau^k} \rightharpoonup \rho^T=\rho(T)  \text{ weakly in } L^1(\Omega)
\]
as $k\to+\infty$.
Moreover, the limiting density $\rho$ is the unique EDI solution with initial datum $\rho^0$ in the sense of Definition~\ref{def:EDI}.
\end{theorem}
\begin{proof}
First, notice that no subsequence is involved in the above statement, as a by-product of the uniqueness of EDI solutions, cf. Proposition \ref{prop:props_EDI} in the Appendix. Showing compactness for $\left(\rho_{\Tt^k, \tau^k}\right)_{k}$ and 
$\left(\theta_{\Tt^k, \tau^k}\right)_{k}$, and the fact that any cluster point $\rho$ is an EDI solution then automatically gives the convergence of the whole sequence. So in what follows, we will not indicate when convergences hold up to a subsequence.

First of all, the convergence of the reconstructions is a consequence of Proposition~\ref{prop:convL1}.
Summing Proposition \ref{prop:ede} in time we get the discrete EDI estimate
\begin{equation}\label{eq:EDI.1}
\E_\Tt (\rho_k^N)  +\sum\limits_n\tau \left\{ \mc{D}_\psi(\btheta^{n+1/2},\bs{F}^{n+1/2}) + \mathcal{R}_\psi(\btheta^{n+1/2})\right\}
\leq \E _\Tt(\bs\rho_k^0).
\end{equation}
Jensen's inequality and the definition~\eqref{eq:Vpirho0} of $\brho_k^0$ gives 
\[
\mc{H}(\rho^0_{\Tt^k}) = \sum_{K\in\Tt^k} m_K H(\rho_K^0) \leq \mc{H}(\rho^0),
\]
whereas 
\[
\left| \sum_{K\in\Tt^k} m_K \rho_K^0 V_K - \int_\Omega \rho^0 V \right| \leq \sum_{K\in\Tt^k} \int_K \rho^0(x) |V(x) - V(x_K)| \operatorname{d}x
\leq C \operatorname{size}(\Tt^k)
\]
thanks to the regularity of $V$ and~\eqref{eq:volumebound}.
Therefore, 
\[
\E _\Tt(\bs\rho_k^0) \leq \E(\rho^0) + \mc{O}(\operatorname{size}(\Tt^k)).
\]
We deduce from similar arguments that 
\[
\left| \E_\Tt (\rho_k^N) -  \E(\rho_{\Tt^k,\tau^k}^T) \right|
=
\left| \sum_{K\in\Tt^k} m_K \rho_K^N V_K - \int_\Omega \rho^T_{\Tt^k,\tau^k} V\right|
\leq
C \operatorname{size}(\Tt^k).
\]
Taking the above estimates into acount in~\eqref{eq:EDI.1} leads to
\begin{equation}\label{eq:EDI.2}
 \E(\rho_{\Tt^k,\tau^k}^T)  +\sum\limits_n\tau \left\{ \mc{D}_\psi(\btheta^{n+1/2},\bs{F}^{n+1/2}) + \mathcal{R}_\psi(\btheta^{n+1/2})\right\}
\leq \E(\rho^0) + C \operatorname{size}(\Tt^k). 
\end{equation}
The first term on the left immediately passes to the liminf by standard weak-$L^1$ lower semi-continuity of the convex functional $\E$ together with the weak convergence~\eqref{eq:rhoTlimit}.
Moreover, let us recall from \eqref{eq:Flimit} that $ F_{\Sigma^k,\tau^k}\rightharpoonup F$ weakly in $L^1(Q_T)$.
Propositions \ref{prop:fisher} and \ref{prop:bb} thus allow to take the liminf in the dissipation terms and conclude that
 \[
 \E (\rho^T) + 2 \int_{Q_T} \pi\left| \nabla \sqrt{\frac{\rho}{\pi}}\right|^2 +  \int_{Q_T}  B(\rho,F) \leq \E (\rho^0)\,.
 \]
It only remains to show that the pair $(\rho,F)$ solves the continuity equation with initial and terminal data $\rho^0$ and $\rho^T$ taken in the $C([0,T];L^1_w(\Omega))$ sense.
For this, take any test-function $\varphi \in C^2(\overline{Q}_T)$ and denote
\[
\varphi_{K}^n = \frac{1}{m_K} \int_K\varphi(t^n,\cdot)\,.
\]
For fixed $\Tt,\tau$, observe on the one hand that
\[
\begin{aligned}
\langle \rho_{\Tt,\tau}, \partial_t \varphi\rangle  &= \sum_n  \sum_{K}  m_K \rho^{n+1}_K (\varphi^{n+1}_K - \varphi^n) \\
& = \sum_n  \sum_{K}  m_K (\rho^n_K -\rho^{n+1}_K) \varphi^{n}_K + \sum_K m_K (\rho^{N}_K\varphi^{N}_K - \rho^0_K\varphi^0_K)\\
& =  \sum_n  \sum_{K}  m_K (\rho^n_K -\rho^{n+1}_K) \varphi^{n}_K + \int_\Omega \varphi(T,\cdot) \rho^T_{\Tt,\tau} - \int_\Omega \varphi(0,\cdot) \rho^0_{\Tt},
\end{aligned}
\]
and on the other hand that 
\[
\begin{aligned}
\langle F_{\Sigma,\tau},\nabla \varphi \rangle & = \frac{1}{2} \sum_{n,K} \sum_{\sigma \in\Sigma_K} m_\sigma d_{\sigma} F^{n+1/2}_{K\sigma} \frac{1}{\Delta_\sigma} \int_{t^n}^{t^{n+1}}\int_{\Delta_{\sigma}} \nabla \varphi \cdot n_{K\sigma}\\
&
=
\frac{1}{2} \sum_{n,K} \sum_{\sigma \in\Sigma_K} \tau m_\sigma d_{\sigma} F^{n+1/2}_{K\sigma} \frac{\varphi^{n}_L - \varphi^{n}_K}{d_\sigma} +
\mc O\left(C_\varphi(\size(\Tt)+\tau) \|F_{\Sigma,\tau}\|_{L^1(Q_T)}\right)
\\
&
=
- \sum_{n,K} \sum_{\sigma \in\Sigma_K} \tau m_\sigma  F^{n+1/2}_{K\sigma} \varphi^{n}_K
+ \mc O\left(C_\varphi(\size(\Tt)+\tau) \|F_{\Sigma,\tau}\|_{L^1(Q_T)}\right),
\end{aligned}
\]
where $C_\varphi$ is a constant depending on $\|D^2_{t,x}\varphi\|_\infty$ only.
Hence by \eqref{eq:discretePDE}
\begin{multline*}
\left| \langle \rho_{\Tt,\tau}, \partial_t \varphi\rangle + \langle F_{\Sigma,\tau},\nabla \varphi \rangle  - \int_\Omega \varphi(T,\cdot) \rho^T_{\Tt,\tau} + \int_\Omega \varphi(0,\cdot) \rho^0_{\Tt} \right|
\\
=
\mc O\left(C_\varphi(\size(\Tt)+\tau) \|F_{\Sigma,\tau}\|_{L^1(Q_T)}\right).
\end{multline*}
Since $F_{\Sigma^k,\tau^k}\rightharpoonup F$ weakly in $L^1(Q_T;\R^d)$, since $\rho_{\Tt^k,\tau^k}\to \rho$ in $L^1(Q_T)$,  since ${\rho^0_{\Tt^k}\rightarrow\rho^0}$ in $L^1(\Omega)$ and since $\rho^T_{\Tt^k,\tau^k}\rightharpoonup\rho^T$ weakly in $L^1(\Omega)$ we can take the limit to retrieve the weak formulation \eqref{eq:FP.cons.weak} of the continuity equation for $C^2(\bar Q_T)$ test functions, and therefore for all $\varphi\in C^1(\bar Q_T)$ by density.
Finally, it is well-known \cite{AGS08} that any pair $(\rho,F)$ solving the continuity equation with finite kinetic energy $\int_{Q_T}B(\rho,F)<+\infty$ is $L^1_w(\Omega)$-continuous in time with initial/terminal data $\rho^0,\rho^T$, and the proof is complete.
\end{proof}

\begin{remark}\label{rem:weakenough}
Note that, since Proposition \ref{prop:fisher} and \ref{prop:bb} were established using weak convergence only, the strong convergence of the reconstructions in Proposition~\ref{prop:convL1} is itself not strictly needed to prove that the limit is an EDI solution.
\end{remark}
%
\section{Numerical implementation and results}\label{sec:num}
In this section we describe the implementation of the scheme defined by \eqref{eq:discretePDE}--\eqref{eq:rhoG}.
We also present some numerical tests confirming the second order accuracy, both in time and space.
%
\subsection{Nested Newton method}\label{ssec:Newton}
Given $\brho^n$, computing $\brho^{n+1}$ requires solving the nonlinear system \eqref{eq:discretePDE}--\eqref{eq:rhoG} with in particular $\rho^{n+1}_K = \Xi(\rho^n_K,\theta^{n+1/2}_K)$.
For practical numerical purposes, solving for $\btheta^{n+1/2}$ as the primary variable requires solving a nonlinear scalar system in each cell in order to evaluate the function $\Xi(\rho^n_K,\cdot)$ itself and its derivatives.
Since $\Xi$ is $C^1$, convex, and $|\Xi'(a,\cdot)|\leq {e}$ this can be achieved efficiently with a Newton-Rhapson method.
In order to limit the number of linear systems to be solved in practice, however, we follow here an alternative reparametrization strategy inspired from~\cite{BC17}.
Recall from Section~\ref{ssec:scheme} that our interpolation $\Xi(a,\cdot)$ is defined in terms of the convex function $g=f^{-1}$ in \eqref{eq:GH}, whose graph $\{y=g(x)\}\subset \R\times\R_+$ we choose to reparametrize as
\[
\R\ni s  \longmapsto (x(s), y(s))
\coloneqq
\begin{cases}
(s +\xi,g(s+\xi) )  &\text{if } s\leq 0,\\
(f({\lambda}s+g(\xi)),{\lambda}s+g(\xi)) & \text{else}.
\end{cases}
\]
Here $\xi>e^{-1}$ is an arbitrary cutoff threshold:
for $s\leq 0$ one runs the graph $y=g(x)$ at unit speed, while for $s>\xi$ one rather chooses to run the inverse graph $x=f(y)$ at speed $\lambda$.
We impose $\lambda = g'(\xi)$, so that $x(\cdot)$ and $y(\cdot)$ are $C^1$ across $s=0$.
Note that, by construction, $\rho^{n+1}_K = \Xi(\rho^n_K, \theta^{n+1/2}_K)$ if and only if there exists (a unique) $s_K \in \mathbb{R}$ such that
\begin{equation}\label{eq:rhoX}
\rho^{n+1}_K = X(\rho^n_K;s_K) \coloneqq 
\left\{
\begin{array}{ll}
s_K & \text{if } \rho^n_K = 0\\
\rho^n_K x(s_K) &\text{otherwise}
\end{array}
\right.
\end{equation}
and
\[
\theta^{n+1/2}_K = Y(\rho^n_K;s_K) \coloneqq
\left\{
\begin{array}{ll}
e^{-1} s_K & \text{if } \rho^n_K = 0\\
\rho^n_K y(s_K) &\text{otherwise}
\end{array}
\right..
\]
At each time step, we first look for $\bs{s}\in\mathbb{R}^{\Tt}$ solving
\begin{equation}\label{eq:discretePDEs}
m_K \frac{ X(\rho^n_K; s_K) - \rho^n_K}{\tau} + \sum_{\sigma \in\Sigma_K} \frac{m_\sigma}{d_\sigma} \pi_\sigma  \left(\frac{ Y(\rho^n_K;s_K)}{\pi_K} - \frac{Y(\rho^n_L;s_L)}{\pi_L}\right) =0\,, \quad \forall\,K \in\Tt\,,
\end{equation}
by the Newton-Raphson method, and then update $\brho^{n+1}$ according to \eqref{eq:rhoX}.
Note that this reparametrization does not change the exact solution, which by Proposition \ref{prop:existence} should satisfy $\rho^{n+1}_K\geq 0$ and $\theta^{n+1/2}_K>0$ for all $K \in \Tt$.

\begin{remark} 
The evaluation of $X(\rho^n_K;\cdot)$, $Y(\rho^n_K;\cdot)$ and their derivatives requires the solution of an inner nonlinear system only when $\rho^n_K>0$ and $s_K\leq 0$, which at convergence corresponds to $\rho^{n+1}_K \leq \xi \rho^n_K$.
In practice we verified numerically that, even choosing $\xi \approx e^{-1}$, the outer Newton method for \eqref{eq:discretePDEs} only require very few iterations.
In that case, the inner Newton method is merely required in order to guarantee robustness of the scheme when dealing with solutions with steep gradients or vanishing densities.
\end{remark}
%
\subsection{Test-cases and numerical results}\label{ssec:results}
We set $\Omega =[0,1]^2$ and consider the two refinement patterns illustrated in Figure \ref{fig:meshes}.
Note that only the subdivision refinement satisfies all the requirements on the mesh geometry from Section~\ref{ssec:mesh2}.
\begin{figure}
\begin{subfigure}{.33\textwidth}
  \centering
  \includegraphics[width=.85\linewidth, trim = 90 0 90 0, clip]{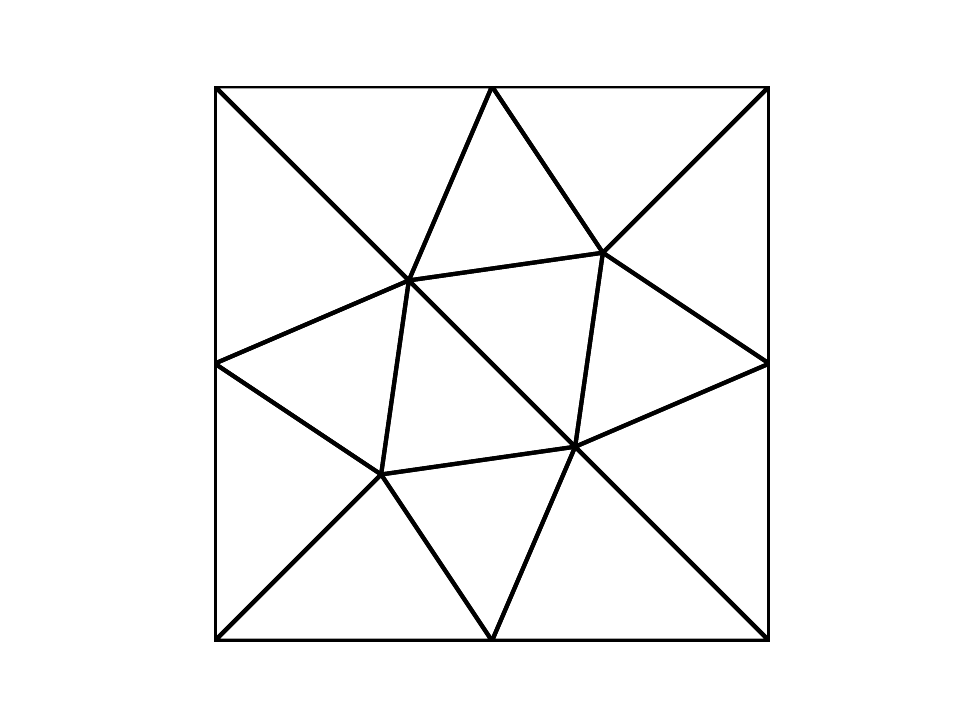}
  \caption{Coarsest mesh}
  \label{fig:sfig1}
\end{subfigure}%
\begin{subfigure}{.33\textwidth}
  \centering
  \includegraphics[width=.85\linewidth, trim = 90 0 90 0, clip]{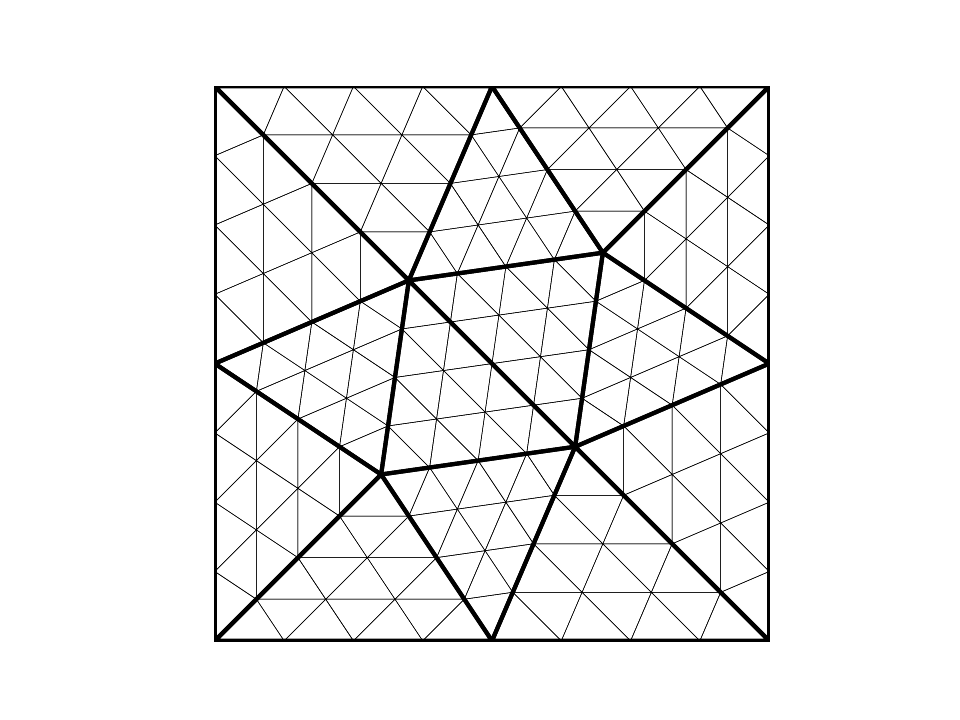}
  \caption{Refinement by subdivision}
  \label{fig:sfig2}
\end{subfigure}%
\begin{subfigure}{.33\textwidth}
  \centering
  \includegraphics[width=.85\linewidth, trim = 90 0 90 0, clip]{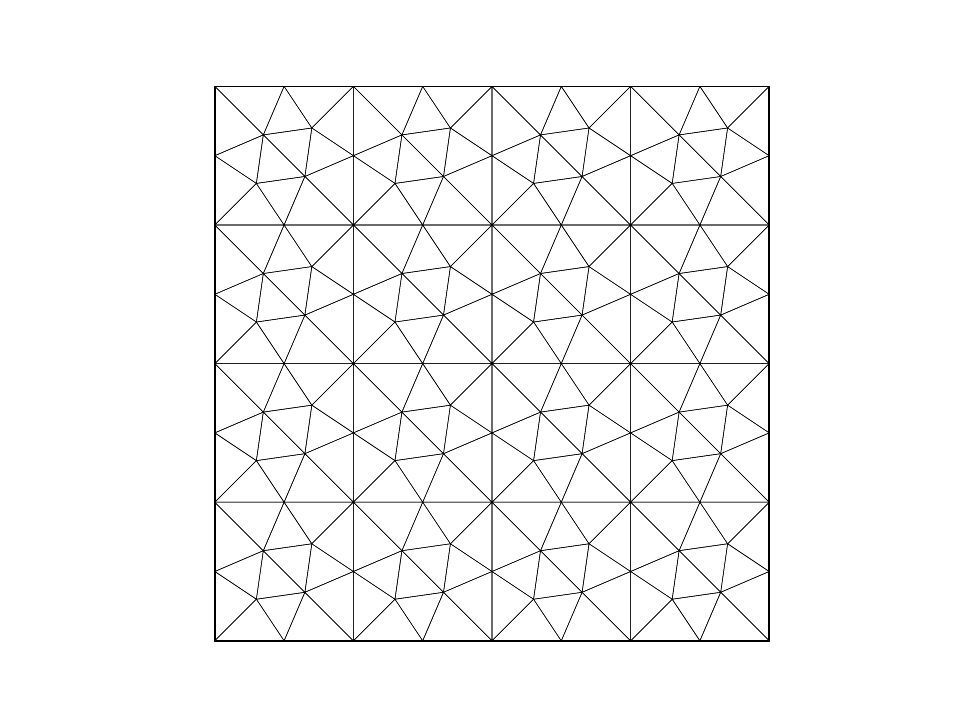}
  \caption{Refinement by repetition}
  \label{fig:sfig3}
\end{subfigure}
\caption{Illustration of the refinement patterns used in the numerical tests; \ref{fig:sfig1} is the base mesh used for both refinement patterns, while \ref{fig:sfig2} and \ref{fig:sfig3} are the second meshes for the two different refinement patterns used in the tests.}
\label{fig:meshes}
\end{figure}
\subsubsection{Convergence test}
In order to investigate convergence and accuracy of our scheme we consider a test-case with gravitational potential $V(x) = -gx_1$ for $g\in\mathbb{R}_+$ and $x=(x_1,x_2)\in \Omega$, in which case the continuous model reduces to the following linear Fokker-Planck equation:
\[
\partial_t \rho - \Delta \rho - g \partial_{x_1} \rho = 0 \quad \text{ in } Q_T\,.
\]
An exact solution given by
\[
\rho(t,x) = \exp\left(- \alpha (t+\delta) + \frac{g}{2}x_1\right)
\left( \pi \cos(\pi x_1) + \frac{g}{2}\sin(\pi x_1)\right)
+ \pi \exp\Big( g \Big(x_1-\frac{1}{2} \Big)\Big)
\]
with $\alpha = \pi^2 + g^2/4$ and $\delta>0$.
Note that the initial condition satisfies
\[
\rho^0(1,x_2)= \pi \exp\Big(  \frac{g}{2} \Big) \left(1- \exp\left(- \alpha \delta \right) \right)
\]
for all $x_2\in[0,1]$, and in particular $\rho^0(1,x_2)=0$ if $\delta=0$, and $\rho(t,x_1,x_2)>0$ for all $t>0$ and $(x_1,x_2)\in\Omega$.
In the following we fix $g=1$.
Tables~\ref{table:subdiv}--\ref{tab:delta0rep} show the $L^1(Q_T)$, $L^2(Q_T)$, and $L^\infty(Q_T)$ errors computed for the reconstructed density $\rho_{\tau,\Tt}$ for the two values $\delta=0.0001,\delta=0$ and both refinement methods.
We observe that the aimed second order in time and space convergence is practically achieved for the $L^1(Q_T)$ norm whatever the initialization and the mesh refinement strategy.
In the presence of vacuum, when the initial profile $\rho^0$ partially vanishes along $x_1=1$ for $\delta=0$, second order accuracy is lost for the $L^\infty(Q_T)$ norm, and to a lesser extent for the $L^2(Q_T)$ norm.

\begin{table}
\begin{tabular}{lllllllll}
\hline
 $\tau$    & $\size(\Tt)$      & $L^1$    & rate     & $L^2$    & rate     & $L^\infty$   & rate     & $\rho_{min}$   \\
\hline
 2.50e-02 & 3.06e-01 & 4.48e-03 &  & 1.04e-02 &  & 5.96e-02     &  & 6.68e-02      \\
 1.25e-02 & 1.53e-01 & 9.84e-04 & 2.19 & 2.44e-03 & 2.09 & 3.65e-02     & 7.08e-01 & 5.55e-02      \\
 6.25e-03 & 7.65e-02 & 2.39e-04 & 2.04 & 6.01e-04 & 2.03 & 1.46e-02     & 1.32 & 5.28e-02      \\
 3.13e-03 & 3.82e-02 & 5.95e-05 & 2.01 & 1.48e-04 & 2.02 & 4.94e-03     & 1.56 & 5.22e-02      \\
 1.56e-03 & 1.91e-02 & 1.48e-05 & 2.01 & 3.65e-05 & 2.02 & 1.42e-03     & 1.79 & 5.21e-02      \\
 7.81e-04 & 9.56e-03 & 3.67e-06 & 2.01 & 9.05e-06 & 2.01 & 3.38e-04     & 2.07 & 5.21e-02      \\
\hline
\end{tabular}
\caption{$\delta= 0.001$, refinement by subdivision.}
\label{table:subdiv}
\end{table}

\begin{table}
\begin{tabular}{lllllllll}
\hline
 $\tau$    & $\size(\Tt)$      & $L^1$    & rate     & $L^2$    & rate     & $L^\infty$   & rate     & $\rho_{min}$   \\
\hline
 2.50e-02 & 3.06e-01 & 2.75e-03 &  & 7.64e-03 &  & 7.52e-02     &  & 6.68e-02      \\
 1.25e-02 & 1.53e-01 & 6.10e-04 & 2.17 & 1.87e-03 & 2.03 & 2.81e-02     & 1.42 & 5.55e-02      \\
 6.25e-03 & 7.65e-02 & 1.50e-04 & 2.03 & 4.85e-04 & 1.95 & 1.43e-02     & 9.76e-01 & 5.28e-02      \\
 3.13e-03 & 3.82e-02 & 3.75e-05 & 2.00 & 1.24e-04 & 1.97 & 5.29e-03     & 1.43 & 5.22e-02      \\
 1.56e-03 & 1.91e-02 & 9.36e-06 & 2.00 & 3.10e-05 & 2.00 & 1.64e-03     & 1.69 & 5.21e-02      \\
 7.81e-04 & 9.56e-03 & 2.34e-06 & 2.00 & 7.73e-06 & 2.00 & 4.31e-04     & 1.93 & 5.21e-02      \\
\hline
\end{tabular}
\caption{$\delta= 0.001$, refinement by repetition.}
\label{table:repete}
\end{table}

\begin{table}
\begin{tabular}{lllllllll}
\hline
 $\tau$    & $\size(\Tt)$      & $L^1$    & rate     & $L^2$    & rate     & $L^\infty$   & rate     & $\rho_{min}$   \\
\hline
 2.50e-02 & 3.06e-01 & 4.60e-03 &  & 1.09e-02 &  & 8.49e-02     &  & 1.61e-02      \\
 1.25e-02 & 1.53e-01 & 1.05e-03 & 2.13 & 2.85e-03 & 1.93 & 6.43e-02     & 4.00e-01 & 4.06e-03      \\
 6.25e-03 & 7.65e-02 & 2.64e-04 & 1.99 & 8.26e-04 & 1.79 & 3.73e-02     & 7.85e-01 & 1.02e-03      \\
 3.13e-03 & 3.82e-02 & 6.68e-05 & 1.98 & 2.45e-04 & 1.75 & 2.00e-02     & 8.99e-01 & 2.55e-04      \\
 1.56e-03 & 1.91e-02 & 1.66e-05 & 2.01 & 7.31e-05 & 1.75 & 1.04e-02     & 9.48e-01 & 6.39e-05      \\
 7.81e-04 & 9.56e-03 & 4.13e-06 & 2.01 & 2.18e-05 & 1.74 & 5.29e-03     & 9.72e-01 & 1.60e-05      \\
\hline
\end{tabular}
\caption{$\delta= 0$, refinement by subdivision.} \label{tab:delta0sub} 
\end{table}

\begin{table}
\begin{tabular}{lllllllll}
\hline
 $\tau$    & $\size(\Tt)$      & $L^1$    & rate     & $L^2$    & rate     & $L^\infty$   & rate     & $\rho_{min}$   \\
\hline
 2.50e-02 & 3.06e-01 & 2.84e-03 &  & 8.16e-03 &  & 8.17e-02     &  & 1.61e-02      \\
 1.25e-02 & 1.53e-01 & 6.74e-04 & 2.07 & 2.33e-03 & 1.81 & 5.55e-02     & 5.57e-01 & 4.06e-03      \\
 6.25e-03 & 7.65e-02 & 1.74e-04 & 1.96 & 7.23e-04 & 1.69 & 3.43e-02     & 6.94e-01 & 1.02e-03      \\
 3.13e-03 & 3.82e-02 & 4.43e-05 & 1.97 & 2.25e-04 & 1.69 & 1.91e-02     & 8.43e-01 & 2.55e-04      \\
 1.56e-03 & 1.91e-02 & 1.12e-05 & 1.99 & 6.92e-05 & 1.70 & 1.01e-02     & 9.15e-01 & 6.39e-05      \\
 7.81e-04 & 9.56e-03 & 2.82e-06 & 1.99 & 2.11e-05 & 1.71 & 5.23e-03     & 9.55e-01 & 1.60e-05      \\
\hline
\end{tabular}
\caption{$\delta= 0$, refinement by repetition.} \label{tab:delta0rep} 
\end{table}

In figure \ref{fig:balance} we plot the error in the dissipation balance
\[
\Delta(t^n)
\coloneqq
\frac{1}{\tau}
\big[\E _{\Tt}(\brho^{n}) - \E _{\Tt}(\brho^{n+1})) - \mc{D}_\psi(\btheta^{n+1/2}) - \mc{D}_{\psi^*}(\btheta^{n+1/2})\Big]
\]
as a function of time.
Recall that $\Delta(t)\equiv 0$ is expected in the limit for continuous EDI solutions of the Fokker-Planck equation, so the smaller the numerical $\Delta$ the more accurately dissipation is captured by the scheme.

\begin{figure}
\includegraphics[scale=.8, trim = 0 10 0 0, clip]{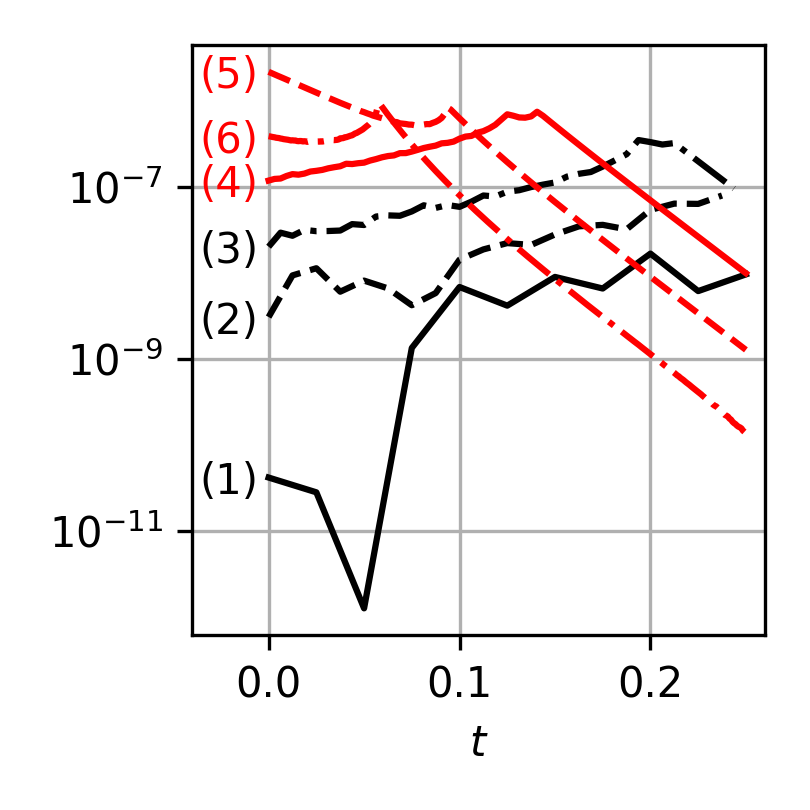}
\includegraphics[scale=.8, trim = 0 10 0 0, clip]{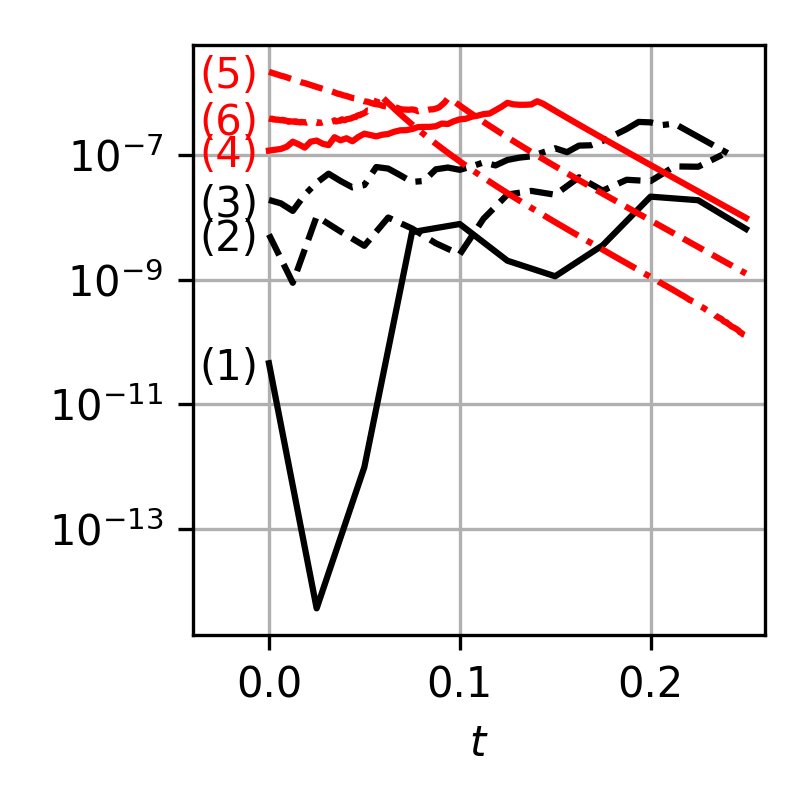}
\caption{Error in the energy balance $\Delta$, for $\delta=0$. The labels refer to different meshes and time steps corresponding to the row numbers of Table \ref{tab:delta0sub} (left, refinement by subdivision) and \ref{tab:delta0rep} (right, refinement by repetition)} \label{fig:balance}
\end{figure}

\begin{figure}
\includegraphics[scale=.8, trim = 0 10 0 0, clip]{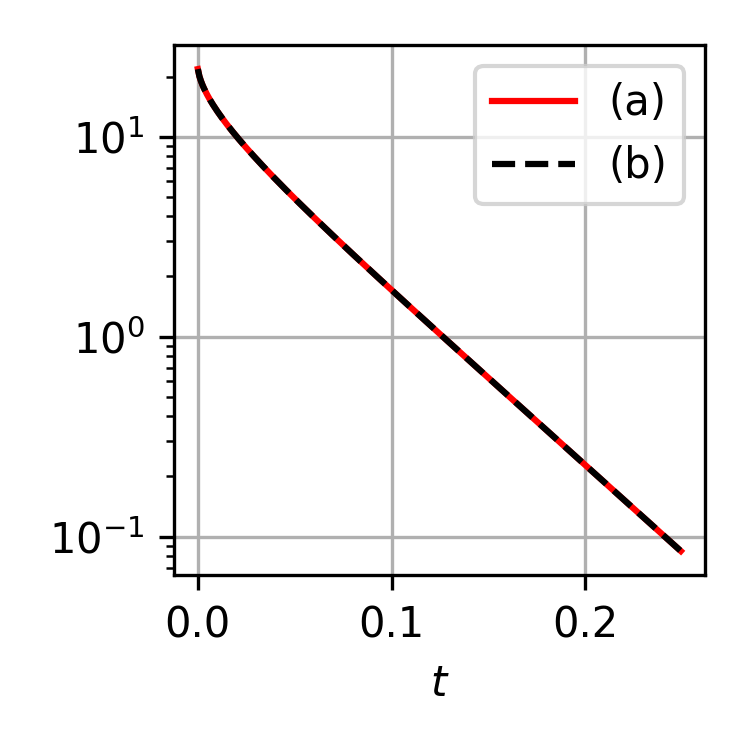}
\includegraphics[scale=.8, trim = 0 10 0 0, clip]{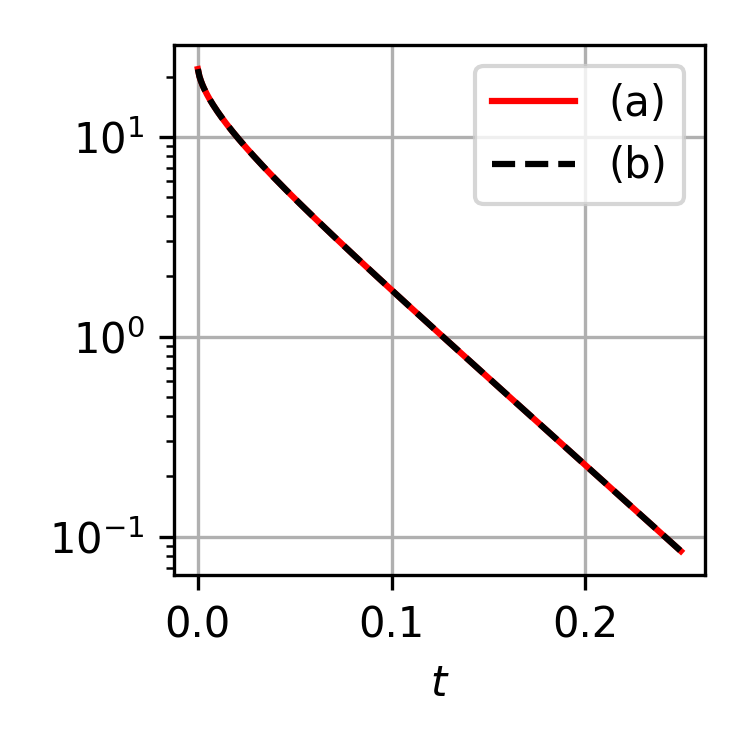}
\caption{Dissipation rates $\mc{D}_\psi(\btheta^{n+1/2})$ (a), and $\mc{D}_{\psi^*}(\btheta^{n+1/2})$ (b), for $\delta=0$ and using the finest mesh and time step, corresponding to the last row of Table \ref{tab:delta0sub} (left, refinement by subdivision) and \ref{tab:delta0rep} (right, refinement by repetition)}
\end{figure}

\begin{figure}
\begin{subfigure}{.33\textwidth}
  \centering
  \includegraphics[width=.85\linewidth, trim = 90 20 115 0, clip]{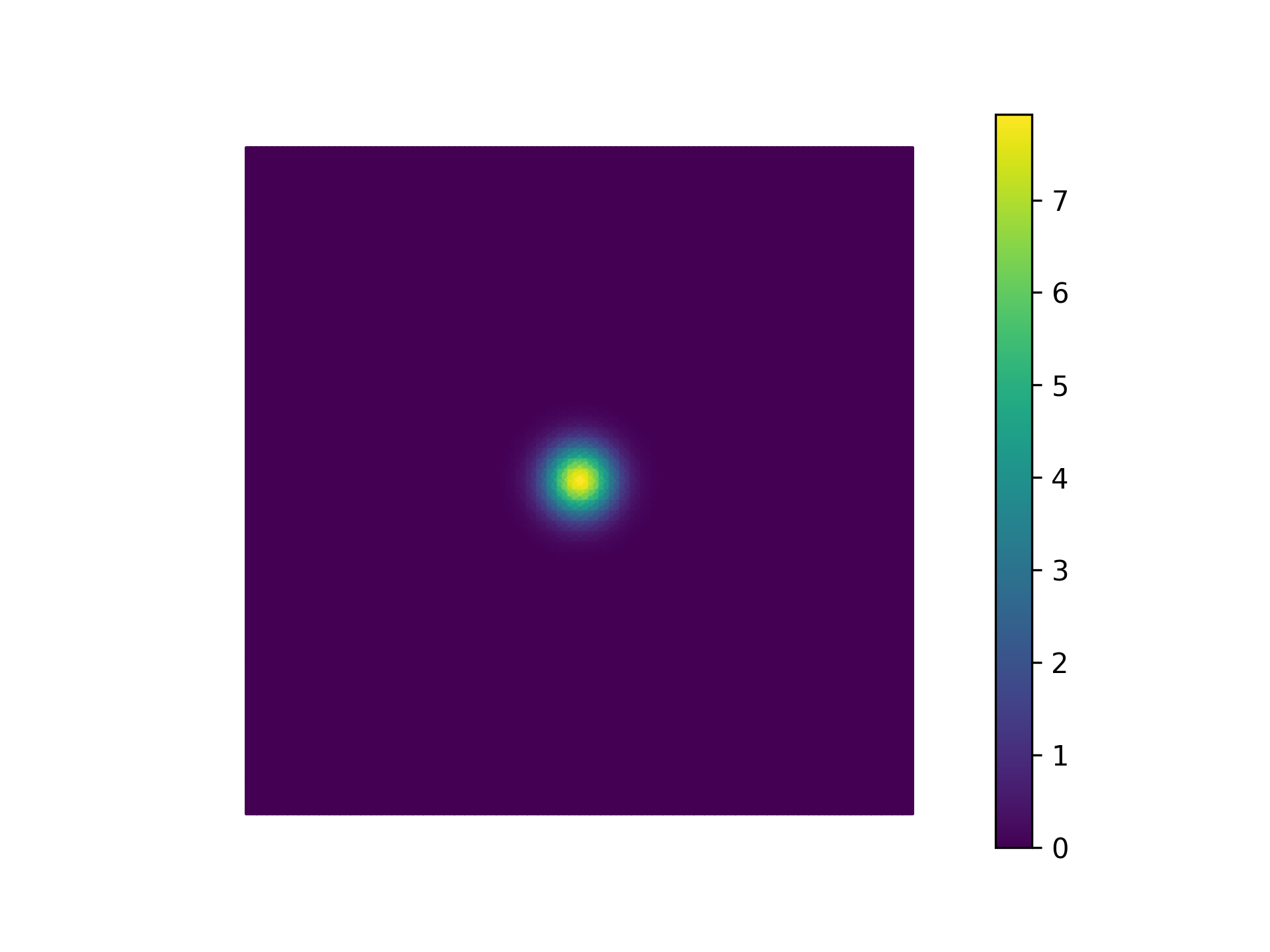}
  \caption{$t=0$}
\end{subfigure}%
\begin{subfigure}{.33\textwidth}
  \centering
  \includegraphics[width=.85\linewidth, trim = 90 20 115 0, clip]{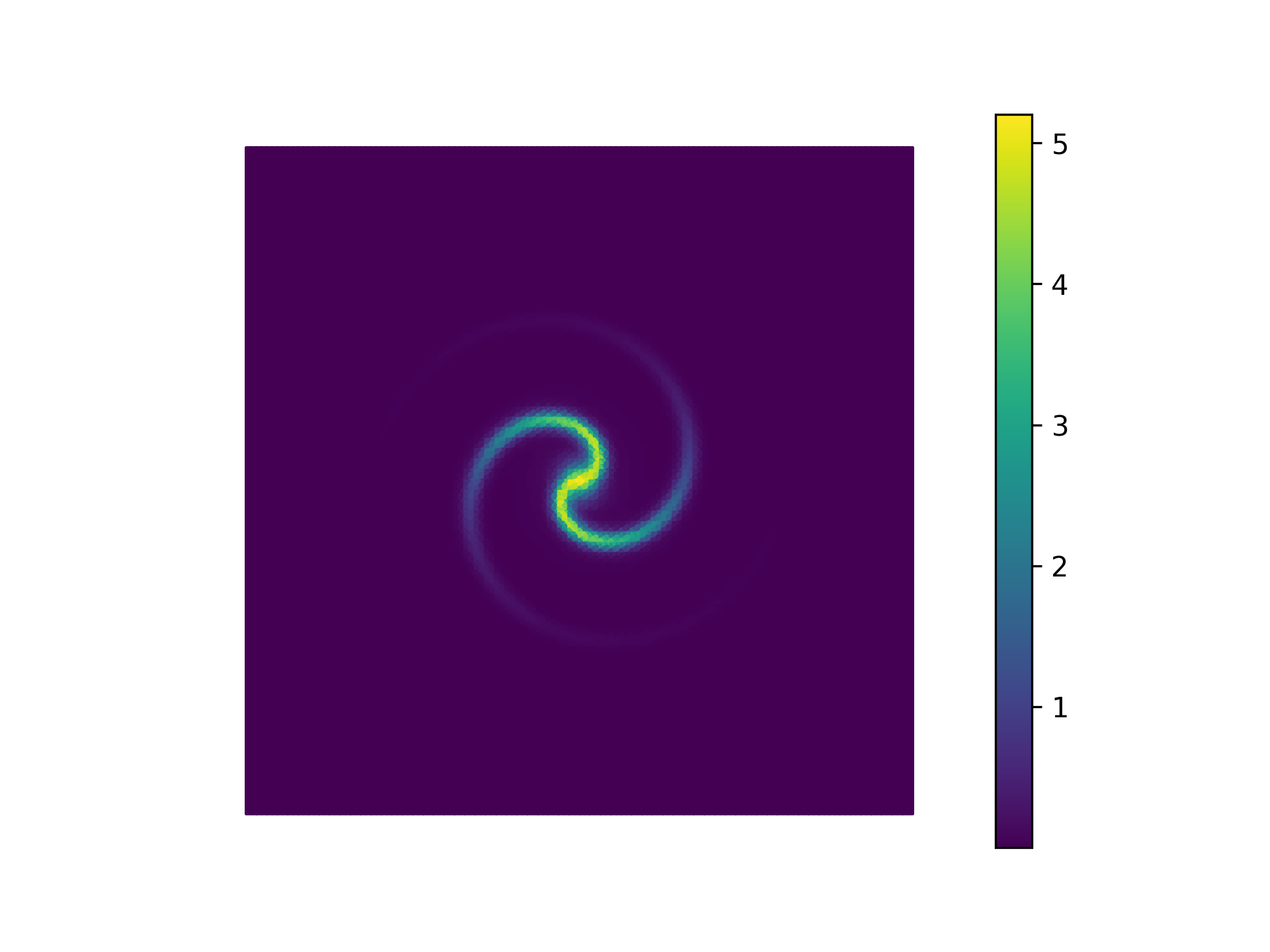}
  \caption{$t=1.2\cdot 10^{-2}$}
\end{subfigure}%
\begin{subfigure}{.33\textwidth}
  \centering
  \includegraphics[width=.85\linewidth, trim = 90 20 115 0, clip]{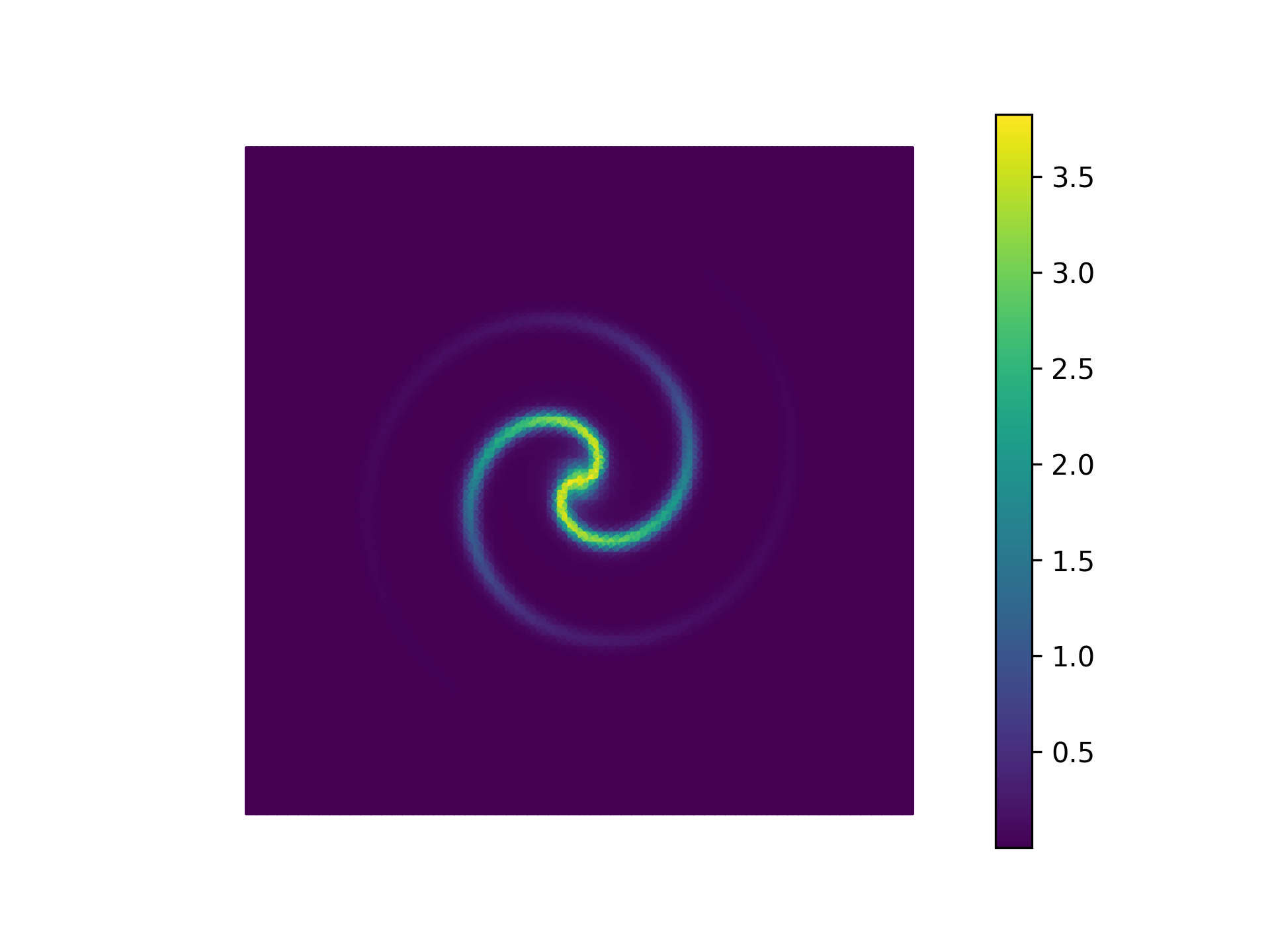}
  \caption{$t=2.5\cdot 10^{-2}$}
\end{subfigure}\\
\begin{subfigure}{.33\textwidth}
  \centering
  \includegraphics[width=.85\linewidth, trim = 90 20 115 0, clip]{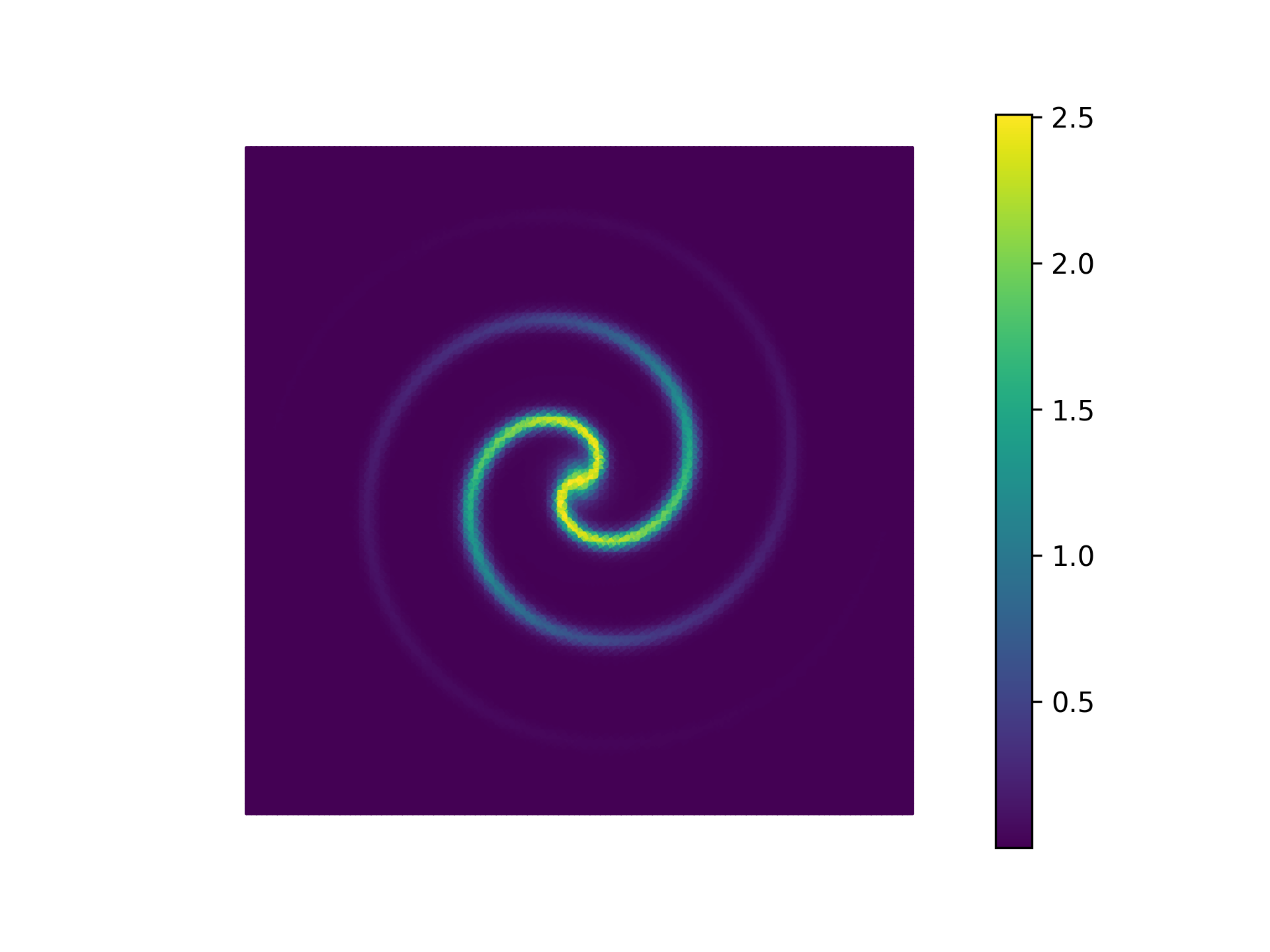}
  \caption{$t=5.0\cdot 10^{-2}$}
\end{subfigure}%
\begin{subfigure}{.33\textwidth}
  \centering
  \includegraphics[width=.85\linewidth, trim = 90 20 115 0, clip]{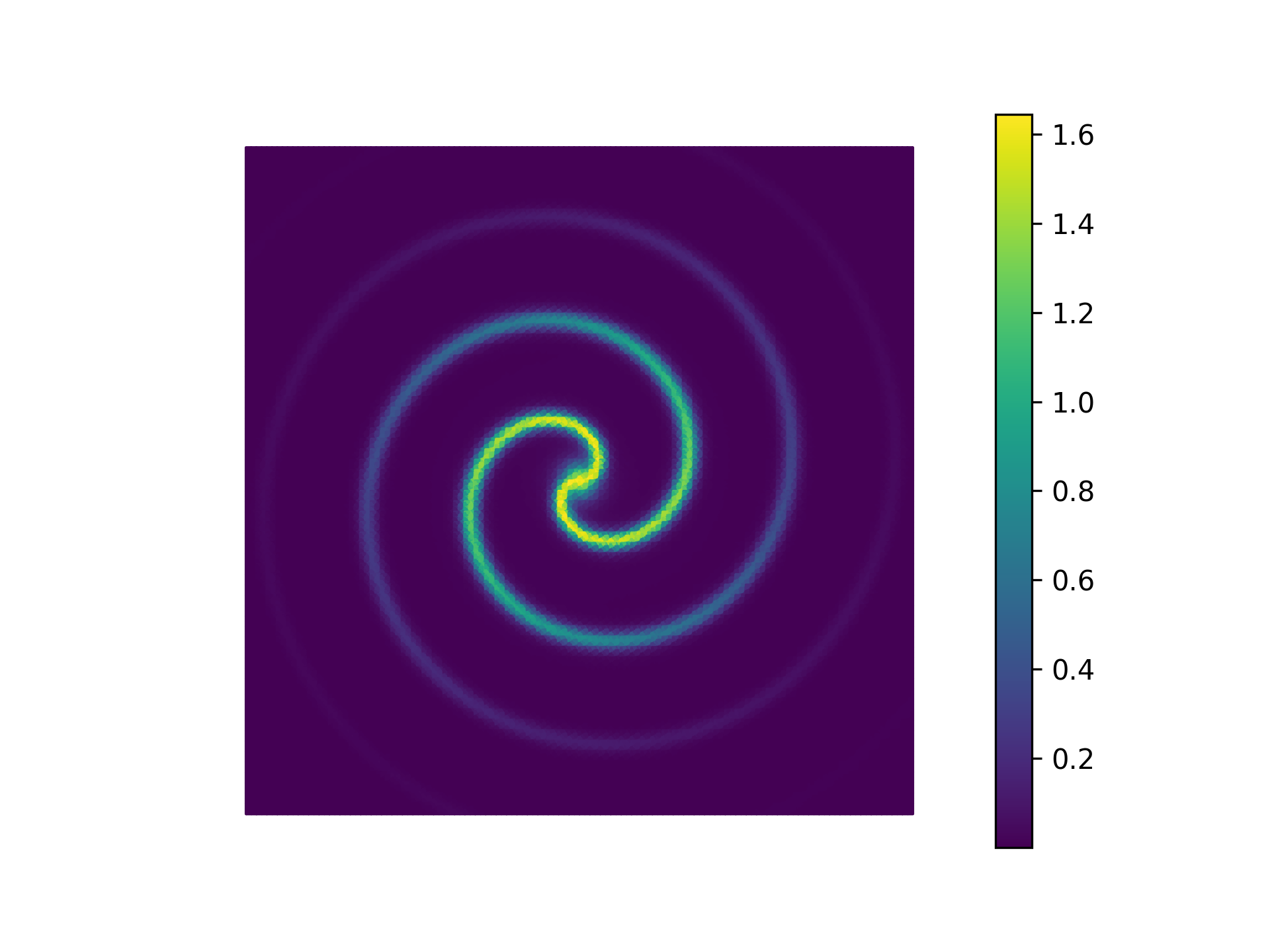}
  \caption{$t= 10^{-1}$}
\end{subfigure}%
\begin{subfigure}{.33\textwidth}
  \centering
  \includegraphics[width=.85\linewidth, trim = 90 20 115 0, clip]{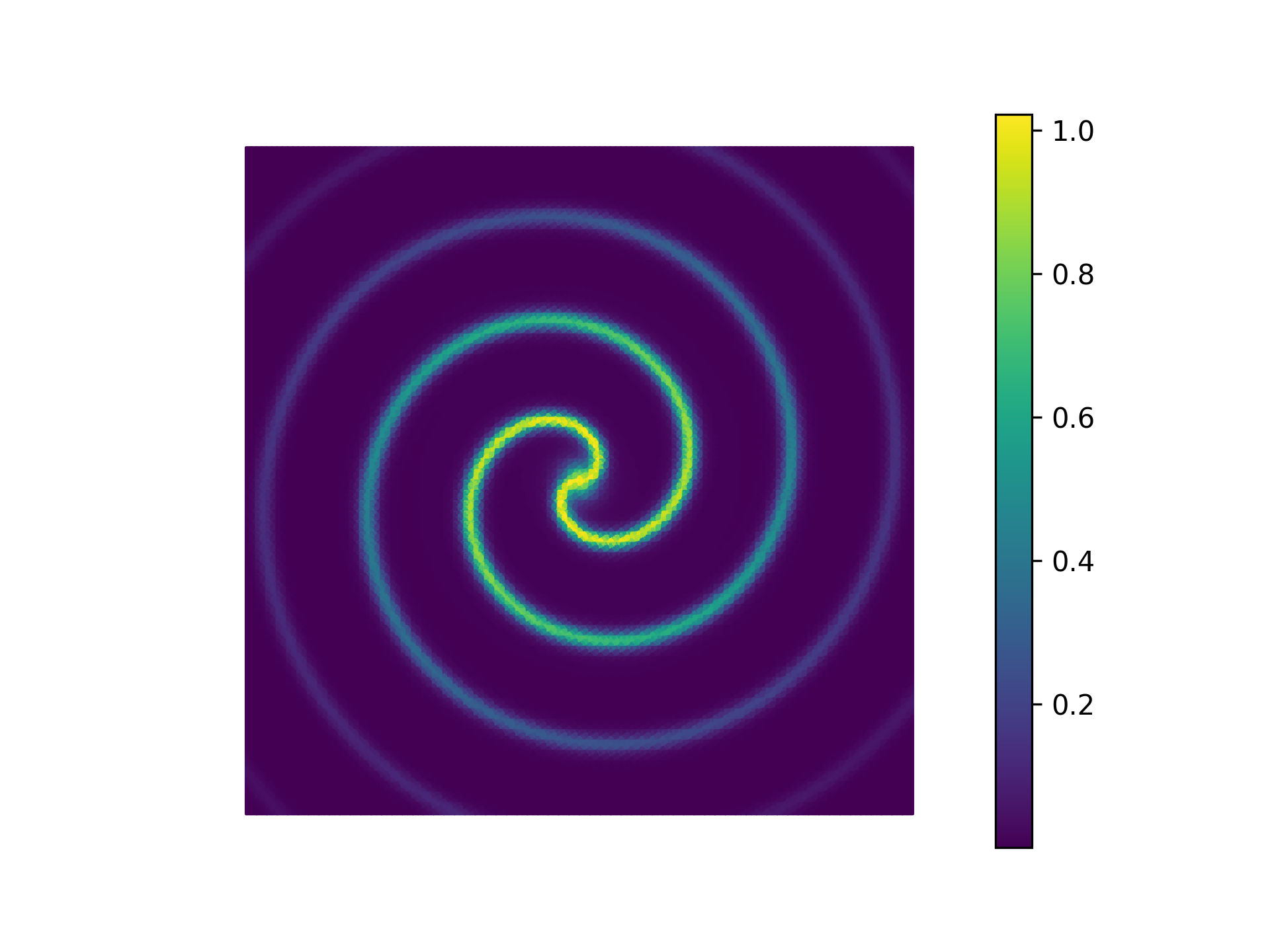}
  \caption{$t=2.0\cdot 10^{-1}$}
\end{subfigure}
\caption{Density evolution for the second test case (note that the color scale is renormalized in each picture).}\label{fig:spiral}
\end{figure}

\begin{figure}
\includegraphics[scale=.8, trim = 0 10 0 0, clip]{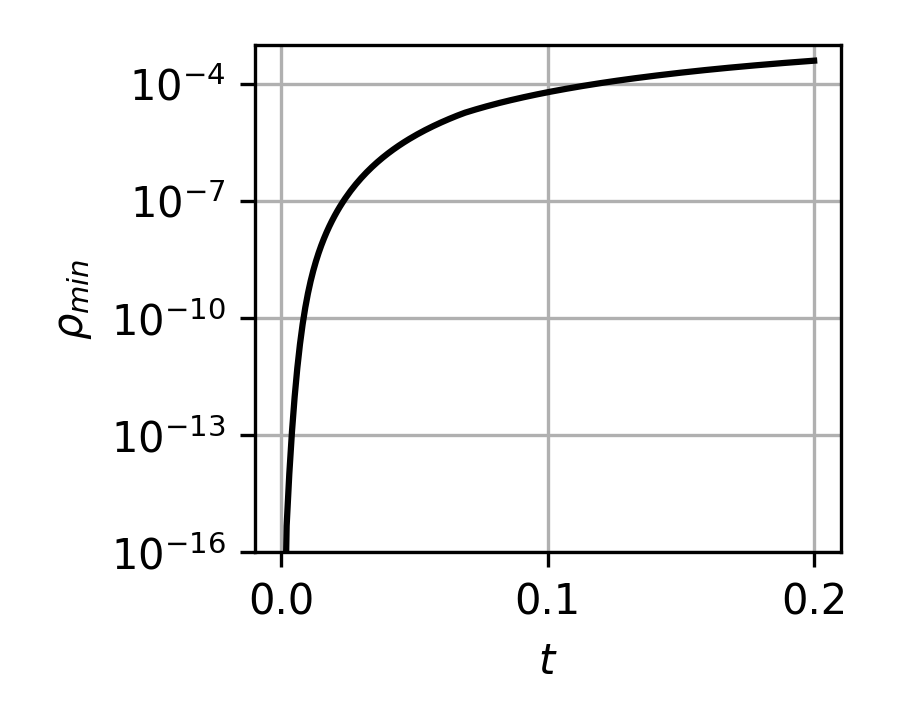}
\includegraphics[scale=.8, trim = 0 10 0 0, clip]{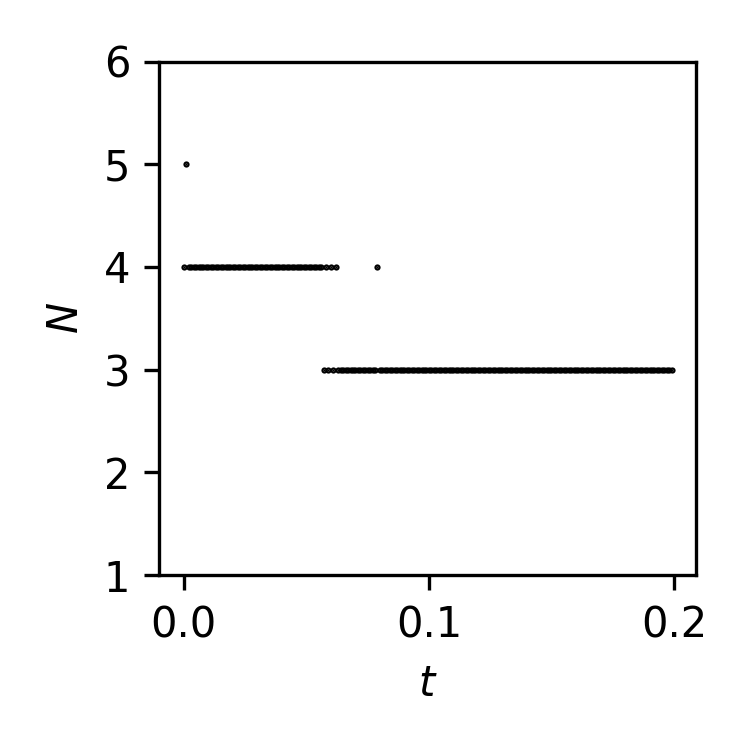}
\caption{Time evolution of density minimum $\min \rho_{\Tt,\tau}(t,\cdot)$ and number of outer Newton iterations to solve the nonlinear system at each time step.} \label{fig:spiral_newton}
\end{figure}
\subsubsection{A test case with a steep potential}
We define now in polar coordinates
\[
V(x) = 5(1-\exp(-r^2/\sigma^2)) (1-\cos^6(20r-\theta))
,\qquad
\rho^0(x) = \frac{1}{\sqrt{2\pi}\sigma} \exp(-r^2/\sigma^2),
\]
where the origin $r=0$ is set at the center of the domain $(1/2,1/2)$ and $\sigma =10^{-2}$.
The solution for these data concentrates on a small area of the domain centered around the curve $20r = \theta$.
The density profile at different times is displayed in figure~\ref{fig:spiral}.
In figure~\ref{fig:spiral_newton} we show the time evolution of the minimum of the density and the number of outer Newton iterations.
 
 \section{Conclusion and prospects}\label{sec:conclusion}
 
The numerical strategy we propose in this paper shows great promises, as its structure allows to carry out the rigorous numerical analysis.
The convergence can even be established under reasonable assumptions on the mesh that might be relaxed even further.
The scheme preserves positivity and is compatible with thermodynamics in the sense that free energy decays in time at a rate which accurately approaches the exact one.
Moreover, our method is computationally efficient thanks to the local-in space character of the time extrapolation.
In particular,  the resolution of the nonlinear systems arising from our scheme does not seem to be computationally demanding in practice, at least in the test cases we have run.

We only validated our approach so far on the simple linear Fokker-Planck equation.
The extension to nonlinear (possibly degenerate) parabolic equations (or even systems) with gradient structure remains to be done.
We also plan to extend our strategy to the case of Poisson-Nernst-Planck systems, where the extrapolation step is no longer local but still computable at reasonable cost.

\appendix

\section{Properties of EDI solutions}
\label{sec:appendix_EDI}
We discuss here basic properties of EDI solutions needed for our purpose.
The starting point is the following chain rule:
\begin{proposition}[Chain rule]
Let $\Omega\subset \R^d$ be convex and $V\in C^2(\bar\Omega)$.
Take any $(\rho,F)$ satisfying the continuity equation \eqref{eq:FP.cons.weak} in time $[0,T]$, with moreover finite kinetic energy and Fisher information
 $$
 \int_0^T \mc R(\rho) + \int_{Q_T} B(\rho, F)<+\infty.
 $$
 Then $t\mapsto \E(\rho^t)$ is absolutely continuous with distributional derivative
 \begin{equation}
 \label{eq:upper_chain_rule_abstract}
 \frac{d}{dt}\E(\rho^t)=\int_\Omega F^t \cdot \nabla\log\left(\frac{\rho^t}{\pi}\right)
 \qquad
 \in L^1(0,T).
 \end{equation}
\end{proposition}
As $\Omega$ is convex, one can directly apply results from \cite{AGS08} (see in particular \S 10.1.2.E, Proposition 10.3.18, and thm. 10.4.9 therein).
Roughly speaking, the convexity of $\Omega$ combined with the regularity of $V$ guarantee that the relative entropy $\rho\mapsto \E(\rho)=\mc H(\rho\,\vert\,\pi)$ is $\lambda$-displacement convex for some $\lambda\in\R$, which then opens the way to the subdifferential calculus developped in \cite{AGS08}.
The extension to non-smooth and non-convex Lipschitz domains of the above chain rule is an open problem up to our knowledge.

The following properties of EDI solutions is then an easy corollary:
\begin{proposition}
\label{prop:props_EDI}
Under the same assumptions, take in addition $\rho^0$ with finite energy $\E(\rho^0)<+\infty$.
Then EDI solutions with initial datum $\rho^0$ are unique, solve the Fokker-Planck equation \eqref{eq:FP} at least in the distributional sense, and satisfy in fact \emph{Energy Dissipation Equality} in the sense that $t\mapsto \E(\rho^t)$ is absolutely continutous with
$$
\frac{d}{dt}\E(\rho^t)=-\mc R(\rho^t)-\int_\Omega B(\rho^t,F^t)
\qquad\in L^1(0,T).
$$
\end{proposition}
\begin{proof}
Let us first show that any EDI solution is a distributional solution, which amounts to proving that the flux driving the continuity equation \eqref{eq:FP.cons.weak} is $F=-\nabla \rho -\rho\nabla V=-\rho \nabla \log\left(\frac\rho\pi\right)$.
To this end we first note from \eqref{eq:upper_chain_rule_abstract} and Young's inequality that
$$
\left|\E(\rho^{t_1})-\E(\rho^{t_0})\right|
\leq
\frac 12\int_{t_0}^{t_1}\int_\Omega \left\{\frac{|F|^2}{\rho}+\rho \left|\nabla\log\left(\frac{\rho}{\pi}\right)\right|^2\right\}
= \int_{t_0}^{t_1}\int_\Omega B(\rho,F) + \int_{t_0}^{t_1}\mc R(\rho)
$$
for arbitrary subinterval $[t_0,t_1]\subset [0,T]$.
As a consequence
$$
J(\rho,F;t_0,t_1)\coloneqq \mc E(\rho^{t_1})+\int_{t_0}^{t_1}\mc R(\rho^t)+\int_{t_0}^{t_1}\int_\Omega B(\rho^t,F^t) -\mc E(\rho^{t_0})\geq 0
$$
is nonnegative.
By additivity, if $(\rho,F)$ is an EDI solution in time $[0,T]$ -- meaning $J(\rho,F;0,T)\leq 0$ -- we have that
$$
0\leq J(\rho,F;t_0,t_1) = J(\rho,F;0,T) - J(\rho,F;0,t_0)-J(\rho,F;t_1,T)\leq 0
$$
and $(\rho,F)$ is therefore an EDI solution in any subinterval.
With the absolute continuity from \eqref{eq:upper_chain_rule_abstract} this gives
\begin{multline*}
 -\frac 12 \int_\Omega \frac{|F^t|^2}{\rho^t} -\frac 12 \int_\Omega\rho^t\left|\nabla\log\left(\frac{\rho^t}{\pi}\right)\right|
\\
\leq \int_\Omega F^t \cdot \nabla\log\left(\frac{\rho^t}{\pi}\right)=\frac{d}{dt}\E(\rho^t)
\overset{\text{EDI}}{\leq} -\int_\Omega B(\rho^t,F^t)-\mc R(\rho^t)
\\
=-\frac 12 \int_\Omega \frac{|F^t|^2}{\rho^t} -\frac 12 \int_\Omega\rho^t\left|\nabla\log\left(\frac{\rho^t}{\pi}\right)\right|.
\end{multline*}
This forces equality in Young's inequality, hence $F^t=-\rho^t\nabla\log\left(\frac{\rho^t}{\pi}\right)$ and $(\rho,F)$ is indeed a distributional solution.
This also gives equality in EDI, and $(\rho,F)$ is in fact an EDE solution as in our statement.

For the uniqueness we implement the so-called \emph{Gigli's trick} \cite{gigli2010heat}.
Let $(\rho_1,F_1),(\rho_2,F_2)$ be two EDI solutions with common initial datum $\rho^0$ and set
$$
\tilde \rho\coloneqq \frac 12\left(\rho_1+\rho_2\right)
\qquad\text{and}\qquad
\tilde F\coloneqq \frac 12\left(F_1+F_2\right).
$$
By linearity $(\tilde \rho,\tilde F)$ still solves the continuity equation with initial datum $\tilde\rho^0=\rho^0$.
Fix any $\tau\in (0,T]$ and recall form the first part of the proof that any EDI solution in  $[0,T]$ is also an EDI solution in $[0,\tau]$.
Summing the two EDI inequalities for $i=1,2$ and leveraging the joint convexity of the Fisher information and Benamou-Brenier functionals, we see that
\begin{multline*}
\frac 12\left(\E(\rho_1^{\tau})+\E(\rho_2^{\tau})\right) + \int_0^{\tau} \mc R(\tilde\rho) + \int_{Q_{\tau}}B(\tilde\rho,\tilde F)
\\
\leq
\frac 12\left(\E(\rho_1^{\tau})+\E(\rho_2^{\tau})\right) + \int_0^{\tau} \frac 12 \left[\mc R(\rho_1)+\mc R(\rho_2)\right] + \int_{Q_{\tau}}\frac 12\left[ B(\rho_1,F_1)+B(\rho_2,F_2)\right]
\\
\leq
\frac 12\left(\E(\rho_1^0)+\E(\rho_2^0)\right)
=
\E(\rho^0)=\E(\tilde\rho^0).
\end{multline*}
This shows in particular that $(\tilde\rho,\tilde F)$ has finite kinetic energy and Fisher information in $[0,\tau]$.
Owing to \eqref{eq:upper_chain_rule_abstract} and Young's inequality we see that $ \E(\tilde\rho^0)\leq \E(\tilde\rho^{\tau}) +\int_0^{\tau} \mc R(\tilde\rho) + \int_{Q_{\tau}} B(\tilde\rho, \tilde F)$, which substituted into the above right-hand side yields
$$
\E\left(\frac{\rho^{\tau}_1+\rho^{\tau}_2}{2}\right)=\E(\tilde \rho^{\tau})\geq \frac 12\left(\E(\rho^{\tau}_1)+\E(\rho^{\tau}_2)\right).
$$
By strict convexity of $\E$ this implies $\rho^{\tau}_1=\rho^{\tau}_2$, and since $\tau\in (0,T)$ was arbitrary the proof is complete.
\end{proof}


\section{An Aubin-Lions-Moussa compensation-compactness result}\label{sec:appendix_Moussa}
The technical statement below is a kind of compensation-compactness argument, allowing to pass to the limit for the product of two weakly converging sequences.
\begin{proposition}
\label{prop:aubin_lions}
Let $\Omega\subset \R^d$ be Lipschitz and bounded, and consider two sequences $\rho_k,f_k$ such that
\begin{enumerate}[(i)]
 \item
 \label{item:rho_f_L1_Linfty}
 $\rho_k$ is bounded in $L^1(Q_T)$ and $f_k$ is bounded in $L^\infty(Q_T)$
 \item
 \label{item:rho_equi}
 $\rho_k$ is $L^1(Q_T)$-equiintegrable
 \item
 \label{item:sp_compact}
 for some fixed modulus of continuity $\eta(\cdot)$ and all $\omega\subset\subset \Omega$, there holds
 \begin{equation}
 \label{eq:f_space_compact}
 \sup\limits_{k}\int_0^T \int_{\omega}|f_k(t,x)-f_k(t,x+h)|\leq \eta(|h|)
 \qquad\text{for }|h|\leq \operatorname{dist}(\omega,\partial\Omega)
 \end{equation}
 \item
 \label{item:time_compact}
 $\partial_t \rho_k$ is bounded in $\mc M(0,T;\left(W^{m,q}(\Omega)\right)')$ for some fixed $m\in \N,q\geq 1$
\end{enumerate}
Then, up to extraction of a subsequence, $\rho_k\rightharpoonup \rho$ weakly in $L^1(Q_T)$, $f_k\overset{\ast}{\rightharpoonup} f$ weakly-$\ast$ in $L^\infty(Q_T)$, and $\rho_kf_k\rightharpoonup \rho f$ in $\mc M(Q_T)$ in the sense that
$$
\int_{Q_T} \rho_k f_k\varphi \to \int_{Q_T} \rho f\varphi,
\qquad
\forall\,\varphi\in C(\bar Q_T).
$$
\end{proposition}
\noindent
We stress that this is merely a variant on \cite[Proposition 1]{moussa2016some}, with however the subtle difference that the space compactness is obtained therein via suitable $W^{1,p}(\Omega)$ bounds, whereas we use here the more versatile difference quotient estimate \eqref{item:sp_compact}.
This has two main advantages: first, this covers the case where space compactness is obtained via discrete difference quotients, as is typical for finite volume schemes (see our practical application for the proof of Proposition~\ref{prop:convL1}).
Second, and more importantly, this allows more general scenarios since in \cite{moussa2016some} some restriction is imposed between the (spatial) Sobolev exponent $p\in [1,d]$ on $f$ and the limitation $\alpha \in [1,p^*)$ on some dual $L^{\alpha'}(\Omega)$ estimates on $\rho_k$.
This will be circumvented in our particular setting via the equiintegrability assumption \eqref{item:rho_equi} on $\rho_k$, but at the expense of our $L^\infty$ bounds on $f_k$.

\begin{proof}
Assume first that $\Omega=\R^d$ and that our space compactness \eqref{eq:f_space_compact} holds with $\omega=\R^d$.
We reproduce below the proof of \cite[Proposition 1]{moussa2016some} almost verbatim, and only point out the main differences.
Pick a mollifying sequence $\zeta_n(x)=n^d\zeta(nx)$ (acting in space only and supported in $B_{1/n}$), and write for any test-function $\varphi$
$$
\begin{aligned}
 \int_0^T\int_{\R^d}( \rho f -\rho_k f_k)\varphi
 & =
\int_0^T\int_{\R^d}[ \rho f -(\rho*\zeta_n) f]\varphi
\\
& \phantom{=}+\int_0^T\int_{\R^d}[(\rho*\zeta_n) f - (\rho_k*\zeta_n) f_k]\varphi
\\
& \phantom{=}+\int_0^T\int_{\R^d}[(\rho_k*\zeta_n) f_k- (\rho_k f_k)*\zeta_n]\varphi
\\
& \phantom{=}+\int_0^T\int_{\R^d}[(\rho_k f_k)*\zeta_n-\rho_kf_k]\varphi
\\
& \eqqcolon  I_1(n)+ I_2(k,n)+I_3(k,n)+ I_4(k,n).
\end{aligned}
$$
Arguing as in \cite{moussa2016some} one shows without too much trouble that $I_1(n)\to 0$ as $n\to\infty$, that $I_2(k,n)\to 0$ for fixed $n$ as $k\to\infty$, and that $I_4(k,n)\to 0$ uniformly in $k$ as $n\to\infty$.
The main difference lies here in the Friedrich commutator $I_3$, which we handle now with care.
To this end we will show that
$$
S_{k,n}(t,x)\coloneqq[(\rho_k*\zeta_n) f_k- (\rho_k f_k)*\zeta_n](t,x)
$$
converges strongly to zero in $L^1((0,T)\times\R^d)$ as $n\to\infty$, uniformly in $k$.
We first observe that, by Fubini's theorem,
$$
\|S_{k,n}\|_{L^1((0,T)\times\R^d)}
\leq
\int_{B_{1/n}}\int_0^T\int_{\R^d}|\rho_k(t,x-y)|\,|f_k(t,x)-f_k(t,x-y)|\zeta_n(y) \ed x\ed t \ed y.
$$
Take $r(n)\to+\infty$ as $n\to\infty$, and let $E_{k,n}^y=\{(t,x)\in(0,T)\times\R^d: |\rho_k(t,x-y)|> r(n)\}$.
Owing to our  $L^1$ bound \eqref{item:rho_f_L1_Linfty} we have that $\meas(E_{k,n}^y)\leq \frac 1{r(n)}\int_{Q_T}\rho_k\leq \frac{C}{r(n)}\to 0$ uniformly in $k$ as $n\to\infty$.
Whence by the equiintegrability assumption \eqref{item:rho_equi}
\begin{multline*}
 \int_{E_{k,n}^y}|\rho_k(t,x-y)|\,|f_k(t,x)-f_k(t,x-y)|\ed x\ed t
\\
\leq 2\|f_k\|_{L^\infty((0,T)\times\R^d)}\int_{E_{k,n}^y}|\rho_k(t,z)|\ed z\ed t
\xrightarrow[n\to\infty]{}0,
\qquad \text{uniformly in }y,k,
\end{multline*}
and integrating in $y\in B_{1/n}$
\begin{equation}
\label{eq:commutator_E}
 \int_{B_{1/n}}\int_{E_{k,n}^y}|\rho_k(t,x-y)|\,|f_k(t,x)-f_k(t,x-y)|\zeta_n(y)\ed x\ed t\ed y \xrightarrow[n\to\infty]{}0,
\qquad \text{uniformly in }k.
\end{equation}
In $(E_{k,n}^y)^\complement$ we have by definition $|\rho_k(t,x-y)|\leq r(n)$ and we use instead the space compactness \eqref{eq:f_space_compact} (recalling also that we can take $\omega=\R^d$ at this stage) to estimate
\begin{multline*}
 \int_{(E_{k,n}^y)^\complement}|\rho_k(t,x-y)|\,|f_k(t,x)-f_k(t,x-y)|\ed x\ed t
\\
\leq
\int_{(E_{k,n}^y)^\complement}r(n)\,|f_k(t,x)-f_k(t,x-y)|\ed x\ed t
\\
\hspace{3cm}
\leq r(n) \sup\limits_k\int_{(0,T)\times\R^d}|f_k(t,x)-f_k(t,x-y)|\ed x\ed t
\\
\leq r(n) \eta(|y|).
\end{multline*}
Choosing the speed $r(n)=o\left(\frac{1}{\eta(1/n)}\right)$ as $n\to\infty$ and integrating in $y$, we end up with
\begin{multline}
\label{eq:commutator_Ec}
 \int_{B_{1/n}}\int_{(E_{k,n}^y)^\complement}|\rho_k(t,x-y)|\,|f_k(t,x)-f_k(t,x-y)|\zeta_n(y)\ed x\ed t\ed y
 \\
 \leq
\int_{B_{1/n}}r(n)\eta(|y|)\zeta_n(y)\ed y
\leq r(n)\eta(1/n)
 \xrightarrow[n\to\infty]{}0,
\qquad \text{uniformly in }k.
\end{multline}
Gathering \eqref{eq:commutator_E}--\eqref{eq:commutator_Ec} gives $\|S_{k,n}\|_{L^1}\to 0$ and therefore
$$
|I_3(k,n)|\leq \|S_{k,n}\|_{L^1}\|\varphi\|_{L^\infty}
\xrightarrow[n\to\infty]{}0,
\qquad \text{uniformly in }k
$$
at least in the whole space $\Omega=\R^d$.
The rest of the proof is then identical to \cite[Proposition 1]{moussa2016some} and we omit the details for the sake of brevity.

Coming back now to the case of bounded domains $\Omega\subset\R^d$, we can apply the same localization argument from \cite[Proposition 3]{moussa2016some}:
Take an exhausting sequence of compact sets $K_l\subset\subset \Omega$ and a sequence of bump functions $0\leq \chi_l(x)\leq 1$ such that $\chi\equiv 1$ on $K_l$, with $\meas(\Omega\setminus K_l)\leq 1/l$ as $l\to +\infty$.
Extending $\rho_k,f_k$ by zero outside of $\Omega$, it is easy to check that the sequences $\{\chi_l\rho_k\}_k,\{\chi_l f_k\}_k$ satisfy the assumptions in the previous step for fixed $l$, hence $\chi_l^2\rho_k f_k\rightharpoonup \chi^2_l \rho f$ in the sense of measures as $k\to\infty$.
Writing $E_l=\Omega\setminus K_l$, we get
$$
\begin{aligned}
\left|\int_{Q_T}(\rho_kf_k-\rho f)\varphi\right|
& =\Big|\int_{Q_T}(\chi_l^2\rho_kf_k-\chi_l^2\rho f)\varphi
+ \int_{Q_T}(1-\chi_l^2)\rho_kf_k\varphi
-\int_{Q_T}(1-\chi_l^2)\rho f\varphi
\Big|
\\
& \leq
\left|\int_{Q_T}(\chi_l^2\rho_kf_k-\chi_l^2\rho f)\varphi\right|
+\int_{Q_T}(1-\chi_l^2)|\rho_kf_k\varphi|
+\int_{Q_T}(1-\chi_l^2)|\rho f\varphi|
\\
&
\leq
\left|\int_{Q_T}(\chi_l^2\rho_kf_k-\chi_l^2\rho f)\varphi\right|
+ \|\varphi\|_{L^{\infty}(Q_T)}\int_0^T\int_{E_l}|\rho_k f_k|+|\rho f|.
\end{aligned}
$$
Pick an arbitrary $\eps>0$.
With our assumptions \eqref{item:rho_f_L1_Linfty}--\eqref{item:rho_equi} it is easy to check that 
\[
\sup\limits_k \int_0^T\int_E |\rho_k f_k|\to 0 \quad\text{as}\;  \meas(E)\to 0,\quad E\subset\Omega.
\]
Owing to $\meas(E_l)\leq 1/l\to 0$ and $\rho f\in L^1(Q_T)$, we can first choose $l=l_0$ large enough so that the second term in the r.h.s. is less than $\eps/2$.
For this fixed $l=l_0$ the first term can also be made smaller than $\eps/2$ if $k\geq k_0$ is large enough, and the claim finally follows.
\end{proof}
\medskip

\noindent
\textbf{Acknowledgements}
This work was supported in part by the Labex CEMPI  (ANR-11-LABX-0007-01) and by the project MATHSOUT of the PEPR Mathematics in Interaction (ANR-23-EXMA-0010) funded by the French National Research Agency. 
LM was supported by \textit{Funda\c{c}\~ao para a Ci\^encia e a Tecnologia} through grants 10.54499/UIDB/00208/2020 and 10.54499/2020.00162.CEECIND/CP1595/CT0008.
\bibliographystyle{plain}      
\bibliography{refs}

\end{document}